\documentclass[reqno, 11pt, a4paper]{amsart}

\usepackage{latexsym,ifthen,xspace}
\usepackage{amsmath,amssymb,amsthm}
\usepackage{bbm}
\usepackage{enumerate, calc}
\usepackage[colorlinks=true, pdfstartview=FitV, linkcolor=blue,
citecolor=blue, urlcolor=blue, pagebackref=false]{hyperref}
\usepackage[text={33pc,605pt},centering]{geometry}    % 11pt
\usepackage{orcidlink}
\usepackage{graphicx}

\newtheorem{theorem}{Theorem}[section]
\newtheorem{lemma}[theorem]{Lemma}
\newtheorem{prop}[theorem]{Proposition}
\newtheorem{assumption}[theorem]{Assumption}
\newtheorem{cor}[theorem]{Corollary}

\theoremstyle{definition}
\newtheorem{definition}[theorem]{Definition}
\newtheorem{example}[theorem]{Example}

\theoremstyle{remark}
\newtheorem{remark}[theorem]{Remark}

\numberwithin{equation}{section}

\DeclareMathAlphabet{\mathsl}{OT1}{cmss}{m}{sl}
\SetMathAlphabet{\mathsl}{bold}{OT1}{cmss}{bx}{sl}

\makeatletter
\newcommand\initCounter[1]{%
  \@ifundefined{c@constcnt-#1}%
  {\newcounter{constcnt-#1}}%
  {}%
}

\newcommand\const[2][c]{%
  \initCounter{#1}%
  \@ifundefined{const-#1-#2}%
  {%
    \stepcounter{constcnt-#1}%
    \expandafter\xdef\csname const-#1-#2\endcsname{\arabic{constcnt-#1}}%
  }%
  {}%
  #1_{\csname const-#1-#2\endcsname}%
}
\makeatother

\newcommand{\al}{\ensuremath{\alpha}}
\newcommand{\be}{\ensuremath{\beta}}
\newcommand{\ga}{\ensuremath{\gamma}}
\newcommand{\de}{\ensuremath{\delta}}

\newcommand{\et}{\ensuremath{\eta}}
\renewcommand{\th}{\ensuremath{\theta}}

\newcommand{\ka}{\ensuremath{\kappa}}
\newcommand{\la}{\ensuremath{\lambda}}
\newcommand{\rh}{\ensuremath{\rho}}
\newcommand{\si}{\ensuremath{\sigma}}
\newcommand{\ta}{\ensuremath{\tau}}

\newcommand{\om}{\ensuremath{\omega}}

\newcommand{\ve}{\ensuremath{\varepsilon}}

\newcommand{\vp}{\ensuremath{\varphi}}

\newcommand{\Ga}{\ensuremath{\Gamma}}
\newcommand{\De}{\ensuremath{\Delta}}

\newcommand{\La}{\ensuremath{\Lambda}}
\newcommand{\Si}{\ensuremath{\Sigma}}

\newcommand{\Om}{\ensuremath{\Omega}}
\newcommand{\cA}{\ensuremath{\mathcal A}}

\newcommand{\cC}{\ensuremath{\mathcal C}}

\newcommand{\cE}{\ensuremath{\mathcal E}}
\newcommand{\cF}{\ensuremath{\mathcal F}}
\newcommand{\cG}{\ensuremath{\mathcal G}}
\newcommand{\cH}{\ensuremath{\mathcal H}}

\newcommand{\cJ}{\ensuremath{\mathcal J}}

\newcommand{\cL}{\ensuremath{\mathcal L}}

\newcommand{\cN}{\ensuremath{\mathcal N}}

\newcommand{\cP}{\ensuremath{\mathcal P}}

\newcommand{\bbE}{\ensuremath{\mathbb E}}

\newcommand{\bbN}{\ensuremath{\mathbb N}}

\newcommand{\bbR}{\ensuremath{\mathbb R}}

\newcommand{\bbZ}{\ensuremath{\mathbb Z}} 
\newcommand{\me}{\ensuremath{\mathrm{e}}}

\newcommand{\md}{\ensuremath{\mathrm{d}}}

\newcommand{\scprfield}[2]{%
  \ensuremath{%
    \big\langle
    #1, #2
    \big\rangle
  }
}

\newcommand{\scprreal}[3]{%
  \ensuremath{%
    \big\langle
    #1, #2
    \big\rangle_{\raisebox{-0ex}{$\scriptstyle L^{\raisebox{.1ex}{$\scriptscriptstyle 2$}} (#3)$}}
  }
}

\newcommand{\scpr}[3]{%
  \ensuremath{%
    \big\langle
    #1, #2
    \big\rangle_{\raisebox{-0ex}{$\scriptstyle \ell^{\raisebox{.1ex}{$\scriptscriptstyle 2$}} (#3)$}}
  }
}

\newcommand{\Norm}[2]{%
  \ensuremath{%
    \mathchoice{\big\lVert #1 \big\rVert}
    {\lVert #1 \rVert}
    {\lVert #1 \rVert}
    {\lVert #1 \rVert}_{\raisebox{-.0ex}{$\scriptstyle #2$}}
  }
}

\newcommand{\wick}[1]{\mathopen{:}#1\mathclose{:}}

\DeclareMathOperator{\mean}{\mathbb{E}}
\DeclareMathOperator{\Mean}{\mathrm{E}}
\DeclareMathOperator{\prob}{\mathbb{P}}
\DeclareMathOperator{\Prob}{\mathrm{P}}

\DeclareMathOperator{\dist}{\mathrm{dist}}
\DeclareMathOperator{\id}{\mathrm{id}}

\DeclareMathOperator{\meanPhi}{\mathbf{E}}
\DeclareMathOperator{\probPhi}{\mathbf{P}}

\DeclareMathOperator{\supp}{\mathrm{supp}}

\newcommand{\ldef}{\ensuremath{\mathrel{\mathop:}=}}
\newcommand{\rdef}{\ensuremath{=\mathrel{\mathop:}}}

\def\indicator{{\mathchoice {1\mskip-4mu\mathrm l}%
    {1\mskip-4mu\mathrm l}{1\mskip-4.5mu\mathrm l}%
    {1\mskip-5mu\mathrm l}}}

\begin{document}

\title[Scaling limits of nonlinear DGFF]{Scaling limits for nonlinear functionals of the discrete Gaussian free field with degenerate random conductances}

% 
% Remove any unused author tags.
% 

% author one information
\author[C. F. Peter]{Christof F. Peter}
\address{University of Cambridge}
\curraddr{Centre for Mathematical Sciences, Wilberforce Road, Cambridge CB3 0WB}
\email{cfp33@cam.ac.uk}
\thanks{}

% author two information
\author[M. Slowik]{Martin Slowik \orcidlink{0000-0001-5373-5754}}
\address{University of Mannheim}
\curraddr{Mathematical Institute, B6, 26, 68159 Mannheim}
\email{slowik@math.uni-mannheim.de}
\thanks{}



\date{\today}

\dedicatory{}

\begin{abstract}
  We consider nonlinear functionals of discrete Gaussian free fields with ergodic random conductances on a class of random subgraphs of $\mathbb{Z}^{2}$, including i.i.d.\ supercritical percolation clusters, where the conductances are possibly unbounded but satisfy an integrability condition. As our main result, we show that, for almost every realisation of the environment, the nonlinear functionals of the rescaled field converge to their continuum counterparts in the Sobolev space $H^{-s}(D)$ for suitable $s > 0$. To obtain the latter, we establish pointwise bounds for the Green's function of the associated random walk among random conductances with Dirichlet boundary conditions, which are valid for all $d \geq 2$. 
\end{abstract}

\maketitle

\tableofcontents

\section{Introduction}
Nonlinear functionals of the Gaussian free field (GFF), also known as the massless bosonic free field, have been objects of interest in probability and mathematical physics for several decades. In constructive quantum field theory (QFT) they provide rigorous models of self-interacting Euclidean field theories and arise as scaling limits of a variety of statistical-mechanical models. A fundamental example is the field with quartic interaction which is closely related to the scaling limits appearing in the study of the Ising model \cite{simon1973ising,aizeman2024ising}. For a long time it was unclear whether non-trivial self-interacting field theories satisfying the physically relevant Wightman axioms \cite{wightman1956quantum} could be constructed as rigorous mathematical objects. While this remains open in higher dimensions, the question was partially resolved in the late 1960s and early 1970s by Glimm, Jaffe, Guerra, Rosen, Simon and others, who constructed field theories with polynomial interactions in $1+1$ space–time dimensions, now known as the $P(\phi)_{2}$ Euclidean QFTs \cite{glimm1968lambda, glimm1970lambda, guerra1975euclidean, glimm2012quantum, simon2015p}. H\o egh-Krohn extended this framework to real functions of exponential type in \cite{hoegh1971general}, and later, jointly with Albeverio, verified the Wightman axioms for such models in \cite{albeverio1974wightman}. Although the original constructions are typically based on ultraviolet cut-offs, such as smeared fields or Fourier partial sums, lattice approximations provide an alternative and physically natural route. This approach, introduced by Guerra, Rosen and Simon in \cite[Chapter~4]{guerra1975euclidean}, approximates the macroscopic field by discrete fields on microscopic lattices $h \mathbb{Z}^2$ and then lets the mesh size $h$ tend to zero. 

In the present article we study nonlinear functionals of the \emph{inhomogeneous} discrete Gaussian free field (DGFF) defined as the centred Gaussian field whose covariances are given by the Green's function of the discrete Laplacian on random weighted subgraphs of $\mathbb{Z}^{2}$ modelling microscopic \emph{impurities} in the underlying medium. Various different aspects of inhomogeneous DGFFs with \emph{uniformly elliptic} edge weights, also called conductances, on $\mathbb{Z}^{d}$, $d \geq 2$, have previously been studied, for instance, in \cite{naddaf1997homogenization, caputo2003finite, miller2011fluctuation, biskup2011scaling, cotar2012existence, cotar2015uniqueness} and more recently in \cite{cipriani2020discrete, chiarini2021disconnection, chiarini2024stochastic, deuschel2024ray, chiarini2025hardwall}. Further results have also been obtained for inhomogeneous DGFFs with \emph{uniformly elliptic} edge weights on random subgraphs of $\mathbb{Z}^{2}$, see e.g.\ \cite{schweiger2024maximum}. We allow for environments of random weights in form of ergodic and unbounded conductances under suitable integrability conditions and geometric assumptions on the resulting random subgraph of $\mathbb{Z}^{2}$ defined in terms of the edge set with positive weights. In particular, our framework includes both supercritical i.i.d.\ and certain correlated percolation clusters as a basic examples. In this setting the scaling limit of the inhomogeneous DGFF with Dirichlet boundary conditions to a continous GFF with a deterministic, non-degenerate matrix $\Si$ has recently been obtained in \cite{andres2025scaling}.

Our results may be understood as addressing how microscopic random impurities affect the macroscopic behaviour of the associated self-interacting field theory in the scaling limit. We focus on exponential-type nonlinear functionals of such inhomogeneous DGFFs. Although the random walk among degenerate random conductances as well as the associated inhomogeneous DGFF homogenises to a Brownian motion and a continuum GFF with deterministic, non-degenerate matrix $\Si$ in the scaling limit, respectively; see, for example, \cite{deuschel2018quenched, andres2025scaling}, the presence of microscopic percolation leaves a non-trivial trace in the nonlinear observables. Indeed, we show that the Wick-renormalised $k$-th power of the inhomogeneous DGFF $\Phi_n$ on a subset of the infinite connected component, $nD \cap \cC_{\infty}(\om)$, subject to Dirichlet boundary conditions, converges in the scaling, for almost every environment $\om$,
\begin{align*}
  \wick{\Phi_n^k} 
  \underset{n \to \infty}{\overset{\mathrm{law}}{\;\longrightarrow\;}}
  \prob[0 \in \cC_{\infty}(\om)]^{1-k/2} \wick{\Psi^k},
\end{align*}
where $\Psi$ is a continuum GFF on a suitable bounded domain $D \subset \mathbb{R}^{2}$ with the homogenised matrix $\Si$. Thus, percolation has a damping effect on the field itself $(k=1)$, is not visible at the level of the Wick square $(k=2)$, and has a reinforcing effect on higher Wick powers $(k>2)$, which are precisely the observables relevant for self-interacting field theories. In particular, as the percolation parameter approaches criticality, the limiting free field is suppressed, the second Wick power remains unchanged, and the higher Wick powers diverge.

As an application of our main theorem, we also determine the scaling limit of the Gaussian multiplicative chaos (GMC) associated with the DGFF, equivalently the discrete approximation of Liouville quantum gravity (LQG) in the $L^2$-regime; see \cite{berestycki2024gaussian,chatterjee2025liouville}. For suitable $\ga$ and suitable sets $A \subset \mathbb{R}^{2}$, we prove the quenched convergence
\begin{align*}
  \int_A \wick{e^{\ga \vp^{nD}_{\lfloor nx \rfloor}}}
  \indicator_{\lfloor nx \rfloor \in \cC_{\infty}(\om)}\, \md x
  \underset{n \to \infty}{\overset{\mathrm{law}}{\;\longrightarrow\;}}
  \prob[0 \in \cC_{\infty}(\om)]
  \scprfield{\wick{e^{\ga \prob[0 \in \cC_{\infty}(\om)]^{-1/2} \Psi}}}{\indicator_A}.
\end{align*}
Finally, we apply our results to the self-interacting field theories discussed above. Namely, we identify the scaling limit of the Gibbs measure on $\bbR^{nD \cap \cC_{\infty}(\om)}$ with Hamiltonian
\begin{align*}
  \cH_{V}^{\om}(\vp)
  \,=\,
  \frac{1}{2} \sum_{\{x,y\} \in E(\cC_{\infty}(\om))}
  \om(\{x,y\}) \bigl(\vp_x - \vp_y \bigr)^2 
  \,+\,
  \frac{1}{n^2} \sum_{x \in \cC_{\infty}(\om)} \wick{V(\ga \vp_x)}
\end{align*}
for potentials $V$ of exponential type. 

\subsection{The model}
For $d \geq 2$, let $(\mathbb{Z}^{d}, E_{d})$ be the $d$-dimensional Euclidean lattice of nearest-neighbor bonds. Vertices $x, y \in \mathbb{Z}^{d}$ are adjacent ($x \sim y$) if $\{x, y\} \in E_{d}$. Define the environment space $(\Omega, \mathcal{F}) \ldef ([0, \infty)^{E_{d}}, \mathcal{B}([0, \infty))^{\otimes E_{d}})$, where $\omega(e)$ is called the conductance of edge $e \in E_{d}$. We call an edge $e$ \emph{open} if $\omega(e) > 0$ and denote the set of open edges by $\mathcal{O}(\omega)$. Moreover, infinite connected components of open clusters are denoted by  $\mathcal{C}_{\infty}(\omega)$. We endow $(\Omega, \mathcal{F})$ with the $d$-parameter group of space shifts $(\tau_{x})_{x \in \mathbb{Z}^{d}}$ acting as
\begin{align}\label{eq:def:space_shift}
  (\tau_{z}\, \omega)(\{x, y\})
  \;\ldef\;
  \omega(\{x+z, y+z\}),
  \qquad \forall\, \{x, y\} \in E_{d}.
\end{align}
Further, we will denote by $\prob$ a probability measure on $(\Omega, \mathcal{F})$, and we write $\mean$ for the expectation with respect to $\prob$. Throughout the paper, we will impose assumptions both on the law $\prob$ and on geometric properties of the infinite cluster.
\begin{assumption} \label{ass:law}
  Assume that $\prob$ satisfies the following conditions.
  \begin{enumerate}[(i)]
  \item
    The law $\prob$ is stationary and ergodic with respect to space shifts of $\mathbb{Z}^{d}$, that is, $\prob \circ \tau_{x}^{-1} = \prob$ for all $x \in \mathbb{Z}^{d}$ and $\prob[A] \in \{0, 1\}$ for any $A \in \mathcal{F}$ such that $\tau_{x}^{-1}(A) = A$ for all $x \in \mathbb{Z}^{d}$.     
    
  \item
    For $\prob$-a.e.\ $\omega$, the set $\mathcal{C}_{\infty}(\omega)$ is non-empty and connected, that is, there exists a unique infinite connected component -- also called infinite open cluster -- and $\th_0 \ldef \prob\bigl[0 \in \mathcal{C}_{\infty}\bigr] > 0$.
  \end{enumerate}
\end{assumption}
Let $\Omega^{*} \ldef \{\omega \in \Omega : \mathcal{C}_{\infty}(\omega) \text{ is non-empty and connected} \}$ and $\Omega^{*}_{0} \ldef \{\omega \in \Omega^{*} : 0 \in \mathcal{C}_{\infty}(\omega)\}$. We write $\prob_{0}[ \,\cdot\, ] \ldef \prob\bigl[ \,\cdot \,|\, 0 \in \mathcal{C}_{\infty}\bigr]$ to denote the conditional distribution given the event $\{0 \in \mathcal{C}_{\infty}\}$, and $\mean_{0}$ for the expectation with respect to $\prob_{0}$. In the case of i.i.d.\ conductances Assumption~\ref{ass:law} is fulfilled if $\prob\bigl[\omega(e) > 0\bigr] > p_{c}$, where $p_{c} \equiv p_{c}(d)$ denotes the critical probability for bond percolation on $\mathbb{Z}^{d}$.

For set $\La \subseteq \bbZ^d$, we denote by $g^{\omega}_{\Lambda}$ the Green's function on $\Lambda \cap \mathcal{C}_{\infty}(\omega)$ associated with the operator $\mathcal{L}^{\omega}$ acting on bounded functions $f\colon \mathbb{Z}^{d} \to \mathbb{R}$ as
\begin{align} \label{eq:defL}
  \big(\mathcal{L}^{\omega} f)(x)
  \;=\; 
  \sum_{y \sim x} \omega(\{x, y\})\, \bigl(f(y) - f(x)\bigr).
\end{align}
That is, $g^{\omega}_{\Lambda}$ is the solution of the Poisson equation 
\begin{align}\label{eq:PDE_green_killed}
  \left\{
    \mspace{-6mu}
    \begin{array}{rcll}
      \big(\mathcal{L}^{\omega} g_{\Lambda}^{\omega}(\cdot, y)\big)(x)
      &\mspace{-5mu}=\mspace{-5mu}& -\delta_{y}(x),
      &x \in \Lambda \cap \mathcal{C}_{\infty}(\omega),
      \\[.5ex] 
      g_{\Lambda}^{\omega}(x, y)
      &\mspace{-5mu}=\mspace{-5mu}& 0, 
      &x \in (\Lambda \cap \mathcal{C}_{\infty}(\omega))^{\mathrm{c}}.
    \end{array}
  \right.
\end{align}
We now introduce the inhomogeneous harmonic crystal:
\begin{definition}[Inhomogeneous DGFF]\label{def:inhomogeneous:DGFF}
  Let $\Lambda \subset \mathbb{Z}^{d}$ be bounded. For any $\omega \in \Omega^{*}$, the \emph{inhomogeneous discrete Gaussian free field} on $\Lambda \cap \mathcal{C}_{\infty}(\omega)$ with zero boundary conditions is the Gaussian process $\vp^{\Lambda} = \bigl(\varphi_{x}^{\Lambda} : x \in \mathbb{Z}^{d} \bigr)$ with law $\probPhi^{\om}$ determined by
  \begin{align} \label{eq:covField}
    \meanPhi^{\omega}\bigl[ \varphi_{x}^{\Lambda} \bigr]
    \,=\,
    0
    \quad \text{and} \quad
    \meanPhi^{\omega}\bigl[
      \varphi_{x}^{\Lambda}\, \varphi_{y}^{\Lambda}
    \bigr]
    \,=\,
    g_{\Lambda}^{\omega}(x, y),
    \qquad x, y \in \Lambda \cap \mathcal{C}_{\infty}(\omega).
  \end{align}
\end{definition} 

It is well known that powers of the GFF in $d=2$ require additive renormalisation. For this purpose let us recall that, for a centered Gaussian random variable $X$ with variance $\si^2 >0$ and integer $k \in \bbN$, the \emph{$k$-th Wick power} of $X$ is defined as 
\begin{align}\label{eq:def:wick:by:hermite}
  \wick{X^k}
  \,\ldef\,
  \mathbf{H}_k(X, \si^{2})
  \quad
  \left(
    \ldef k! \sum_{m=0}^{\lfloor \frac{k}{2}\rfloor} \frac{(-1)^m \si^{2m}}{m!(k-2m)!} X^{k-2m}
  \right),
\end{align}
where $\mathbf{H}_k$ is the $k$-th Hermite polynomial with variance $\si^2$ (cf.~\cite[Theorem 3.19]{janson1997gaussian} and \cite[Equation (I.18)-(I.19)]{simon2015p}). The family of Hermite polynomials forms a complete orthonormal basis the Hilbert space associated to the underlying Gaussian measure (cf.~\cite[Chapters 1-3]{janson1997gaussian}) and, most importantly for our purposes, for any $X,Y$ jointly Gaussian centered random variables and any $k,l \in \bbN$ 
\begin{equation} \label{eq:wick:covariance:gaussian:xy}
  \mean\Bigl[ \wick{X^k} \wick{Y^l} \Bigr] 
  \,=\, 
  \begin{cases} 
    k!\, \mean\bigl[ X Y \bigr]^{k} & \text{if } k = l,
    \\%[.5ex] 
    0 & \text{else,} 
  \end{cases}
\end{equation}
(cf.~\cite[Theorem 3.9]{janson1997gaussian} or \cite[Theorem I.3.]{simon2015p}), which also implies that $\mean\bigl[ \wick{X^{k}} \bigr] = 0$. Extending this Definition to an analytical function $F(x) = \sum_{k \geq 0} a_{k} x^{k}$ (cf.~\cite[Chapter 3.2.]{janson1997gaussian}) and applying this to the DGFF we finally encounter our main object of interest, which is 
\begin{align}\label{eq:def:wick:analytic:dgff:pointwise}
  \wick{F(\varphi_{x}^{\Lambda})}
  \,\ldef\,
  \sum_{k \geq 0} a_k\, \wick{\left(\varphi_{x}^{\Lambda}\right)^k}.
\end{align}
Using \eqref{eq:wick:covariance:gaussian:xy}, one easily verifies that, for all $x,y \in \La \cap \cC_{\infty}$ 
\begin{align}\label{eq:dgff:covariance:calculation:intro}
  \meanPhi^{\om} \Bigl[ 
    \wick{F(\varphi_{x}^{\Lambda})} \wick{F(\varphi_{y}^{\Lambda})}
  \Bigr]
  \,=\,
  \sum_{k\geq 0} a_{k}^{2} \meanPhi^{\om} \bigl[ 
    \varphi_{x}^{\Lambda}\, \varphi_{y}^{\Lambda}
  \bigr]^{k}
  \,=\,
  \sum_{k \geq 0} k!\, a_{k}^{2}\, g^{\om}_{\Lambda} (x,y)^{k} 
  \,=\,
  F_{2} \bigl( g^{\om}_{\Lambda} (x,y) \bigr),
\end{align}
with $F_{2}(x) \ldef \sum_{k \geq 0} k!\, a_{k}^{2} x^k$. Since this function plays a central role, we arrive at the following important
\begin{definition}\label{def:beta:entire}
  An analytic function $F(x) = \sum_{k \geq 0} a_{k} x^k$ is called \emph{$\be$-Fock-entire} if there exists an $M \equiv M(\be) < \infty$ such that 
  \begin{align*}
    F_{2}(x) 
    \,=\, 
    \sum_{k \geq 0} k! \, a_k^2 x^k
    \,\leq\,
    M e^{\be x},
    \qquad
    \forall\, x \geq 0.
  \end{align*}
\end{definition}
\begin{example}\label{example:be:fock:entire}
  Important examples of analytic functions satisfying Definition~\ref{def:beta:entire} include:
  \begin{enumerate}[(i)]
  \item
    \textit{Polynomials:} For any $K \in \bbN$ and coefficients $(a_{k})_{k=1}^{K}$, the polynomial $P(x) \ldef \sum_{k=0}^K a_{k} x^k$ is $\be$-Fock-entire for any $\be>0$. 
    
  \item
    \textit{Trigonometric and hyberbolic functions:} For $F \equiv \sin$, we have $F_{2}(x) = \sinh( x) = (e^{x} - e^{- x})/2$, hence $\sin$ is $1$-Fock-entire. Similarly, $\cos$, $\sinh$ and $\cosh$ are all $1$-Fock-entire.
    
  \item
    \textit{Exponentials:} For $F \equiv \exp$, we have $F_{2} \equiv \exp$, hence exponentials are also $1$-Fock-entire.
  \end{enumerate}
\end{example}
\begin{remark}
  \begin{enumerate}[(i)]
  \item
    While the $\be$-dependence may not be seen in standard definitions of this kind, it provides us with the flexibility to incorporate both polynomials, where we will always choose $\ga=1$ and use $\be$ to tune the regularity, and exponential type functions, where we will always choose $\be=1$ and tune regularity with $\ga$ (see Example~\ref{examples:mainresult1}). 

  \item
    The expression $F_2(1) = \sum_k k! a_k^2 (\equiv \Norm{F}{\cF}^2)$ denotes the so called Fock-norm of the function $F$ (cf.~\cite[Equation 2.8]{hu2009wick}). Hence, a function $F$ satisfying Definition~\ref{def:beta:entire} for some $\be>0$ will also be an element of Fock-space, explaining the name. For an introduction to Fock spaces, we refer to \cite[Chapter 4]{janson1997gaussian} and \cite[Chapter 2]{hu2009wick}. For an introduction from a mathematical physics perspective, we refer to \cite[Chapter 6.3]{glimm2012quantum} and \cite[Chapter 1.3]{simon2015p}. 

  \item
    The formula in \eqref{eq:dgff:covariance:calculation:intro} for the H\o egh-Krohn model introduced in \cite{hoegh1971general} can be found in \cite[Equations (3.6) and (3.12)]{albeverio1974wightman}. 
  \end{enumerate}
\end{remark}

\subsection{Assumptions and limit objects} 
First, we state the assumptions on the underlying graph under which our first main results will be shown. Given $\omega \in \Omega^{*}$, let $E(\mathcal{C}_{\infty}(\omega)) \ldef \{\{x, y\} \in \mathcal{O}(\omega) : x, y \in \mathcal{C}_{\infty}(\omega)\}$ be the edge set of the infinite cluster. We denote by $d^{\omega}$ the graph‐distance on the pair 
$\bigl(\mathcal{C}_{\infty}(\omega),\,E(\mathcal{C}_{\infty}(\omega))\bigr)$, namely, for any $x, y \in \mathcal{C}_{\infty}(\omega)$, $d^\omega(x, y)$ is the minimal length of a path joining $x$ and $y$ that consists only of open edges. For $x \in \mathcal{C}_{\infty}(\omega)$ and $r \geq 0$, let $B^{\omega}(x, r) \ldef {\{ y \in \mathcal{C}_{\infty}(\omega) : d^{\omega}(x, y) \leq \lfloor r \rfloor\}}$ be the closed ball with centre $x$ and radius $r$ with respect to $d^{\omega}$. Likewise, for $x \in \mathbb{Z}^d$ we write $B(x, r) \ldef \{y \in \mathbb{Z}^d : |x-y|_{1}\leq \lfloor r \rfloor\}$
for balls in $\mathbb{Z}^d$ centred at $x$, where $|\cdot|_{1}$ is the usual graph‐distance on $\mathbb{Z}^d$. More generally, $|\cdot|_{p}$ with $p \in [1, \infty]$ denotes the standard $p$‐norm on both $\mathbb{R}^d$ and $\mathbb{Z}^d$. 
\begin{definition}[Regular balls]\label{def:regBall}
  Let $C_{\mathrm{V}} \in (0, 1]$, $C_{\mathrm{riso}} \in (0, \infty)$ and $C_{\mathrm{W}} \in [1, \infty)$ be fixed constants.  For $x \in \mathcal{C}_{\infty}(\omega)$ and $n \geq 1$, we say a ball $B^{\omega}(x, n)$ is \emph{regular} if it satisfies the following conditions.
  \begin{enumerate}[(i)]
  \item
    Volume regularity of order $d$, that is,
    \begin{align*}
      C_{\mathrm{V}}\, n^{d} \;\leq\; |B^{\omega}(x, n)|.
    \end{align*}
    
  \item
    (Weak) relative isoperimetric inequality. There exists $\mathcal{S}^{\omega}(x, n) \subset \mathcal{C}_{\infty}(\omega)$ connected such that $B^{\omega}(x, n) \subset \mathcal{S}^{\omega}(x, n) \subset B^{\omega}(x, C_{\mathrm{W}} n)$ and
    \begin{align*}
      |\partial_{\mathcal{S}^{\omega}(x, n)}^{\omega} A|
      \;\geq\;
      C_{\mathrm{riso}}\, n^{-1}\, |A|
    \end{align*}
    for every $A \subset \mathcal{S}^{\omega}(x, n)$ with $|A| \leq \tfrac{1}{2}\, |\mathcal{S}^{\omega}(x, n)|$.
  \end{enumerate}
\end{definition}
\begin{assumption}\label{ass:cluster} 
  For some $\theta, \delta \in (0,1)$, $C_{\mathrm{V}} \in (0, 1]$, $C_{\mathrm{riso}} \in (0, \infty)$ and $C_{\mathrm{d}}, C_{\mathrm{W}} \in [1, \infty)$, assume that there exist $\Omega_{\mathrm{reg}} \in \mathcal{F}$ with $\prob[\Omega_{\mathrm{reg}}] = 1$ and a non-negative random variable $\const[N]{Nconst:ass:cluster}(\omega)$ such that $\const[N]{Nconst:ass:cluster}(\omega) < \infty$ for any $\omega \in \Omega_{\mathrm{reg}} \cap \Omega_{0}^{*}$, and for all $n \geq \const[N]{Nconst:ass:cluster}(\omega)$ the following hold.
  \begin{enumerate}[(i)]
  \item
    The ball $B^{\omega}(0, n)$ is $\theta$-\emph{very regular}, that is, for every $x \in B^{\omega}(0, n)$ and $r \geq n^{\theta/d}$ the ball $B^{\omega}(x, r)$ is regular with constants $C_{\mathrm{V}}$, $C_{\mathrm{riso}}$ and $C_{\mathrm{W}}$.  

  \item[(ii)]
    For any $x, y \in [-n, n]^{d} \cap \mathcal{C}_{\infty}(\omega)$,
    \begin{align*}
      d^{\omega}(x, y)
      \;\leq\;
      \bigl( C_{\mathrm{d}}\, |x - y|_{\infty} \bigr) \vee n^{1-\delta}.
    \end{align*}
  \end{enumerate}
\end{assumption}
\begin{remark} \label{rem:ass_cluster} 
  (i) Assumption~\ref{ass:cluster} holds, for instance, on supercritical i.i.d.\ Bernoulli percolation clusters for all $\theta \in (0,1)$, see~\cite{barlow2004random}. Moreover, Assumption~\ref{ass:cluster} is even satisfied for a class of percolation models with long range correlations, see \cite[Proposition~4.3]{sapozhnikov2017random} and \cite[Theorem~2.3]{drewitz2014chemical}. For more details and examples, including level sets of DGFFs, we refer to \cite[Examples~1.11--1.13]{deuschel2018quenched} and references therein.

  (ii) For any bounded and simply connected $D \subset \mathbb{R}^{d}$, setting $n D \ldef \{z \in \mathbb{R}^{d} : z/n \in D\}$,  there exists $r_{D} \in (0, \infty)$ such that $n D \subset [-r_{D} n, r_{D} n]^{d}$. Then, Assumption~\ref{ass:cluster}-(ii) immediately implies that, for $\prob_{0}$-a.e.\ $\omega$ and all $n \geq \max\{ \const[N]{Nconst:ass:cluster}(\omega)/r_{D}, 1 \}$, we have $n D \cap \mathcal{C}_{\infty}(\omega) \subset B^{\omega}(0, C_{\mathrm{d}} r_{D} n) \rdef B^{\om}_D(n)$.  
\end{remark}
\begin{assumption}\label{ass:pq}
  There exist $p, q \in [1, \infty]$ and $\theta \in (0, 1)$ satisfying
  \begin{align} \label{eq:cond:pq_rg}
    \frac{1}{p} \,+\, \frac{1}{q}
    \;<\;
    \frac{2(1-\theta)}{d-\theta},
  \end{align}
  such that for any $e \in E_{d}$,
  \begin{align*}
    \mean\bigl[ \omega(e)^{p} \bigr] \;<\; \infty
    \qquad \text{and} \qquad
    \mean\bigl[ \omega(e)^{-q} \indicator_{\{e \in \mathcal{O}\}} \bigr] \;<\; \infty,
  \end{align*}
  where we used the convention that $0/0 = 0$.
\end{assumption}
For any realisation $\om \in \Om^{*}$, let $X\equiv (X)_{t \geq 0}$ be the variable speed random walk (VSRW) in random environment with the generator $\cL^{\om}$ from \eqref{eq:defL}. Introducing the measure $\mu^{\om}(x) \ldef \sum_{y \sim x} \om(\{x,y\})$ for any $x \in \cC_{\infty}(\om)$, the VSRW is the continuous time Markov chain on $\cC_{\infty}(\om)$ that waits at a vertex $x$ for an $\mathrm{exp}(\mu^{\om}(x))$-distributed time and then chooses an adjacent vertex $y\sim x $ that is open in the cluster $\cC_{\infty}(\om)$ with probability $\om(\{x,y\})/\mu^{\om}(x)$. The VSRW is reversible with respect to the counting measure and, denoting its quenched law and the corresponding expectation when started from $x \in \cC_{\infty}(\om)$ by $\Prob^{\om}_x$ and $\Mean^{\om}_x$, the RCM-Green's function \eqref{eq:PDE_green_killed} describes the expected time the VSRW started at $x$ spends in $y$ before being killed when exiting a set $\La \subseteq \bbZ^d$, i.e.\
\begin{align}\label{eq:def:green:rcm:probabilistic}
  g^{\om}_{\La}(x,y)
  \,=\,
  \Mean^{\om}_x \left[ \int_0^{\ta_{\La}(X)} \indicator_{X_t = y} \, \md t\right]
  \,=\,
  \int_0^{\infty} p^{\om,\La}_{t}(x,y) \, \md t,
\end{align}
where $\ta_{\La}(X) \ldef \inf \{t \geq 0 \,:\, X_t \not\in \La\}$ is the first exit time from $\La$ and $p^{\om,\La}_{t}(x,y) \ldef \Prob^{\om}_x\left[X_t =y, t < \ta_{\La}(X)\right]$ is the killed heat kernel. The quenched homogensation of the VSRW on the macroscopic scale is described by the \emph{quenched invariance principle} (QIP), one of the most important results in this field. Denoting by $\Prob^{\Si}_x$ the law of a Brownian motion $W$ started at $x$ with non-degenerate covariance matrix $\Si^2 = \Si \cdot \Si^{T}$, it states:
\begin{theorem} [QIP \cite{deuschel2018quenched}]\label{thm:QIP:RG}
  Suppose there exist $\theta \in (0, 1)$ and $p, q \in [1, \infty]$ such that Assumptions~\ref{ass:law}, \ref{ass:cluster}-(i) and \ref{ass:pq} hold.  Set $X_{t}^{n} \ldef n^{-1} X_{t n^{2}}$ for any $n \in \mathbb{N}$ and $t \geq 0$. Then, for $\prob_{0}$-a.e.\ $\omega$, the process $X^{n} \equiv \bigl(X_t^{n}\bigr)_{t \geq 0}$, converges (under $\Prob_{0}^\omega$) in law towards a Brownian motion with law $\Prob^{\Si}_x$ on $\mathbb{R}^{d}$ with a deterministic non-degenerate covariance matrix $\Sigma^2$.
\end{theorem}
From now on we will fix both the covariance matrix $\Si$ from the preceding Theorem and a bounded domain $D \subset \mathbb{R}^{d}$, that is, an open, bounded and simply connected subset of $\mathbb{R}^{d}$, and throughout the paper we will assume that its boundary points are regular in the following sense (cf.~\cite[Exercise~10.2.20]{stroock2010probability}, see also \cite[Remark 1.9]{andres2025scaling} for a discussion).
\begin{definition} \label{def:regular}
  We call a point $z \in \partial D$ \emph{strongly regular} if $\Prob_{z}^{\Sigma}[ \tau_{\bar{D}}(W) = 0] = 1$. We say that $D$ is strongly regular if every point $z \in \partial D$ is strongly regular.
\end{definition}
Among other things, this assumption ensures the existence of the killed Green's kernel $g^{\Si}_{D}$ and the corresponding killed Green's function defined by $(G^{\Si}_D f)(x) \ldef \int_D g^{\Si}_D(x,y) f(y) \md y$, which is the unique solution (in a distributional sense) of the Poisson problem 
\begin{align}\label{eq:PDE_green_killed:continuous}
  \left\{
    \mspace{-6mu}
    \begin{array}{rcll}
      \De^{\Si} g^{\Si}_D (\cdot, y)
      &\mspace{-5mu}=\mspace{-5mu}& -\delta_{y}(\cdot),
      & \text{in}\; D,
      \\[.5ex] 
      g^{\Si}_D (\cdot, y)
      &\mspace{-5mu}=\mspace{-5mu}& 0, 
      &\text{in}\; D^c,
    \end{array}
  \right.
\end{align}
where, since $\Si^2$ is positive definite, $\De^{\Si} \ldef \nabla \cdot \Si \nabla$ is strictly elliptic. From a probabilistic perspective, $g^{\Si}_{D}$ is simply the occupation-time density of the Brownian motion $W$ with law $\Prob^{\Si}_{x}$, killed upon exiting $D$, akin to \eqref{eq:def:green:rcm:probabilistic}.

The continuum GFF with Dirichlet boundary conditions, can be interpreted as a generalisation of Brownian motion indexed by a $d$-dimensional set $D \subset \bbR^d$. Unlike Brownian motion, however, the free field in $d \geq 2$ can not be defined pointwise as a function, which is due to the variances at any point, corresponding to the diagonal of the Green's kernel, being infinite. Instead, following for instance \cite[Chapter 2]{sheffield2007gaussian}, one can naturally introduce the field as an element of the dual of the fractional Sobolev space $H^{s}_0(D)$, which is the Hilbert space arising as the $L^2(D)$-completion of the space $C^{\infty}_0(D)$ of smooth functions vanishing at the boundary with respect to the scalar product 
\begin{align}\label{eq:def:sobolev:scalarproduct:Hs}
  (f,g)_{H^{s}_0(D)} 
  \,\ldef\,
  \scprreal{f}{(I - \De)^s g}{D}.
\end{align}
Then, for any $s> d/2-1$, the field lives in the (distributional) space $H^{-s}(D)$, the dual space of $H^{s}_0(D)$ with respect to the $L^2(D)$ scalar product (cf.~\cite[Chapter~5]{evans2010partial} or \cite[Chapter~4]{taylor2023partial}). Hence, the regularity $s$ of the field decreases as the dimension $d$ increases. In $d=2$, the field lives in $H^{-s}(D)$ for any $s>0$. Given that $H^{0}(D)=L^2(D)$, this gives a first hint at the special role of the planar field, which we will revisit immediately after giving the definition (cf.~\cite[Definition 2.5]{sheffield2007gaussian}):
\begin{definition}[Continuum Gaussian free field (CGFF)]\label{def:cgff}
  Let $d \geq 2$, $s > d/2 -1$, $D \subset \mathbb{R}^{d}$ be a bounded domain and $\Sigma \in \mathbb{R}^{d \times d}$ be a positive definite matrix. Then the \emph{continuum Gaussian free field on $D$ with parameter $\Sigma$} is the unique random variable $\Psi \equiv \Psi^{D,\Si}$ on $H^{-s}(D)$ that assigns $H^{s}_0(D) \ni f \mapsto\scprfield{\Psi}{f}$ such that
  \begin{enumerate}[(i)]
  \item 
    $\Psi$ is a.s.\ linear, i.e.\ for all $f, g \in H^{s}_0(D)$ and $a, b \in \mathbb{R}$,
    \begin{align*}
      \scprfield{\Psi}{a f + b g}
      \;=\;
      a\, \scprfield{\Psi}{f} + b\, \scprfield{\Psi}{g}
      \qquad \text{a.s.},
    \end{align*}

  \item for all $f \in H^{s}_0(D)$,
    \begin{align*}
      \scprfield{\Psi}{f}
      \overset{\mathrm{law}}{\;=\;}
      \mathcal{N}\bigl(0, \scprreal{f}{G^{\Si}_Df}{D} \bigr).
    \end{align*}
  \end{enumerate}
  Moreover, we denote by $\probPhi^{\Si}$ the law of $\Psi\equiv \Psi^{D,\Si}$ on the (separable) Hilbert space $H^{-s}(D)$.
\end{definition}
The existence of powers or other nonlinear functionals of the GFF boils down to the integrability of the singularity of the Green's kernel, which, as is well known, obeys the following bounds 
\begin{align}\label{eq:continuous:green:pointwise:log}
  g^{\Si}_D(x,y)
  \,\leq\,
  \left\{
    \mspace{-6mu}
    \begin{array}{lcll}
      2 C_{\Si} \log \left(\dfrac{1}{|x-y|_2}\right) + c,
      &\quad d=2,
      \\[2ex] 
      \tilde{C}_{\Si}(d) |x-y|_2^{2-d},
      &\quad d\geq3
    \end{array}
  \right.
\end{align}
(cf.~\cite[Chapter 11.2]{stroock2010probability}) for constants $C_{\Si}$, $\tilde{C}_{\Si}(d)$, $c < \infty$. Consequently, the Green's kernel is in $L^p(D)$ for every $p < \infty$ when $d = 2$, and for every $p < d/(d-2)$ when $d > 2$. Since $g_{D}^{\Si}$ is square integrable in $d = 3$, one can still define the square of the CGFF in this dimension (cf.~\cite{sznitman2013scaling, deuschel2024ray}), which plays a particularly important role in modern probability theory, appearing, for instance, in the study of occupation times of critical loop soups \cite{le2010markov, werner2021lecture} or random interlacements \cite{sznitman2013scaling, deuschel2024ray}. As we are studying scaling limits of the squared inhomogeneous GFF in $d \in \{2, 3\}$ and its relation to random Schr\"{o}dinger operators in a forthcoming paper \cite{peterslowik2026schroedinger}, we will only focus on the $d = 2$ case in this article. Furthermore, the logarithmic singularity of the planar Green's kernel allows for a much broader class of nonlinear functionals to be considered, which leads us to the construction of the limit object via an ultraviolet cutoff. We will give a sketch here and refer to Proposition~\ref{prop:L2:convergence:cgff:smeared} for the precise statement. As, for example, done in \cite[Chapter 8.5]{glimm2012quantum}, we define a smeared field $\Psi^{\ve}(x) \ldef \scprfield{\Psi}{\rho^{\ve}_x}$, with $\rh^{\ve}_x \ldef \ve^{-d} \rh(\tfrac{\cdot - x}{\ve})$, and a non-negative $\rh \in C^{\infty}_c(\bbR^d)$ such that $\int \rh(x)\, \md x  = 1$. For any $\ve > 0$ fixed, $\Psi^{\ve} \ldef \{\Psi^{\ve}(x) : x \in D\}$ is a continuous Gaussian process with finite variance $\si^{\ve}(x)^2 \ldef \scprreal{\rh^{\ve}_x}{G^{\Si}_D \rh^{\ve}_x}{D} < \infty$ (cf.\ Proposition~\ref{prop:weak:convergence:smooth:approximation}). Recalling \eqref{eq:def:wick:by:hermite} and \eqref{eq:def:wick:analytic:dgff:pointwise}, we can now rigorously define for any $d \geq 2$ and any fixed $\ve > 0$ 
\begin{align*}
  \wick{F\left(\frac{\ga \Psi^{\ve}(x)}{\sqrt{\th_0}}\right)}
  \,\ldef\,
  \sum_{k\geq 0} a_k \left(\frac{\ga}{\sqrt{\th_0}}\right)^k \wick{\Psi^{\ve}(x)^k}
  \qquad \left(\text{any} \; d \geq 2\right),
\end{align*}
where, since $\si^{\ve}(x)^2 < \infty$, the Wick power $\wick{\Psi^{\ve}(x)^k} = \mathbf{H}_k(\Psi^{\ve}(x),\si^{\ve}(x)^2)$ is well-defined pointwise. While this construction works in any dimension $d \geq 2$, it is only possible in $d = 2$ to pass $\ve \downarrow 0$ and thus remove the UV-cutoff. More precisely, we show in Proposition~\ref{prop:L2:convergence:cgff:smeared} that in $d=2$ one can, for any $\be$-Fock-entire function $F$, $f \in L^2(D)$, and $|\ga| < \sqrt{\th_0/\be C_{\Si}}$, define 
\begin{align}\label{def:cgff:analytic:limit:in:L2:introduction}
  \scprfield{\wick{F\left(\frac{\ga \Psi}{\sqrt{\th_0}}\right)}}{f}
  \,\ldef\,
  \lim_{\ve \downarrow 0} \scprreal{\wick{F\left(\frac{\ga \Psi^{\ve}}{\sqrt{\th_0}}\right)}}{f}{D}
  \quad
  \text{in}
  \; L^2(\probPhi^{\Si})
  \qquad \left(\text{only} \; d = 2 \right).
\end{align}
This is a (generally non-Gaussian) random variable, unique in an $L^2(\probPhi^{\Si})$ sense, linear in the sense of Definition~\ref{def:cgff}-(i), and its first two moments are given by 
\begin{align*}
  \meanPhi^{\Si} \left[\scprfield{\wick{F\left(\frac{\ga \Psi}{\sqrt{\th_0}}\right)}}{f}\right]
  &\,=\,
  a_0 \scprreal{1}{f}{D} 
  \\[.5ex] %\intertext{and} 
  \meanPhi^{\Si} \left[\scprfield{\wick{F\left(\frac{\ga \Psi}{\sqrt{\th_0}}\right)}}{f}^2\right] 
  &\,=\,
  \scprreal{f}{F_2\left(\ga^{2} \th_0^{-1} G^{\Si}_D\right) f}{D},
\end{align*}
where $(F_2\left(\ga^{2} \th_0^{-1} G^{\Si}_D\right) f) (x) \ldef \int_D F_2 \left(\ga^{2} \th_0^{-1} g^{\Si}_D(x,y) \right) f(y) \, \md y$ and $F_2$ is defined as in \eqref{eq:dgff:covariance:calculation:intro}. A construction following a similar approach can, for instance, be found in \cite[Chapter 5]{sznitman2013scaling} or \cite[Chapter 5]{deuschel2024ray}, which introduce the square of the whole space GFF in $d=3$ in this way. Equivalently, as done, for instance, in \cite[Chapter 3.4.2]{werner2021lecture} for the square, construction can be done using partial sums of the white noise decomposition of the CGFF, which at its heart is again an ultraviolet cutoff. 

\subsection{Results}
As discussed above, the existence of nonlinear functionals of the GFF relies critically on the integrability of the Green's kernel near the diagonal. One of our main technical contributions is hence to establish practical bounds for the killed Green's kernel in $d\geq 2$ and the (whole space) Green's kernel in $d\geq3$ in our setting with degenerate, unbounded ergodic weights, resembling those in \eqref{eq:continuous:green:pointwise:log}:
\begin{theorem}\label{thm:green:pointwise:upperbound}
  Let $d \geq 2$ and suppose there exist $\theta \in (0, 1)$ and $p, q \in [1, \infty]$ such that Assumptions~\ref{ass:law}, \ref{ass:cluster}-(i) and~\ref{ass:pq} hold. Further, let $\Om_c \in \cF$ and $\const[N]{Nconst:ergodic:constants} \equiv \const[N]{Nconst:ergodic:constants}(\om,1) (< \infty \;\forall \om \in \Omega_{\mathrm{reg}} \cap \Omega_{0}^{*} \cap \Omega_{c})$ be as in Corollary~\ref{cor:KrengelPyke}. Then, there exists $\ka \equiv \ka(\th,d,p,q) \in (0,\infty)$ such that for any $R \geq 1$, the following holds. There exist constants $\const[C]{const:pointwise:schroedinger:bound:2d}, \const[C]{const:pointwise:schroedinger:bound:3d} < \infty$ depending only on $R,\th,d,p,q$ such that for any $\omega \in \Omega_{\mathrm{reg}} \cap \Omega_{0}^{*} \cap \Omega_{c}$ and any $n \geq \const[N]{Nconst:ass:cluster}(\omega) \vee \const[N]{Nconst:ergodic:constants}(\om,1)\vee R^\ka \rdef \const[N]{Nconst:ass:green:bound}(\om,R)$ 
  \begin{align*}
    g^{\om}_{B^{\om}(0,Rn)}(x,y)
    \,\leq\,
    2 C_{\mathrm{HK}} \log\left(\frac{n}{d^{\om}(x,y)\vee 1}\right) \,+\, \const[C]{const:pointwise:schroedinger:bound:2d} N_{\mathrm{HK}}(\om)
    \quad 
    \forall x,y \in B^{\om}(0,Rn)
  \end{align*}
  in $d=2$ and, for the whole space Green's function in $d\geq 3$,
  \begin{align*}
    g^{\om}_{\bbZ^d}(x,y)
    \,\leq\,
    \const[C]{const:pointwise:schroedinger:bound:3d}  N_{\mathrm{HK}}(\om)^{d-1} \left(d^{\om}(x,y) \vee 1 \right)^{2-d}
    \quad 
    \forall x,y \in B^{\om}(0,Rn),
  \end{align*}
  with $C_{\mathrm{HK}}<\infty$ and $N_{\mathrm{HK}}(\om) \ldef 4 \const[N]{Nconst:ass:green:bound}(\om,R)^2 \vee 1$ from the heat kernel bounds in Theorem~\ref{thm:heat:kernel:bounds}. 
\end{theorem}
Assumptions~\ref{ass:law} and~\ref{ass:cluster}-(i) imply that $\prob_{0}\bigl[ \Omega_{\mathrm{reg}} \cap \Omega_{0}^{*} \bigr] = 1$, while, as shown in Corollary~\ref{cor:KrengelPyke} (cf.~\cite[Corollary 5.5]{andres2025scaling}, \cite{krengel1987uniform}), Assumption~\ref{ass:pq} implies that $\prob_0 \left[\Omega_{c}\right]=1$. Hence, $\const[N]{Nconst:ass:green:bound}(\om,R)$ and, consequently, $N_{\mathrm{HK}}(\om)$ are almost surely finite under $\prob_0$. In other words, the bounds above are effective for $\prob_0$-a.e.\ $\om$. By additionally imposing Assumption~\ref{ass:cluster}-(ii) and recalling Remark~\ref{rem:ass_cluster}-(ii), we obtain a useful corollary for the rescaled killed Green's kernel on $nD \cap \cC_{\infty}(\om)$, which is proved in Chapter 2.
\begin{cor}\label{cor:green:pointwise:upperbound:corollary}
  Let $d \geq 2$ and $D \subset \bbR^d$ bounded. Suppose that there exist $\theta \in (0, 1)$ and $p, q \in [1, \infty]$ such that Assumptions~\ref{ass:law}, \ref{ass:cluster} and~\ref{ass:pq} hold. Then, there exist constants $\const[C]{const:pointwise:schroedinger:bound:corollary:2d}, \const[C]{const:pointwise:schroedinger:bound:corollary:3d} < \infty$ depending only on $r_D,\th,d,p,q$ such that for any $n \geq \const[N]{Nconst:ass:green:bound}(\om,C_d r_D)$ and $\prob_0$-a.e.\ $\om$
  \begin{equation*}
    n^{d-2} g^{\omega}_{nD}(\lfloor nx \rfloor, \lfloor ny \rfloor)
    \leq 
    \begin{cases}
      2 C_{\mathrm{HK}} \log\left( \dfrac{1}{|x-y|_2} \right) + \const[C]{const:pointwise:schroedinger:bound:corollary:2d} N_{\mathrm{HK}}(\omega), & d=2, \\[2ex]
      \const[C]{const:pointwise:schroedinger:bound:corollary:3d} N_{\mathrm{HK}}(\omega)^{d-1} |x-y|_2^{2-d}, & d \geq 3
    \end{cases}
  \end{equation*}
  for any $x \not= y$. If $d \geq 3$, the same bound also holds for $n^{d-2} g^{\omega}_{\bbZ^d}(\lfloor nx \rfloor, \lfloor ny \rfloor)$ and any $x,y \in D$.
\end{cor}
\begin{remark}\label{rem:Green:bounds:remark}
  \begin{enumerate}[(i)]
  \item
    For simple random walk, these bounds are well known \cite[Chapter 4]{lawler2010random}. In $d\geq 3$, similar bounds have been established for SRW on supercritical percolation clusters by Barlow and Hambly in \cite[Theorem 1.2]{hambly2009parabolic} and, for walks on a larger class of random subgraphs, by Sapozhnikov in \cite[Theorem 1.17]{sapozhnikov2017random}. All aforementioned models are included in our Assumptions (cf.\ Remark~\ref{rem:ass_cluster}-(i)). 

  \item
    The enormous utility of Theorem~\ref{thm:green:pointwise:upperbound} for our purposes stems from the fact that $N_{\mathrm{HK}}$, the random time until the heat kernel bounds become effective, does not depend on $x$ or $y$. This is achieved by tailoring the heat kernel bounds in \cite[Theorem 3.2]{andres2019heat} to our setting, which is done in Theorem~\ref{thm:heat:kernel:bounds} and relies crucially on the uniform ergodic theorem in ~\cite{krengel1987uniform} (see Proposition~\ref{prop:krengel_pyke}) and Assumption~\ref{ass:cluster}-(i). 
    
  \item
    In contrast to the heat kernel bounds for SRW on supercritical percolation by Barlow in \cite[Theorem 1]{barlow2004random} or by Sapozhnikov in \cite[Theorem 1.13]{sapozhnikov2017random}, we do not have control on the moments or tails of the almost surely finite random variable $N_{\mathrm{HK}}$, which in their cases depends on $x$ and corresponds to $S_x$ and $R_{\mathrm{vgb}}(x)$ respectively in \cite[Lemma 2.24]{barlow2004random} and \cite[Theorem 1.13(c)]{sapozhnikov2017random}. 
    
  \item
    The near diagonal control on the Green's function obtained in Theorem~\ref{thm:green:pointwise:upperbound} complements the on-diagonal result obtained in ~\cite{andres2020green}, where a local limit theorem has been shown. However, due to the constant $C_{\mathrm{HK}}$ disagreeing with the correct constant $\bar{g}$ given in \cite[Theorem 1.3]{andres2020green} and the missing control on the tails of $N_{\mathrm{HK}}$, the obtained upper bound does not seem to be sufficient for the verification of the missing \cite[Assumption (A.2)]{ding2017maximum}, required for proving convergence of the centered maximum of the DGFF. For a detailed discussion of the remaining assumptions we refer to \cite[Conjecture 1.16]{andres2025scaling}. 
  \end{enumerate}
\end{remark}
The constant $C_{\mathrm{HK}}$ determines the integrability of exponential functions of the Green's kernel near the diagonal. Hence, just as $C_{\Si}$ did in \eqref{def:cgff:analytic:limit:in:L2:introduction}, it will give us another constraint on the choice of $\ga$ in the scaling limits, which we will state now. Recalling \eqref{eq:def:wick:analytic:dgff:pointwise}, we define for any analytic function $F$ and any measurable function $f \colon D \to \bbR$,
\begin{align}\label{eq:def:dgff:nonlinear:tested}
  f \mapsto 
  \scprfield{\wick{F(\ga \Phi_n)}}{f}
  \,\ldef\,
  \int_D \wick{F(\ga \, \vp^{nD}_{\lfloor nx \rfloor})} \indicator_{\lfloor nx \rfloor \in \cC_{\infty}(\om)} f(x)\, \md x.
\end{align}
Restricting to $f \in H^{s}_0(D)$, we may interpret $\wick{F(\ga \Phi_n)}$ as a linear functional on $H^{s}_0(D)$.
\begin{theorem}\label{thm:mainresult:1}
  Let $d = 2$ and $D \subset \mathbb{R}^{d}$ be a bounded, strongly regular domain. Suppose there exist $\theta \in (0, 1)$ and $p, q \in [1, \infty]$ such that Assumptions~\ref{ass:law},~\ref{ass:cluster} and~\ref{ass:pq} hold. Further, let $F$ be a $\be$-Fock entire analytic function. Then, for any
  \begin{align}\label{eq:ga:convergence:condition:mainresult}
    |\ga|
    \,<\,
    \sqrt{\frac{\th_0}{\be} \left(\frac{1}{C_{\Si}} \wedge \frac{\th_0}{C_{\mathrm{HK}}}\right)},
  \end{align}
  and $\prob_0$-a.e.\ $\om$ under $\probPhi^{\om}$ 
  \begin{align}
    \wick{F(\ga \Phi_n)}
    \underset{n \to \infty}{\overset{\text{law}}{\;\longrightarrow\;}}
    \wick{\th_0 F\left(\frac{\ga \Psi}{\sqrt{\th_0}}\right)}
  \end{align}
  in the strong topology of $H^{-s}(D)$ for any $s >  \ga^2 \be C_{\mathrm{HK}}$, where $\th_0 \equiv \prob[0 \in \cC_{\infty}(\om)]$ and $\Psi \equiv \Psi^{D,\Si}$ with $\Sigma$ as in Theorem~\ref{thm:QIP:RG}.
\end{theorem}
\begin{example}\label{examples:mainresult1}
  Under the assumptions of the Theorem, we collect some examples (cf.\ Example~\ref{example:be:fock:entire}):
  \begin{enumerate}[(i)]
  \item
    \textit{Polynomials:} Polynomials are $\be$-Fock-entire for any $\be>0$. Hence, for any given $s>0$, we can pick $\be>0$ sufficiently small such that we may choose $\ga=1$ in \eqref{eq:ga:convergence:condition:mainresult} and $s>\be C_{\mathrm{HK}}>0$. Thus, for any polynomial $P(x) \equiv \sum_{k=1}^K a_k x^k$,
    \begin{align*}
      \wick{P(\Phi_n)}
      \underset{n \to \infty}{\overset{\text{law}}{\;\longrightarrow\;}}
      \wick{\th_0 P\left(\frac{\Psi}{\sqrt{\th_0}}\right)}
      \quad
      \left(\,=\, \sum_{k =1}^K a_k \th_0^{1-\frac{k}{2}} \wick{\Psi^k}\right)
    \end{align*}
    in $H^{-s}(D)$ for any $s>0$. 

  \item
    \textit{Trigonometric and hyperbolic functions:} Since $\sin,\cos,\sinh$, and $\cosh$ are all $1$-Fock-entire, we may naturally choose $\be=1$, hence \eqref{eq:ga:convergence:condition:mainresult} becomes 
    \begin{align}\label{eq:ga:condition:applied}
      |\ga|< \sqrt{\th_0\left(\frac{1}{C_{\Si}} \wedge \frac{\th_0}{C_{\mathrm{HK}}}\right)}.
    \end{align}
    Thus, the Sobolev-regularity will be governed by the choice of the parameter $\ga$. Taking, for instance, $F\equiv\sin$, 
    \begin{align*}
      \wick{\sin(\ga \Phi_n)}
      \underset{n \to \infty}{\overset{\text{law}}{\;\longrightarrow\;}}
      \wick{\th_0 \sin\left(\frac{\ga \Psi}{\sqrt{\th_0}}\right)}
    \end{align*}
    in $H^{-s}(D)$ for any $s > \ga^2 C_{\mathrm{HK}}$.

  \item
    \textit{Exponentials:} Again, since $F \equiv \exp$ is $1$-Fock-entire, we choose $\be=1$ and $\ga$ as in \eqref{eq:ga:condition:applied}, which yields
    \begin{align*}
      \wick{e^{\ga \Phi_n}}
      \underset{n \to \infty}{\overset{\text{law}}{\;\longrightarrow\;}}
      \wick{\th_0 e^{\ga \th_0^{-\frac{1}{2}} \Psi}}
    \end{align*}
    in $H^{-s}(D)$ for any $s > \ga^2 C_{\mathrm{HK}}$. 

  \item
    \textit{LQG/GMC-measure:} In Theorem~\ref{thm:marginal:convergence:analytic:field}, we actually show that, in the setting of Theorem~\ref{thm:mainresult:1}, the marginals $\scprfield{\wick{F(\ga \Phi_n)}}{f}$ converge for any $f \in L^2(D)$. Hence, for any Lebesgue measurable set $A \subseteq D$ with positive measure, choosing $f=\indicator_A$ and $\ga$ as in \eqref{eq:ga:condition:applied}, we get 
    \begin{align*}
      \int_A \wick{e^{\ga \vp^{nD}_{\lfloor nx \rfloor}}} \indicator_{\lfloor nx \rfloor \in \cC_{\infty}(\om)}\, \md x 
      \underset{n \to \infty}{\overset{\text{law}}{\;\longrightarrow\;}}
      \th_0 \scprfield{\wick{e^{\ga \th_0^{-1/2} \Psi}}}{\indicator_{A}},
    \end{align*}
    which may be seen as the scaling limit of a (degenerate) LQG-measure in the subcritical $L^2$ regime, akin to the discrete Liouville measure introduced in \cite[Theorem 5.12]{rhodes2014gaussian}. In the latter, construction was done following the standard approach presented in \cite{robert2010gaussian}, which seems to be different from ours. For a comprehensive introduction to the GMC associated to the GFF, we refer to \cite{berestycki2024gaussian,chatterjee2025liouville}.
  \end{enumerate}
\end{example}
In fact, our approach enables us to establish the following joint scaling limit. 
\begin{theorem}\label{thm:mainresult:2}
  Let $d = 2$, $K \in \bbN$ and $D \subset \mathbb{R}^{d}$ be a bounded, strongly regular domain. Suppose there exist $\theta \in (0, 1)$ and $p, q \in [1, \infty]$ such that Assumptions~\ref{ass:law},~\ref{ass:cluster} and~\ref{ass:pq} hold. Further, let $F_k$ be $\be_k$-Fock-entire functions. Then, for any 
  \begin{align}
    |\ga_k|
    \,<\,
    \sqrt{\frac{\th_0}{\be_k} \left(\frac{1}{C_{\Si}} \wedge \frac{\th_0}{C_{\mathrm{HK}}}\right)},
  \end{align}
  and $\prob_0$-a.e.\ $\om$ under $\probPhi^{\om}$
  \begin{align}
    \left(\wick{F_1(\ga_1 \Phi_n)},\dots,\wick{F_K(\ga_K \Phi_n)}\right)
    \,\underset{n \to \infty}{\overset{\text{law}}{\;\longrightarrow\;}}\,
    \left(\wick{\th_0 F_1\left(\frac{\ga_1 \Psi}{\sqrt{\th_0}}\right)}, \dots, \wick{\th_0 F_K\left(\frac{\ga_K \Psi}{\sqrt{\th_0}}\right)}\right)
  \end{align}
  in the strong topology of $H^{-s_1}(D) \times \dots \times H^{-s_K}(D)$ for any $s_k > \ga_k^2 \be_k C_{\mathrm{HK}}$, with $\Sigma$ as in Theorem~\ref{thm:QIP:RG}.
\end{theorem}
Recalling that polynomials are $\be$-Fock-entire for any $\be>0$ and arguing as in Example~\ref{examples:mainresult1}-(i) yields the following immediate Corollary:
\begin{cor}\label{cor:finite:chaos:convergence}
  Let $d = 2$, $K \in \bbN$ and $D \subset \mathbb{R}^{d}$ be a bounded, strongly regular domain. Suppose there exist $\theta \in (0, 1)$ and $p, q \in [1, \infty]$ such that Assumptions~\ref{ass:law},~\ref{ass:cluster} and~\ref{ass:pq} hold. Then, for $\prob_0$-a.e.\ $\om$ under $\probPhi^{\om}$
  \begin{align*}
    \bigl(\wick{\Phi_n^k }\bigr)_{k=1}^K
    \,\underset{n \to \infty}{\overset{\text{law}}{\;\longrightarrow\;}}\,
    \Bigl(\th_0^{1-\frac{k}{2}}\wick{\Psi^k }\Bigr)_{k=1}^K
  \end{align*}
  in the strong topology of $H^{-s_1}(D) \times \dots \times H^{-s_K}(D)$ for any $s_k > 0$, with $\Sigma$ as in Theorem~\ref{thm:QIP:RG}.
\end{cor}
\begin{remark}\label{rem:remark:finite:chaos:convergence}
  \begin{enumerate}[(i)]
  \item
    As an interpretation, one may define the \emph{$k$-percolated GFF} as $\Psi_{\mathrm{perc}(k)} \ldef \th_0^{(2-k)/(2k)} \Psi$ for any $k$. Notice that, as infinite dimensional Gaussians, $\Psi_{\mathrm{perc}(k)}$ and $\Psi_{\mathrm{perc}(l)}$ are mutually singular for every $k \not= l$ if $\th_0<1$ (cf.~\cite[Chapter 6]{bogachev1998gaussian}). Hence, Corollary~\ref{cor:finite:chaos:convergence} reads as 
    \begin{align*}
      \bigl(\wick{\Phi_n^k} \bigr)_{k=1}^{K}
      \underset{n \to \infty}{\overset{\text{law}}{\;\longrightarrow\;}}
      \bigl(\wick{\Psi_{\mathrm{perc}(k)}^k} \bigr)_{k=1}^{K}.
    \end{align*}
    Thus, as soon as $\th_0<1$, the first $K$ powers of $\Phi_n$ converge to the powers of $K$ mutually singular, hence different, fields $\Psi_{\mathrm{perc}(k)}$. If $\th_0=1$, this does of course not happen.

  \item
    The tightness established here is of the correct order $s>0=d/2-1$. Choosing $K=1$, our result provides an improvement over the $s>d/2$ tightness established in \cite[Theorem 1.11]{andres2025scaling} in $d=2$. Note that for $s>d/2$, functions in $H^s_0(D)$ are Lebesque a.e.\ bounded, making this case fundamentally different from the situation where $s\leq d/2$. Indeed, the new bounds in Theorem~\ref{thm:green:pointwise:upperbound} and the method used to show Proposition~\ref{prop:DGFF:tightness} further allow to establish the correct $s>d/2-1$ tightness of the inhomogeneous field in all remaining dimensions $d\geq 3$ (cf.~\cite{peterslowik2026schroedinger}).
  \end{enumerate}
\end{remark}
As another Corollary of Theorem~\ref{thm:mainresult:2}, we now determine the scaling limits of the inhomogeneous analogs of the aforementioned lattice approximations for the self-interacting field theories in \cite{simon2015p,glimm2012quantum,hoegh1971general,albeverio1974wightman,frohlich1976classical}. For an analytic $V$ and a measurable function $g \colon D \to \bbR$, we denote the interaction potential by 
\begin{align*}
  \wick{V^{g}_n}(\vp^{nD})
  \,\ldef\,
  \int \wick{V(\ga \, \vp^{nD}_{\lfloor nx \rfloor})} \indicator_{\lfloor nx \rfloor \in \cC_{\infty}(\om)} g(x)\, \md x
  \quad \bigl( \,=\, \scprfield{\wick{V(\ga \Phi_n)}}{g} \bigr).
\end{align*}
Akin to, for instance, \cite[Chapter 4]{albeverio1974wightman}, we define the self interacting measure $\mu^{\om,V^{g}_n}_n$ with respect to the law $\probPhi^{\om}$ of the inhomogeneous DGFF  (cf. Definition~\ref{def:inhomogeneous:DGFF}) as 
\begin{align*}
  \frac{\md \mu^{\om,V^{g}_n}_n}{\md \probPhi^{\om}}(\vp)
  \,=\,
  \frac{e^{-\wick{V^{g}_n}(\vp)}}{\meanPhi^{\om}\left[e^{-\wick{V^{g}_n}(\vp)}\right]}.
\end{align*}
Recalling the form of the density of the inhomogeneous DGFF (cf.~\cite[Definition 1.2]{andres2025scaling}), we can equivalently give an expression of $\mu^{\om,V^{g}_n}_n$ in the form of a Gibbs measure on $\bbR^{\bbZ^d}$. Namely, 
\begin{align}\label{eq:def:gibbs:measure:interacting}
  \mu^{\om,V^{g}_n}_n(\md \vp) 
  \,=\,
  \frac{1}{Z_{V^{g}_n,nD}^{\omega}}
  \me^{-\cH^{\om}_{V^{g}_n}(\vp)}
  \mspace{-6mu}
  \;\prod_{x \in nD \cap \mathcal{C}_{\infty}(\omega)}\mspace{-12mu} \mathrm{d}\varphi_{x}
  \prod_{x \in (nD \cap \mathcal{C}_{\infty}(\omega))^{\mathrm{c}}}\mspace{-12mu} \delta_{0}[\mathrm{d} \varphi_{x}],
\end{align}
with partition function $Z_{V^{g}_n,nD}^{\omega} \equiv Z^{\om}_{nD} \meanPhi^{\om}\bigl[e^{-\wick{V^{g}_n}(\vp)}\bigr]$, where $Z^{\om}_{nD}$ is as in \cite{andres2025scaling}, and, most importantly, the Hamiltonian is given by 
\begin{align}\label{eq:def:hamiltonian:V}
  \cH^{\om}_{V^{g}_n}(\vp)
  \,\ldef\,
  \frac{1}{2} \cE^{\om}(\vp) + \wick{V^{g}_n}(\vp).
\end{align}
Here, $\cE^{\om}$ is the energy with respect to the Laplacian from \eqref{eq:defL}, defined as
\begin{align}\label{eq:def:rcm:laplace:energy}
  \cE^{\om}(\vp) 
  \,\ldef\,
  \scpr{\vp}{-\cL^{\om} \vp}{\cC_{\infty}(\om)}
  \,=\, 
  \sum_{\{x,y\} \in E(\cC_{\infty}(\om))} \om(\{x,y\}) \left(\vp_x - \vp_y \right)^2,
\end{align}
for functions $\vp \colon \bbZ^d \to \bbR$ with finite support. 

Furthermore, denoting by $\probPhi^{\Si}_{\th_0}$ the law of $\sqrt{\th_0} \Psi$ , we define the limiting measure $\mu^{\Si,V,\th_0}$ on $\cP(H^{-s}(D))$, the space of probability measures on $H^{-s}(D)$, as 
\begin{align*}
  \frac{\md \mu^{\Si,V,\th_0}}{\md \probPhi^{\Si}_{\th_0}}
  \,=\,
  \frac{e^{-\wick{V^{g}}(\sqrt{\th_0}\Psi)}}{\meanPhi^{\Si}\left[e^{-\wick{V^{g}}(\sqrt{\th_0}\Psi)}\right]},
  \quad
  \text{where}
  \;\;
  \wick{V^{g}}(\sqrt{\th_0}\Psi) 
  \,\ldef\,
  \scprfield{\wick{\th_0 V\left(\frac{\ga (\sqrt{\th_0}\Psi) }{\th_0}\right)}}{g}.
\end{align*}
While this may formally be done in the physics literature, we can of course not define $\mu^{\Si,V,\th_0}$ via a density as in \eqref{eq:def:gibbs:measure:interacting}, since an infinite dimensional Lebesgue measure on $H^{-s}(D)$ does not exist. Moreover, as touched upon in Remark~\ref{rem:remark:finite:chaos:convergence}-(iii), $\Psi$ and $\th_0^{1/2} \Psi$ are mutually singular if $\th_0<1$. Hence, it is important to use the law $\probPhi^{\Si}_{\th_0}$ instead of $\probPhi^{\Si}$ in the Definition above.
\begin{cor}
  Let $d = 2$ and $D \subset \mathbb{R}^{d}$ be a bounded, strongly regular domain. Suppose there exist $\theta \in (0, 1)$ and $p, q \in [1, \infty]$ such that Assumptions~\ref{ass:law},~\ref{ass:cluster} and~\ref{ass:pq} hold. Further, let $V$ be a $\be$-Fock-entire function. Then, for $\ga$ satisfying \eqref{eq:ga:convergence:condition:mainresult}, $s'>\ga^2 \be C_{\mathrm{HK}}$ and $g \in H^{s'}(D)$ assume that for $\prob_0$-a.e.\ $\om$ 
  \begin{equation}\label{eq:ui:condition:self:interacting}
    \left\{e^{-\scprfield{\wick{V(\ga \Phi_n)}}{g}}\right\}_{n \in \bbN} \text{ is UI under } \probPhi^{\om},
    \; \text{and}\; 
    \meanPhi^{\Si}\big[ e^{-\wick{V^{g}}(\sqrt{\th_0}\Psi)} \big] < \infty.
  \end{equation}
  Then, for any $s>0$, 
  \begin{align*}
    \mu^{\om,V^{g}_n}_n
    \underset{n \to \infty}{\implies}
    \mu^{\Si,V,\th_0} 
    \qquad 
    \text{on} \;\cP(H^{-s}(D))
  \end{align*}
  for $\prob_0$-a.e.\ $\om$, with $\Sigma$ as in Theorem~\ref{thm:QIP:RG}.
\end{cor}
\begin{proof}
  Note that by Jensen and since $\meanPhi^{\om}\left[\wick{V^{g}_n}(\vp)\right] = a_0 \int_D g(x) \md x \leq a_0 c \Norm{g}{2}$, 
  \begin{align*}
    \meanPhi^{\om}\left[e^{-\wick{V^{g}_n}(\vp)}\right]
    \,\geq\,
    e^{-\meanPhi^{\om}\left[\wick{V^{g}_n}(\vp)\right]}
    \,\geq\,
    e^{-a_0 c \Norm{g}{2}}
    >0,
  \end{align*}
  since $g \in H^{s'}(D) \subset L^2(D)$ and $D$ is bounded. Arguing as in Example~\ref{examples:mainresult1}-(i) and applying Theorem~\ref{thm:mainresult:2} with $F_1 \equiv \id$, $F_2 \equiv V$, $\ga_1 =1$ and $\ga_2 = \ga$, where $\ga$ satisfies \eqref{eq:ga:convergence:condition:mainresult}, yields that 
  \begin{align}\label{eq:weak:convergence:applied:self:interacting}
    \bigl(\Phi_n, \wick{V(\ga \Phi_n)} \bigr)
    \,\underset{n \to \infty}{\overset{\text{law}}{\;\longrightarrow\;}}\,
    \left(\sqrt{\th_0}\Psi, \wick{\th_0 V\left(\frac{\ga \Psi}{\sqrt{\th_0}}\right)}\right)
  \end{align}
  under $\probPhi^{\om}$ for $\prob_0$-a.e.\ $\om$ in the strong topology of $H^{-s}(D) \times H^{-s'}(D)$ for any $s>0$ and any $s'>\ga^2 \be C_{\mathrm{HK}}$. Let $\cG \colon H^{-s}(D) \to \bbR$ be an arbitrary bounded and continuous function. Then, 
  \begin{align*}
    \cG(\Phi_n) e^{-\wick{V^{g}_n}(\Phi_n)}
    \,\rdef\,
    \tilde{\cG}(\Phi_n,\wick{V(\ga \Phi_n)})
  \end{align*}
  is a continuous function of $(\Phi_n,\wick{V(\ga \Phi_n)})$ on $H^{-s}(D) \times H^{-s'}(D)$. Moreover, since $\cG$ is bounded, $\{\tilde{\cG}(\Phi_n,\wick{V(\ga \Phi_n)})\}_n$ is uniformly integrable under $\probPhi^{\om}$ by \eqref{eq:ui:condition:self:interacting}. Hence, the weak convergence in \eqref{eq:weak:convergence:applied:self:interacting} combined with \eqref{eq:ui:condition:self:interacting} implies that 
  \begin{align*}
    \meanPhi^{\om} \left[
      \cG(\Phi_n) e^{-\wick{V^{g}_n}(\Phi_n)}
    \right]
    \,\xrightarrow{n \to \infty}\,
    \meanPhi^{\Si} \left[
      \cG(\sqrt{\th_0}\Psi) e^{-\wick{V^{g}}(\sqrt{\th_0}\Psi)}
    \right].
  \end{align*}
  Noting that by choosing $\cG \equiv 1$ in the above yields the convergence of the denominator, we get
  \begin{align*}
    \meanPhi_{\mu^{\om,V^{g}_n}_n} \left[
      \cG(\Phi_n)
    \right]
    &\,=\,
    \frac{\meanPhi^{\om} \left[
        \cG(\Phi_n) e^{-\wick{V^{g}_n}(\Phi_n)}
      \right]}{\meanPhi^{\om} \left[
        e^{-\wick{V^{g}_n}(\Phi_n)}
      \right]}
    \nonumber\\[.5ex]
    &\,\xrightarrow{n \to \infty}\,
    \frac{\meanPhi^{\Si} \left[
        \cG(\sqrt{\th_0}\Psi) e^{-\wick{V^{g}}(\sqrt{\th_0}\Psi)}
      \right]}{\meanPhi^{\Si} \left[
        e^{-\wick{V^{g}}(\sqrt{\th_0}\Psi)}
      \right]}
    \,=\,
    \meanPhi_{\mu^{\Si,V,\th_0} } \left[
      \cG(\sqrt{\th_0}\Psi) 
    \right],
  \end{align*}
  which concludes the claim.
\end{proof}
\begin{example}
  \begin{enumerate}[(i)]
  \item
    \textit{H{\o}egh-Krohn model/Exponential interaction:} Note that $\wick{e^X} = e^{X-\bbE[X^2]/2}\geq 0$ for any centered Gaussian $X$ (cf.~\cite[Theorem 3.33 ]{janson1997gaussian}). Hence, choosing $V \equiv \exp$, we have $\wick{V^{g}_n}(\vp^{nD}) \geq 0$ and thus 
    \begin{align*}
      e^{-\wick{V^{g}_n}(\vp^{nD})}
      \,\leq\,
      1.
    \end{align*}
    verifying \eqref{eq:ui:condition:self:interacting} (cf.~\cite[Theorem 4.2]{albeverio1974wightman}). 

  \item
    \textit{Polynomials of even degree:} Since odd degree polynomials are unbounded from below, \eqref{eq:ui:condition:self:interacting} is not satisfied for this class. However, even in the case of even-degree polynomials showing \eqref{eq:ui:condition:self:interacting} is not trivial. For instance, $\bbE[\wick{X^{2m}}]=0$ for any $m\in \bbN$ and $X$ Gaussian (cf.\ \eqref{eq:wick:covariance:gaussian:xy}). Hence, unlike in the previous case, the Wick-power $\wick{X^{2m}}$ is, of course, not almost surely non-negative. For the lattice approximation of the well known $P(\phi)_2$-model, however, one verifies \eqref{eq:ui:condition:self:interacting} by showing that for every polynomial $P$ of even degree $2m$, the measure of the set where $\scprfield{\wick{P(\Phi_n)}}{g}$ is smaller than $-b\log(n)^m$ for some $b$ and large enough $n$ is sufficiently small (cf.\ \cite[Theorem IV.1-(c)]{guerra1975euclidean} or \cite[Lemma 9.6.2]{glimm2012quantum}). 
    
  \item
    \textit{Sine-Gordon:} The Hamiltonian of the (quantum) sine-Gordon Model is given by choosing $V\equiv\cos$ in \eqref{eq:def:hamiltonian:V} (cf.~\cite{coleman1975quantum}). Here, one faces problems similar as those encountered with polynomials. Indeed, even though $|\cos| \leq 1$, we have $\wick{\cos(\ga \vp_{\lfloor nx \rfloor})} = e^{(\ga^2/2) g^{\om}_{nD}(\lfloor nx \rfloor, \lfloor nx \rfloor)} \cos(\ga \vp_{\lfloor nx \rfloor})$, which is not bounded uniformly in $n$, since $g^{\om}_{nD}(\lfloor nx \rfloor, \lfloor nx \rfloor)$ scales like $\log(n)$. For $\ga$ as in Example~\ref{examples:mainresult1}-(ii), one may be able to verify \eqref{eq:ui:condition:self:interacting} following similar arguments as in the polynomial case. We refer to \cite{frohlich1976classical} for further reference. 
  \end{enumerate}
\end{example}

\subsection{Main ideas and methods}
The proof of Theorem~\ref{thm:green:pointwise:upperbound} rests on the observation that the Green's kernel, although initially defined probabilistically in \eqref{eq:def:green:rcm:probabilistic} in terms of the VSRW, is independent of the speed measure of the underlying random walk, in the sense made precise in \eqref{eq:Greenkernel:CS:equals:VS}. We therefore choose the particular speed measure $\th^{\om} \equiv 1 \vee \mu^{\om}$ \eqref{eq:def:speed:measure}. This choice combines two practical features of both the VSRW, for which $\th^{\om} \equiv 1$ and CSRW, for which $\th^{\om} \equiv \mu^{\om}$. First, it ensures that the intrinsic metric appearing in the heat kernel bounds of \cite[Theorem 3.2]{andres2019heat} coincides with the graph metric. This yields the Gaussian bounds in Theorem~\ref{thm:heat:kernel:bounds} that resemble those for the CSRW in \cite[Theorem 1.6]{andres2016heat}. Secondly, the lower bound $\th^{\om}\geq 1$ allows us to control the heat kernel on time scales preceding the random time $N_{\mathrm{HK}}(\om)$ after which the heat kernel estimates become effective (cf.\ Lemma~\ref{lem:heatkernel:before:diffusive}). These ingredients are sufficient to prove Theorem~\ref{thm:green:pointwise:upperbound}, and hence Corollary~\ref{cor:green:pointwise:upperbound:corollary} in $d \geq 3$. In dimension $d= 2$, however, controlling the contribution from times larger than $n^{2}$ additionally requires the spectral representation of the heat kernel in \eqref{eq:heat:kernel:killed:Semigroup}, together with the weighted Sobolev inequality, in order to bound the principal Dirichlet eigenvalue in \eqref{eq:principal:eigenvalue:GSRW}.

Corollary~\ref{cor:green:pointwise:upperbound:corollary} is also the main novel ingredient in the proof of Proposition~\ref{prop:tested:green:convex:F}, which forms the technical backbone of the scaling limit and establishes the convergence of certain operators, defined in \eqref{eq:def:operator:H:g}, involving nonlinear functionals of the following three related kernels: the rescaled Green's kernel $g_{n}$, defined in \eqref{eq:def:rescaled:greens:function:d2:abbreviate} and estimated in Corollary~\ref{cor:green:pointwise:upperbound:corollary}, the smeared kernel $g^{\ve,\ve}_{n}$, obtained by smoothing both variables of $g_n$ with $\rh^{\ve}$, see \eqref{eq:def:smeared:green:kernels:rcm}, and the half-smeared kernel $g^{0,\ve}_n$, in which only one of the two entries of $g_n$ is smoothed by $\rh^{\ve}$ (see \eqref{eq:def:smeared:green:kernels:rcm}). The proof combines the local central limit theorem (LCLT) for the killed Green's kernel obtained in \cite{andres2025scaling}, recalled here as Theorem~\ref{thm:LCLT_green}, an extension of the ergodic theorem in Lemma~\ref{eq:ergodic:lemma:extend:LocHoelder}, and the uniform-integrability (UI) estimate on the diagonal in Proposition~\ref{prop:UI:assumption:analytic}. The latter is the crucial application of the bounds in Corollary~\ref{cor:green:pointwise:upperbound:corollary}. 

Since the objects we consider here are, in general, non-Gaussian, convergence of covariances alone is no longer sufficient, in contrast to the case of the free field itself, cf.~\cite[Theorem 4.1]{andres2025scaling}. Instead, we follow a strategy similar to that of \cite{deuschel2024ray} and introduce a regularised version of the DGFF akin to the smeared field $\Psi^{\ve}$ appearing in \eqref{def:cgff:analytic:limit:in:L2:introduction}. More precisely, setting $\Phi^{\ve}_n \ldef \{\scprfield{\Phi_n}{\rh^{\ve}_x} : x \in D\}$, an application of \cite[Theorem 4.1]{andres2025scaling}, or alternatively of Proposition~\ref{prop:tested:green:convex:F}, together with the  Kolmogorov-\v{C}entsov theorem, yields convergence of $\Phi^{\ve}_n$ to $\sqrt{\th_0}\Psi^{\ve}$, as \mbox{$n \to \infty$} for fixed $\ve > 0$, in the space of continuous functions. Using regularity properties of Hermite polynomials and the continuous mapping theorem, we then prove in Proposition~\ref{prop:convergence:analyticfunction:smooth:approximation} that 
\begin{align}\label{eq:smooth:approximation:convergence:introduction}
  \scprfield{\wick{\th_0 F\left(\frac{\ga \Phi^{\ve}_n}{\th_0}\right)}}{f}
  &\underset{n \to \infty}{\overset{\text{law}}{\;\longrightarrow\;}}
  \scprfield{\wick{\th_0 F\left(\frac{\ga \Psi^{\ve}}{\sqrt{\th_0}}\right)}}{f}
  \qquad 
  (d \geq 2)
\end{align}
for every fixed $\ve > 0$ and every bounded test function $f$. The main remaining technical challenge is to remove the ultraviolet cut-off on the discrete field in a suitable sense. More precisely, in Proposition~\ref{prop:L2:convergence:dgff:smeared} we show that 
\begin{align}\label{eq:L2:smeared:approximation:introduction}
  \lim_{\ve \downarrow 0} \lim_{n \to \infty} \Norm{\scprfield{\wick{F(\ga \Phi_n)}}{f} - \scprfield{\wick{ \th_0 F\left(\frac{\ga \Phi^{\ve}_n}{\th_0}\right)}}{f}}{L^2(\probPhi^{\om})}
  \,=\,
  0
  \qquad
  (d = 2).
\end{align}
After expanding the square in the above expression, one obtains three terms involving the kernels $g_{n}$, $g^{\ve,\ve}_{n}$, and $g^{0,\ve}_{n}$, whose convergence is established in Proposition~\ref{prop:tested:green:convex:F}. Thus, \eqref{eq:L2:smeared:approximation:introduction} can be viewed as the discrete analog of \eqref{def:cgff:analytic:limit:in:L2:introduction}, the continuum counterpart of which will be proved shortly afterwards in Proposition~\ref{prop:L2:convergence:cgff:smeared}. Recalling that $L^2$-convergence implies weak convergence, we readily show marginal convergence of the field tested against suitable $f$ in Theorem~\ref{thm:marginal:convergence:analytic:field}. Finally, tightness is established in Proposition~\ref{prop:DGFF:tightness}. Compared with the $s>d/2$ tightness result of \cite[Proposition~4.2]{andres2025scaling}, the improvement obtained here comes from using the fractional Green's kernel \eqref{eq:definition:G_s:fractional:kernel}, rather than estimating the inverse eigenvalues in the $H^{-s}(D)$ norm via Weyl's law, cf.~\cite[Proposition~4.2, Step~1]{andres2025scaling}. The crucial ingredient for obtaining sharp results in the interesting regime $s \leq d/2$ is, once again, the Green's kernel bounds from Corollary~\ref{cor:green:pointwise:upperbound:corollary} 
\medskip

We conclude the introduction by describing the organisation of this article. Section~\ref{sec:greens:function} is the main technical part of the paper and is devoted entirely to the killed Green's kernel. In Subsection~\ref{subsec:greens:function:near:diagonal} we prove Theorem~\ref{thm:green:pointwise:upperbound} and Corollary~\ref{cor:green:pointwise:upperbound:corollary}. These results are then combined with the LCLT and an extension of the ergodic theorem in Subsection~\ref{subsec:greens:function:limit:tested} to prove Proposition~\ref{prop:tested:green:convex:F}. In Section~\ref{sec:scaling:limit}, we first show the aforementioned convergence of the smooth approximations (cf.\ \eqref{eq:smooth:approximation:convergence:introduction}) in Subsection~\ref{subsec:scaling:limit:smooth:approximation}. We then establish the corresponding $L^{2}$-convergence statements, namely \eqref{eq:L2:smeared:approximation:introduction} and \eqref{def:cgff:analytic:limit:in:L2:introduction}, in Subsection~\ref{subsec:scaling:limit:L2:convergence}. Based on this, marginal convergence is shown in Subsection~\ref{subsec:scaling:limit:marginal:convergence}, and tightness is proved in Subsection~\ref{subsec:scaling:limit:tightness}. Appendix~\ref{sec:appendix:ergodic:theorem} contains the proof of the extended ergodic theorem stated in Lemma~\ref{lem:ergodic:lemma:extend:LocHoelder}, while Appendix~\ref{sec:appendix:analytical:results} collects the proofs of several auxiliary analytic results. Throughout the paper, $c$ denotes a positive constant whose value may change from line to line, whereas  constants denoted by $C_i$ are fixed throughout. 

\section{Green's function}\label{sec:greens:function}
\subsection{Near diagonal pointwise bounds on the RCM Green's kernel}\label{subsec:greens:function:near:diagonal}
In this subsection we will give the proofs of Theorem~\ref{thm:green:pointwise:upperbound} and Corollary~\ref{cor:green:pointwise:upperbound:corollary}. For this purpose, we introduce the random walk $Y$ with speed measure $\th^{\om} \colon \cC_{\infty}(\om) \to (0,\infty)$, which we define as a time change of the VSRW $X$. Namely, let $a_t \ldef \inf \{ s\geq 0 \,:\, A_s >t \}$ be the right continuous inverse of the functional 
\begin{align*}
  A_t 
  \,=\,
  \int_0^t \th^{\om}(X_s) \,\md s, 
  \qquad
  t \geq 0. 
\end{align*}
This walk is the process $Y = (Y_t \,:\, t\geq 0) = (X_{a_t}\,:\, t\geq 0)$, which is generated by $\cL^{\om}_{\th^{\om}} f(x) \,=\, \th^{\om}(x)^{-1} \sum_{y \in \cC_{\infty}(\om)} \om(x,y) \left(f (y) - f(x)\right)$ (cf.~\cite{andres2019heat}). Natural choices are $\th^{\om} = \mu^{\om}$, corresponding to the constant speed random walk (CSRW) and $\th^{\om} = 1$, corresponding to the VSRW itself. For our purposes, the following speed measure turned out to be an efficient choice to obtain upper bounds
\begin{align}\label{eq:def:speed:measure}
  \th^{\om}(x)
  \,\ldef\,
  \mu^{\om}(x) \vee 1.
\end{align}
The walk corresponding to this speed measure inherits useful properties of the VSRW while always moving at least as fast as the CSRW, which moves according to $\mathrm{Exp}(1)$ clocks independent of its current position. For a set $\La \subset \cC_{\infty}(\om)$, we define the heat kernel and the killed heat kernel associated to this walk as
\begin{align*}
  q^{\om}_t(x,y)
  \,\ldef\,
  \frac{\Prob^{\om}_x [Y_t=y]}{1 \vee \mu^{\om}(y)},
  \qquad
  \text{and}
  \qquad
  q^{\om,\La}_t(x,y)
  \,\ldef\,
  \frac{\Prob^{\om}_x [Y_t=y, \ta_{\La}(Y) \geq t]}{1 \vee \mu^{\om}(y)}.
\end{align*}
The normalisation ensures, by detailed balance, symmetry of the heat kernel and also of the corresponding killed Green's kernel on a set $\La \subset \cC_{\infty}(\om)$, which, similarly to \eqref{eq:def:green:rcm:probabilistic}, is defined as 
\begin{align*}
  g^{\om,\th^{\om}}_{\La}(x,y) 
  \,\ldef\,
  \Mean^{\om}_x \left[ \int_0^{\ta_{\La}(Y)} \frac{\indicator_{\{Y_t = y\}}}{\th^{\om}(y)}  \,\md t\right]
  \,=\,
  \int_0^\infty \frac{\Prob^{\om}_x [Y_t=y, t<\ta_{\La}(Y)]}{\th^{\om}(y)} \,\md t
\end{align*}
for any $x,y \in \cC_{\infty}(\om)$. Using that $Y = X \circ a$ and substituting $s = a_t$ with $\md s = (1/\th^{\om}(X_s)) \md t$ in the integral inside of the expecation, it actually turns out that this Definition is independent of the choice of speed measure $\th^{\om}$, i.e.\ 
\begin{align}\label{eq:Greenkernel:CS:equals:VS}
  g^{\om,\th^{\om}}_{\La}(x,y) 
  \,=\,
  \Mean^{\om}_x \left[ \int_0^{\ta_{\La}(X_s)} \frac{\indicator_{\{X_s = y\}}}{\th^{\om}(y)}  \,\th^{\om}(X_s) \md s\right]
  \,=\,
  g^{\om}_{\La}(x,y),
\end{align}
where, as expected, $g^{\om}_{\La} \equiv g^{\om,1}_{\La}$ agrees with the Definition in \eqref{eq:def:green:rcm:probabilistic}. We are thus free to choose a speed measure of our liking and settle for \eqref{eq:def:speed:measure}.

To proceed, we introduce the weighted scalar product with respect to the measure $\th^{\om}$ as
\begin{align*}
  \scpr{f}{g}{\cC_{\infty}(\om),\th^{\om}} 
  \,\ldef\, 
  \sum_{x \in \cC_{\infty}(\om)} f(x) g(x) \th^{\om}(x)
\end{align*}
for which one readily verifies that $\scpr{-\cL^{\om}_{\th} f}{g}{\cC_{\infty}(\om),\th^{\om}} = \scpr{f}{-\cL^{\om}_{\th} g}{\cC_{\infty}(\om),\th^{\om}}$. Hence, by virtue of the Perron-Frobenius Theorem (cf.~\cite[Chapter 1]{seneta2006non}), there exists a basis $\{\psi^{\om,\th}_k\}_{k=1}^{|B^{\om}(n)|}$ of the Hilbert space $\bbR^{|B^{\om}(n)|}$ which is orthonormal with respect to the weighted scalar product $\scpr{\cdot}{\cdot}{\cC_{\infty}(\om),\th^{\om}}$ and such that there are $0<\la^{\om,\th}_{1} < \la^{\om,\th}_{2} \leq \dots \leq \la^{\om,\th}_{|B^{\om}(n)|}$ satisfying 
\begin{equation} \label{eq:rcm:laplace:eigenvalues}
  \begin{cases}
    -\mathcal{L}^{\omega}_{\theta} \psi^{\omega,\theta}_k = \lambda^{\omega,\theta}_{k} \psi^{\omega,\theta}_k & \text{on } B^{\omega}(n), \\[.5ex]
    \psi^{\omega,\theta}_k = 0 & \text{on } (B^{\omega}(n))^{\mathrm{c}},
  \end{cases}
  \qquad \text{and} \qquad 
  \scpr{\psi^{\om,\th}_i}{\psi^{\om,\th}_j}{\cC_{\infty}(\om),\th^{\om}} = \de_{ij}.
\end{equation}
The following Sobolev inequality from \cite{andres2025scaling} will be sufficient for us to control the principal Dirichlet Eigenvalue $\lambda^{\omega,\theta}_{1}$ on large balls. Before we state it, let us quickly recall some standard notation frequently appearing in the analytical study of the RCM. For any $p \in [1, \infty]$ and any non-empty, finite $A \subset \mathcal{C}_{\infty}(\omega)$, we define space-averaged $\ell^{p}$-norm on functions $f\colon A \to \mathbb{R}$ by 
\begin{align*}
  \Norm{f}{p, A}
  \;\ldef\;
  \bigg(
  \frac{1}{|A|}\; \sum_{x \in A}\, |f(x)|^p\
  \bigg)^{\!\!1/p}
  \qquad \text{and} \qquad
  \Norm{f}{\infty, A} \;\ldef\; \max_{x \in A} |f(x)|,
\end{align*}
where $|A|$ denotes the cardinality of the set $A$. We define two discrete measures $\mu^{\omega}$ and $\nu^{\omega}$ on $\mathbb{Z}^{d}$ by
\begin{align*}
  \mu^{\omega}(x) \;\ldef\; \sum_{y \sim x} \omega(\{x, y\})
  \qquad \text{and} \qquad
  \nu^{\omega}(x) \;\ldef\; \sum_{y \sim x} \frac{1}{\omega(\{x, y\})} \indicator_{\{\{x,y\}\in \mathcal{O}(\omega)\}}.
\end{align*}
Moreover, let us recall that we defined $\cE^{\om}(f) \ldef \scpr{f}{-\cL^{\om}f}{\cC_{\infty}(\om)}$ in \eqref{eq:def:rcm:laplace:energy} for $f \colon \cC_{\infty}(\om) \to \bbR$ with finite support. 
\begin{prop}[Weighted Sobolev inequality] \label{prop:Sobolev:ineq}
  Let $d \geq 2$, $q \in [1, \infty]$, and suppose there exists $\theta \in (0, 1)$ such that Assumptions~\ref{ass:law} and~\ref{ass:cluster}-(i) hold. Then, for any $\theta' \in (\theta, 1)$ there exists $C_{\mathrm{S}} \equiv C_{\mathrm{S}}(\theta'\!, d, q) \in (0, \infty)$ such that for any $\omega \in \Omega_{\mathrm{reg}} \cap \Omega_{0}^{*}$, $R \geq 1$ and $n \geq \const[N]{Nconst:ass:cluster}(\omega) \vee R^{\theta / (\theta' - \theta)}$ the following hold. For any $x_{0} \in B^{\omega}(0, R n)$ and any $f\colon \mathcal{C}_{\infty}(\omega) \to \mathbb{R}$ with $\supp(f) \subset B^{\omega}(x_{0}, n)$, 
  \begin{align}\label{eq:Sobolev:weighted}
    \Norm{f^2}{\rho, B^{\omega}(x_{0}, n)}
    \;\leq\;
    C_{\mathrm{S}}\, n^{2}\, \Norm{\nu^{\omega}}{q, B^{\omega}(x_{0}, n)}\,
    \frac{\mathcal{E}^{\omega}(f)}{|B^{\omega}(x_{0}, n)|},
  \end{align}
  where
  \begin{align}\label{eq:def:rho}
    \rho
    \;\equiv\;
    \rho(d', q)
    \;=\;
    \frac{d'}{d' - 2 + d'/q} \quad \text{with} \quad d' \ldef (d-\theta')/(1-\theta').
  \end{align}
\end{prop}
Note that, using Hoelders inequality and the weighted Sobolev inequality in Proposition~\ref{prop:Sobolev:ineq}, yields for any function $f$ with $\supp(f)\subset B^{\om}(n)$ and any $n \geq \const[N]{Nconst:ass:cluster}(\omega) \vee R^{\theta / (\theta' - \theta)}$ 
\begin{align*}
  \frac{\scpr{f}{f}{\cC_{\infty}(\om),\th^{\om}}}{|B^{\om}(n)|}
  &\,\leq\,
  \Norm{f^2}{\rh,B^{\om}(n)} \Norm{\th^{\om}}{\rh_*,B^{\om}(n)}
  \nonumber\\[.5ex]
  &\,\leq\, 
  C_{\mathrm{S}} \Norm{\nu^{\om}}{q,B^{\om}(n)}\Norm{1 \vee \mu^{\om}}{p,B^{\om}(n)} n^2 \frac{\cE^{\om}(f)}{|B^{\om}(n)|},
\end{align*}
where we also used that by the $p,q$ condition \eqref{eq:cond:pq_rg}, $\rh_*=\rh/(\rh-1) < p$ in the last line and our our choice of $\th^{\om} \equiv 1 \vee \mu^{\om}$. Hence, rearranging and using the Courant-Fischer theorem (cf.\cite[Theorem 1.2.10]{gine2006lectures}), we get a lower bound on the principal eigenvalue 
\begin{align}\label{eq:principal:eigenvalue:GSRW}
  \la^{\om,\th}_{1}
  \,=\,
  \inf_{\substack{f : \supp f \subset B^{\om}(n) \\ f \not= 0}} \, 
  \frac{\cE^{\om}(f)}{\scpr{f}{f}{\cC_{\infty}(\om),\th^{\om}}}
  \,\geq\,
  \left(C_{\mathrm{S}} \Norm{\nu^{\om}}{q,B^{\om}(n)}\Norm{1 \vee \mu^{\om}}{p,B^{\om}(n)}\right)^{-1} n^{-2}.
\end{align}
Furthermore, by representing the killed heat kernel $q^{\om,B^{\om}(n)}_t(x,y)$ in the eigenbasis one get the following useful representation (cf.~\cite[Lemma 1.3.3-(2)]{gine2006lectures} or \cite[Corollary 5.8]{Ba17})
\begin{align}\label{eq:heat:kernel:killed:Semigroup}
  q^{\om,B^{\om}(n)}_t(x,y)
  \,=\,
  \sum_{k=1}^{|B^{\om}(n)|} e^{-\la^{\om,\th}_{k} t} \psi^{\om,\th}_k(x) \psi^{\om,\th}_k(y).
\end{align}
Moreover, the following generalisation of the ergodic theorem will help us to control ergodic averages on scaled balls with varying centre points.
\begin{prop}\label{prop:krengel_pyke}
  Let $R > 0$ and  $\mathcal{B} \ldef \bigl\{ B : B \text{ open Euclidean ball in } [-R, R]^{d}\bigr\}$.  Suppose that Assumption~\ref{ass:law}-(i) holds.  Then, for any $f \in L^{1}(\Omega)$,
  \begin{align*}
    \lim_{n \to \infty} \sup_{B \in \mathcal{B}} \,
    \biggl|
      \frac{1}{n^{d}}\, \sum_{x \in (nB) \cap \mathbb{Z}^{d}}\mspace{-12mu} f \circ \tau_{x}
      \,-\,
      |B| \cdot \mean\bigl[ f \bigr]
    \biggr|
    \;=\;
    0, 
    \qquad \prob\text{-a.s.},
  \end{align*}
  where $|B|$ denotes the Lebesgue measure of $B$.
\end{prop}
\begin{proof}
  See, for instance, \cite[Theorem~1]{krengel1987uniform}.
\end{proof}
\begin{cor}[\cite{andres2025scaling}] \label{cor:KrengelPyke}
  Let $d \geq 2$ and $r > 0$. Suppose there exist $\theta \in (0, 1)$ and $p, q \in [1, \infty]$ such that Assumptions~\ref{ass:law}, \ref{ass:cluster}-(i) and~\ref{ass:pq} hold. Then, there exist $\Omega_{c} \in \mathcal{F}$ with $\prob[\Omega_{c}] = 1$ and $\bar{\mu}, \bar{\nu} \in (0, \infty)$ such that for any  $\delta \in (0,1)$ the following holds. There exists a random variable $\const[N]{Nconst:ergodic:constants} \equiv\const[N]{Nconst:ergodic:constants} (\omega, \delta)$ such that, for all $\omega \in \Omega_{\mathrm{reg}} \cap \Omega_{0}^{*} \cap \Omega_{c}$, $x \in [-r, r]^{d}$,
  \begin{align*}
    \sup_{n \geq \const[N]{Nconst:ergodic:constants} } \Norm{1 \vee \mu^{\omega}}{p, B^{\omega}(\pi_{n}(x), \delta n)}
    \;\leq\;
    \bar{\mu}
    \quad \text{and} \quad
    \sup_{n \geq \const[N]{Nconst:ergodic:constants} } \Norm{1 \vee \nu^{\omega}}{q, B^{\omega}(\pi_{n}(x), \delta n)}
    \;\leq\;
    \bar{\nu}.
  \end{align*}
\end{cor}
With Proposition~\ref{prop:Sobolev:ineq} and Corollary~\ref{cor:KrengelPyke} at hand, we can reformulate the result from \cite{andres2016heat}, tailored to our setting:
\begin{theorem}[\cite{andres2019heat}]\label{thm:heat:kernel:bounds}
  Let $d \geq 2$ and suppose there exist $\theta \in (0, 1)$ and $p, q \in [1, \infty]$ such that Assumptions~\ref{ass:law}, \ref{ass:cluster}-(i) and~\ref{ass:pq} hold. Then, there exists $\ka \equiv \ka(\th,d,p,q) \in (0,\infty)$ and constants $C_{\mathrm{HK}}, c_{1},c_{2},c_{3},c_{4}$ only depending on $d,p,q,\th$ such that, setting $\const[N]{Nconst:ass:green:bound}(\om,R) \ldef \const[N]{Nconst:ass:cluster}(\omega) \vee \const[N]{Nconst:ergodic:constants}(\om,1)\vee R^\ka$, the following holds. For any $\omega \in \Omega_{\mathrm{reg}} \cap \Omega_{0}^{*} \cap \Omega_{c}$, any $R \geq 1$, any $x,y \in B^{\om}(0,Rn)$ and any $t \geq N_{\mathrm{HK}}(\om) \equiv N_{\mathrm{HK}}(\om,R) \ldef 4 \const[N]{Nconst:ass:green:bound}(\om,R)^2 \vee 1$,
  \begin{equation}\label{eq:HKbound:Gaussian}
    q^{\omega}_t(x,y) 
    \leq C_{\mathrm{HK}} t^{-d/2}
    \begin{cases}
      \exp\left( -c_{1} \frac{d^{\omega}(x,y)^2}{t} \right), & \text{if } t \geq c_{2} d^{\omega}(x,y), \\[.75ex]
      \exp\left( -c_{3} d^{\omega}(x,y) \left( 1 \vee \log \frac{d^{\omega}(x,y)}{t} \right) \right), & \text{if } t \leq c_{4} d^{\omega}(x,y).
    \end{cases}
  \end{equation}
\end{theorem}
\begin{proof}
  This follows from the heat kernel bounds in \cite[Theorem 3.2]{andres2019heat}, applied on the graph $V= B^{\om}(0,Rn)$ with speed measure $\th^{\om} = 1 \vee \mu^{\om}$ in the following way: Note that, since, estimating $\th^{\om} \geq \mu^{\om}$ yields $1 \vee \mu^{\om}/\th^{\om} \leq 1$ and estimating $\th^{\om} \geq 1$ yields $1 \vee 1/\th^{\om} \leq 1$, Corollary \ref{cor:KrengelPyke} implies that for any $x \in B^{\om}(0,Rn)$ and any $n \geq \const[N]{Nconst:ergodic:constants}(\om,1)$, \cite[Assumption 3.1]{andres2019heat} is satisfied under the $p-q$-condition in \eqref{eq:cond:pq_rg} (in the the notation used in \cite[Equation (3.1)]{andres2019heat}, this follows from choosing $p=\infty$, $r=p$, $q=q$, $d'=2(1-\th)/(d-\th)$). Since $B^{\om}(0,Rn)$ is a subgraph of $\bbZ^d$, each vertex has at most degree $2d$, hence \cite[Assumption 2.1-(i)]{andres2019heat} is satisfied. Moreover, by Assumption \ref{ass:cluster}(i) and following the argument laid out in the beginning of the proof of \cite[Proposition 2.1]{andres2025scaling}, the volume regularity in \cite[Assumption 2.1(ii)]{andres2019heat} is satisfied for any $x \in B^{\om}(0,Rn)$ and any $n \geq \const[N]{Nconst:ass:cluster}(\omega) \vee R^{\theta / (\theta' - \theta)}$. Finally, the local Sobolev inequality in \cite[Assumption 2.1(iii)]{andres2019heat} follows from equation (2.5) in the proof of \cite[Proposition 2.1]{andres2025scaling} for any $x \in B^{\om}(0,Rn)$ and any $n \geq \const[N]{Nconst:ass:cluster}(\omega) \vee R^{\theta / (\theta' - \theta)}$. Hence, setting $\ka \ldef \theta / (\theta' - \theta)$ one can choose 
  \begin{align*}
    N_{\mathrm{HK}}(\om)
    \,\ldef\,
    \left(2 (\const[N]{Nconst:ass:cluster}(\omega) \vee R^{\ka}) \vee \const[N]{Nconst:ergodic:constants}(\om,1)\right)^2 \vee 1
    \quad 
    \left( \equiv 4 \const[N]{Nconst:ass:green:bound}(\om,R)^2 \vee 1\right)
  \end{align*}
  in \cite[Theorem 3.2]{andres2019heat}, which is finite for $\prob_{0}$-a.e.\ $\om$.

  Further, since $\th^{\om} = 1 \vee \mu^{\om} \geq \mu^{\om}$, one has for any $x,y \in B^{\om}(0,Rn)$
  \begin{align*}
    1 \wedge \frac{\th^{\om}(x) \wedge \th^{\om}(y)}{\om(x,y)}
    \,\geq\,
    1 \wedge \frac{\mu^{\om}(x) \wedge \mu^{\om}(y)}{\om(x,y)}
    \,\geq\,
    1 \wedge \frac{\om(x,y)}{\om(x,y)}
    \,=\,
    1,
  \end{align*}
  hence, the intrinsic metric $d^{\om}_{\th^{\om}}$ defined in \cite[Equation (1.2)]{andres2019heat} is lower bounded by the graph distance $d^{\om}$. Since the upper bound $d^{\om}_{\th^{\om}}(x,y) \leq d^{\om}(x,y)$ also holds (for any choice of $\th$), we have $d^{\om}_{\th^{\om}} = d^{\om}$ on $V = B^{\om}(0,Rn)$. Finally, if $d^{\om}(x,y) \geq 1$ and $t \geq N_{\mathrm{HK}}(\om) \geq 1$
  \begin{align*}
    (1 + d^{\om}(x,y)/\sqrt{t})^{\ga} e^{-c d^{\om}(x,y)} \leq 2^{\ga} e^{-\tilde{c} d^{\om}(x,y)},
  \end{align*}
  which yields the claim in its final form. 
\end{proof}
We start with the following lemma to control the contribution of the heat kernel until the regime where the bounds in Theorem~\ref{thm:heat:kernel:bounds} become Gaussian:
\begin{lemma}\label{lem:heatkernel:before:diffusive}
  Let $d \geq 2$ and suppose there exist $\theta \in (0, 1)$ and $p, q \in [1, \infty]$ such that Assumptions~\ref{ass:law}, \ref{ass:cluster}-(i) and~\ref{ass:pq} hold. Then, for any $\omega \in \Omega_{\mathrm{reg}} \cap \Omega_{0}^{*} \cap \Omega_{c}$, any $R \geq 1$, there exist constants $c,c' >0$ such that for any $x,y \in B^{\om}(0,Rn)$ 
  \begin{align*}
    \int_{0}^{ c_{2} d^{\om}(x,y) \vee N_{\mathrm{HK}}(\om)} q^{\om,B^{\om}(n)}_{t}(x,y) \, \md t
    \,\leq\,
    c' N_{\mathrm{HK}}(\om) e^{-c \frac{(d^{\om}(x,y)\vee 1)}{N_{\mathrm{HK}}(\om)}}.
  \end{align*} 
\end{lemma}
\begin{proof}
  We first establish the following Lemma which is a consequence of the Carne Varopoulus bounds for discrete time reversible Markov Chains (cf.\cite{varopoulos1985long}). However, since we could not find a reference where this is shown for continuous time random walks with general speed measures, we quickly adapt the proof given in \cite[Theorem 5.17]{Ba17}, where this is shown for the CSRW.
  \begin{lemma}[\cite{Ba17}]
    Let $Y$ be a random walk with speed measure $\th^{\om}$ on $\cC_{\infty}(\om)$ and assume that $\th^{\om} \geq \mu^{\om}$. Then, there exists $c>0$ such that for any $x,y \in \cC_{\infty}$ such that $d^{\om}(x,y) \geq 1$
    \begin{equation*}
      \frac{\prob_x \left[Y_t =y \right]}{\th^{\om}(y)}
      \leq 
      \begin{cases}
        \frac{4}{\sqrt{\th^{\om}(x)\th^{\om}(y)}} \exp\left( - \frac{d^{\omega}(x,y)^2}{e^2t} \right), & \text{if } et \geq d^{\omega}(x,y), \\[.75ex]
        \frac{1}{\th^{\om}(x) \vee \th^{\om}(y)} \exp\left( -c d^{\omega}(x,y) \left( 1 \vee \log \frac{d^{\omega}(x,y)}{t} \right) \right), & \text{if } et \leq d^{\omega}(x,y).
      \end{cases}
    \end{equation*}
  \end{lemma}
  \begin{proof}
    Note that since $(\cL^{\om}_{\th^{\om}} \de_y)(x) \,=\, \th^{\om}(x)^{-1} \sum_{z \sim x} \om(x,z) (\de_y(z) - \de_y(x)) \,=\, \om(x,y)/\th^{\om}(x) \indicator_{y \sim x} - \mu^{\om}(x)/\th^{\om}(x) \indicator_{x=y}$, the transition matrix $P^{\th^{\om}} = \id + \cL^{\om}_{\th^{\om}}$ has entries given by
    \begin{align*}
      P^{\th^{\om}}(x,y) \,=\, \frac{\om(x,y)}{\th^{\om}(x)}, \; \text{if} \; x \sim y 
      \quad
      \text{and}
      \quad
      P^{\th^{\om}}(x,x) \,=\, \frac{\mu^{\om}(x)}{\th^{\om}(x)},
    \end{align*}
    which, since we assumed that $\th^{\om} \geq \mu^{\om}$, is a well defined transition matrix of a lazy random walk on $\cC_{\infty}(\om)$. Moreover, $P^{\th^{\om}}$ is reversible with respect to $\th^{\om}$. Hence, denoting $P^{\th^{\om}}_n = (P^{\th^{\om}})^n$, one easiliy verifies that $P^{\th^{\om}}_t(x,y) \ldef \sum_{n \geq 0} \frac{e^{-t}t^n}{n!} P^{\th^{\om}}_n(x,y)$ satisfies $P^{\th^{\om}}_0 = \id$ and the Kolmogorov backward equation, hence $\prob_x \left[Y_t =y \right] = P^{\th^{\om}}_t(x,y)$. Setting $a_t(n) = e^{-t} t^n/n!$ and inserting the Carne Varopoulus on the discrete walk (cf.~\cite[Theorem 4.9]{Ba17}) yields
    \begin{align*}
      \frac{\prob_x \left[Y_t =y \right]}{\th^{\om}(y)}
      \,=\,
      \frac{1}{\th^{\om}(y)} \sum_{n \geq 0} a_t(n) P^{\th^{\om}}_n(x,y)
      \,\leq\,
      \frac{2}{\sqrt{\th^{\om}(x)\th^{\om}(y)}} \sum_{n \geq 0} a_t(n) e^{d^{\om}(x,y)^2/(2n)}
    \end{align*}
    and repeating the argument in the proof of \cite[Theorem 5.17]{Ba17} line by line to estimate the sum concludes the case where $d^{\om}(x,y) \leq et$.

    If $d^{\om}(x,y) \geq et$, one uses that $P^{\th^{\om}}_n(x,y) \leq \indicator_{n \geq d^{\om}(x,y)}$, which implies that $\prob_x \left[Y_t =y \right] \leq \prob[N_t \geq d^{\om}(x,y)]$, where $N_t \sim \mathrm{Poisson}(t)$. Hence using \cite[Lemma 5.13(b)]{Ba17} and repeating the argument in \cite[Theorem 5.17]{Ba17}, the claim follows as laid out in \cite[Remark 5.18 (2)]{Ba17}. 
  \end{proof}
  Since $\th^{\om} = \mu^{\om} \vee 1 \geq 1$, the preceding Lemma implies that for any $d^{\om}(x,y) \geq 1$ and $t \geq d^{\om}(x,y)/e$, that $q^{\om}_t(x,y) \leq e^{-c 2 d^{\om}(x,y)}$. Moreover, since $q^{\om}_t(x,x) \leq 1/\th^{\om} \leq 1 = e^{c} e^{-c (d^{\om}(x,x) \vee 1)}$, one has $q^{\om}_t(x,y) \leq (1 \vee e^{c}) e^{-c (d^{\om}(x,y) \vee 1)}$ for any $t \leq (d^{\om}(x,y)\vee 1)/e$ and any $x,y \in \cC_{\infty}(\om)$. Hence, estimating $N_{\mathrm{HK}}(\om) \vee ((d^{\om}(x,y)\vee 1)e^{-1}) \leq N_{\mathrm{HK}}(\om) (d^{\om}(x,y)\vee 1)$ yields 
  \begin{align*}
    \int_{0}^{N_{\mathrm{HK}}(\om) \vee \frac{(d^{\om}(x,y)\vee 1)}{e}} q^{\om}_t(x,y) \, \md t
    &\,\leq\,
    (1 \vee e^{c}) N_{\mathrm{HK}}(\om) (d^{\om}(x,y)\vee 1) e^{-c (d^{\om}(x,y) \vee 1)}
    \nonumber\\[.5ex]
    &\,\leq\,
    c'  N_{\mathrm{HK}}(\om) e^{- \tilde{c} (d^{\om}(x,y) \vee 1)}
    \\[.5ex]
    &\,\leq\,
    c' N_{\mathrm{HK}}(\om) e^{- \tilde{c} \frac{(d^{\om}(x,y) \vee 1)}{N_{\mathrm{HK}}(\om)}},
  \end{align*}
  where the second step followed by absorbing the $(d^{\om}(x,y)\vee 1)$ factor and the last since $N_{\mathrm{HK}}(\om) \geq 1$.
  
  Note that if $t\geq 1/e = (d^{\om}(x,x)\vee 1)/e$, we have $q^{\om}_t(x,x) \leq 1= e^{1/e} e^{-1/e} \leq e^{1/e} e^{-(d^{\om}(x,x) \vee 1)^2/(e^{2}t)}$. Since, $e^{1/e}\leq 4$, the previous Lemma now implies that $q^{\om}_t(x,y) \leq 4 e^{-(d^{\om}(x,y) \vee 1)^2/(e^{2}t)}$ for any $t \geq (d^{\om}(x,x)\vee 1)/e$ and any $x,y \in \cC_{\infty}(\om)$, hence,
  \begin{align*}
    \int_{N_{\mathrm{HK}}(\om) \vee \frac{(d^{\om}(x,y)\vee 1)}{e}}^{N_{\mathrm{HK}}(\om) \vee c_{2} (d^{\om}(x,y)\vee 1)} q^{\om}_t(x,y) \, \md t
    \,\leq\,
    \tilde{c} \left(N_{\mathrm{HK}}(\om)(d^{\om}(x,y)\vee 1)\right)e^{-\frac{(d^{\om}(x,y)\vee 1)}{e^2 c_{2} N_{\mathrm{HK}}(\om)}},
  \end{align*}
  where we estimated $N_{\mathrm{HK}}(\om) \vee c_{2} (d^{\om}(x,y)\vee 1) \leq c_{2} N_{\mathrm{HK}}(\om)(d^{\om}(x,y)\vee 1)$ in the prefactor and in the exponential. Hence, absorbing the $(d^{\om}(x,y)\vee 1)$ term into the exponential again, yields the claim.
\end{proof}
We now have everything to show upper bounds on the Green's function. Recall that the killed Green's functions of the GSRW and the VSRW coincide \eqref{eq:Greenkernel:CS:equals:VS}.
\begin{proof}[Proof of Theorem~\ref{thm:green:pointwise:upperbound}]
  We will start with the much easier
  
  \smallskip
  \textbf{Case 1:} $ d \geq 3$.
  Note that for any $a \geq 0$ and $b >0$, there exists $c=c(b)$ such that $e^{-a} \leq c (a \vee 1)^{-b}$. Applying this with $b = d-2$ and combining with Lemma~\ref{lem:heatkernel:before:diffusive} yields
  \begin{align*}
    \int_{0}^{N_{\mathrm{HK}}(\om) \vee c_{2} (d^{\om}(x,y)\vee 1)}q^{\om}_t(x,y) \, \md t
    \,\leq\,
    c N_{\mathrm{HK}}(\om)^{d-1} (d^{\om}(x,y)\vee 1)^{2-d}.
  \end{align*}
  First assume that $d^{\om}(x,y) \geq 1$. Then, \eqref{eq:HKbound:Gaussian} in Theorem \ref{thm:heat:kernel:bounds} implies that 
  \begin{align*}
    \int_{N_{\mathrm{HK}}(\om) \vee c_{2} (d^{\om}(x,y)\vee 1)}^{\infty} q^{\om}_t(x,y) \, \md t
    \,\leq\,
    C_{\mathrm{HK}} \int_{1}^{\infty} t^{-d/2} e^{-c_{1} \frac{d^{\om}(x,y)^2}{t}} \, \md t
    \nonumber\\[.5ex]
    \,\leq\,
    c d^{\om}(x,y)^{2-d}
    \,\leq\,
    c N_{\mathrm{HK}}(\om)^{d-1} (d^{\om}(x,y)\vee 1)^{2-d},
  \end{align*}
  where the second inequality is a standard calculation using the Gamma funciton and the last inequality follows since $N_{\mathrm{HK}}(\om) \geq 1$ and since $(d^{\om}(x,y)\vee 1)= d^{\om}(x,y)$. If $x=y$, again by Theorem \ref{thm:heat:kernel:bounds}, 
  \begin{align*}
    \int_{N_{\mathrm{HK}}(\om) \vee c_{2} (d^{\om}(x,x)\vee 1)}^{\infty} q^{\om}_t(x,x) \, \md t
    \,\leq\,
    C_{\mathrm{HK}} \int_{1}^{\infty} t^{-d/2} \, \md t
    \,=\,
    \frac{C_{\mathrm{HK}}}{d/2 -1},
  \end{align*}
  which yields the claim for $d \geq 3$, since $N_{\mathrm{HK}}(\om) \geq 1$ (and $(d^{\om}(x,x) \vee 1)^{2-d} =1$).

  \smallskip
  \textbf{Case 2:} $ d = 2$. In this case, we estimate $e^{-c \frac{(d^{\om}(x,y)\vee 1)}{N_{\mathrm{HK}}(\om)}} \leq 1$ in Lemma \ref{lem:heatkernel:before:diffusive}, hence, 
  \begin{align*}
    \int_{0}^{N_{\mathrm{HK}}(\om) \vee c_{2} (d^{\om}(x,y)\vee 1)}q^{\om}_t(x,y) \, \md t
    \,\leq\,
    c N_{\mathrm{HK}}(\om).
  \end{align*}
  For the integral starting from $N_{\mathrm{HK}}(\om) \vee c_{2} \,  d^{\om}(x,y)$ we use again the heat kernel upper bound from \eqref{eq:HKbound:Gaussian} and substitute $s = 2 c_{1} d^{\om}(x,y)^2/t$, hence $ds/s = - dt/t$, which, on the event that $d^{\om}(x,y)^2 \geq N_{\mathrm{HK}}(\om)$, yields
  \begin{align}\label{eq:gauss:integral:bound}
    &\int_{N_{\mathrm{HK}}(\om) \vee c_{2} \,  d^{\om}(x,y)}^{d^{\om}(x,y)^2 \vee N_{\mathrm{HK}}(\om)} C_{\mathrm{HK}} t^{-d/2} e^{- c_{1} \frac{d^{\om}(x,y)^2}{t}} \,\md t
    \,\leq\,
    C_{\mathrm{HK}} \int_{1}^{ d^{\om}(x,y)^2} t^{-1} e^{- c_{1} \frac{d^{\om}(x,y)^2}{t}} \,\md t
    \nonumber\\[.5ex]
    &\,=\,
    C_{\mathrm{HK}} \int_{2 c_{1}}^{ 2 c_{1} d^{\om}(x,y)^2} \frac{e^{-s/2}}{s} \,\md s 
    \,\leq\,
    C_{\mathrm{HK}} \int_{2  c_{1}}^{\infty} \frac{e^{-s/2}}{s} \,\md s 
    \,\leq\,
    c
    \,\leq\,
    c N_{\mathrm{HK}}(\om).
  \end{align}
  On the event that $d^{\om}(x,y)^2 < N_{\mathrm{HK}}(\om)$, we estimate $t^{-d/2} \leq 1$ for any $t \geq 1$ to get,up to a constant, the same bound:
  \begin{align*}
    \int_{N_{\mathrm{HK}}(\om) \vee c_{2} \,  d^{\om}(x,y)}^{d^{\om}(x,y)^2 \vee N_{\mathrm{HK}}(\om)} C_{\mathrm{HK}} t^{-d/2}  \,\md t
    \,\leq\,
    C_{\mathrm{HK}} N_{\mathrm{HK}}(\om).
  \end{align*}
  Hence, the above yields that the integral from time $0$ until time $d^{\om}(x,y)^2 \vee N_{\mathrm{HK}}(\om)$ is bounded by $c N_{\mathrm{HK}}(\om)$. The logarithmic main contribution comes from the following region:
  \begin{align*} 
    &\int_{d^{\om}(x,y)^2 \vee N_{\mathrm{HK}}(\om)}^{ n^2 \vee N_{\mathrm{HK}}(\om)} q^{\om}_t(x,y)  \,\md t
    \,\leq\,
    \int_{(d^{\om}(x,y) \vee 1)^2}^{n^2 N_{\mathrm{HK}}(\om)} C_{\mathrm{HK}} t^{-1} \, \md t
    \nonumber\\[.5ex]
    &\qquad\,=\,
    2 C_{\mathrm{HK}} \log\left(\frac{n}{d^{\om}(x,y) \vee 1}\right) \,+\, 2 C_{\mathrm{HK}} \log(N_{\mathrm{HK}}(\om)),
  \end{align*} 
  where we estimated $d^{\om}(x,y)^2 \vee N_{\mathrm{HK}}(\om) \geq (d^{\om}(x,y) \vee 1)^2$ and notice that we can absorb $\log(N_{\mathrm{HK}}(\om)) \leq N_{\mathrm{HK}}(\om)$ in the error term. 

  For the remaining integral after time $n^2 \vee N_{\mathrm{HK}}(\om)$, using Cauchy Schwarz, recalling \eqref{eq:heat:kernel:killed:Semigroup} and using \eqref{eq:HKbound:Gaussian} yields for any $t \geq N_{\mathrm{HK}}(\om)$
  \begin{align*}
    \sum_{k=1}^{|B^{\om}(n)|} e^{-\la^{\om,\th}_{k} t} \psi^{\om,\th}_k(x) \psi^{\om,\th}_k(y)
    &\,\leq\,
    \sqrt{\sum_{k=1}^{|B^{\om}(n)|} e^{-\la^{\om,\th}_{k} t} \psi^{\om,\th}_k(x)^2}
    \sqrt{\sum_{k=1}^{|B^{\om}(n)|} e^{-\la^{\om,\th}_{k} t} \psi^{\om,\th}_k(y)^2}
    \nonumber\\[.5ex]
    &\,\leq\,
    e^{-\frac{\la^{\om,\th}_{1}t}{2}} \sqrt{q^{\om,B^{\om}(n)}_{t/2}(x,x)q^{\om,B^{\om}(n)}_{t/2}(y,y)}
    \\[.5ex]
    &\,\leq\,
    2 e^{-\frac{\la^{\om,\th}_{1}t}{2}} t^{-1}.
  \end{align*}
  Further, note that by \eqref{eq:principal:eigenvalue:GSRW} and Corollary~\ref{cor:KrengelPyke}, for any $n \geq \const[N]{Nconst:ergodic:constants}(\om,1) \vee \const[N]{Nconst:ass:cluster}(\omega) \vee R^{\theta / (\theta' - \theta)}$, where $\th'$ and hence $\ka$ chosen as in the proof of Theorem~\ref{thm:heat:kernel:bounds},
  \begin{align*}
    \la^{\om,\th}_{1}
    \,\geq\,
    \left(C_{\mathrm{S}} \bar{\nu} \bar{\mu}\right)^{-1} n^{-2}
    \,\rdef\,
    C_{\la}n^{-2}.
  \end{align*}
  Recalling again \eqref{eq:heat:kernel:killed:Semigroup}, the above yields that $q^{\om,B^{\om}(n)}_t(x,y) \leq t^{-1} \exp\{-\frac{C_{\la}n^{-2}t}{2}\}$ for any $x,y \in B^{\om}(n)$ and $t \geq N_{\mathrm{HK}}(\om)$. Using this and substituting $t= s n^2/C_{\la}$ yields
  \begin{align*}
    \int_{n^2 \vee N_{\mathrm{HK}}(\om)}^{\infty}  q^{\om,B^{\om}(n)}_t(x,y) \,\md t 
    \,\leq\,
    \int_{n^2}^{\infty}  2 t^{-1} e^{-\frac{C_{\la} t}{2 n^2 }} \,\md t 
    \,=\,
    2 \int_{ C_{\la} }^{\infty}  \frac{e^{-\frac{s}{2}} }{s}\,\md s
    \,\leq\,
    c
  \end{align*}
  as in \eqref{eq:gauss:integral:bound}.
\end{proof}
\begin{proof}[Proof of Corollary~\ref{cor:green:pointwise:upperbound:corollary}]
  Firstly, we quickly argue that there exists $c>0$ such that 
  \begin{align}\label{eq:grid:distance:Euclidean:distance}
    \left| \frac{\lfloor nx \rfloor}{n} - \frac{\lfloor ny \rfloor}{n}\right|_2 \vee \frac{1}{n}
    \,\geq\,
    c \left|x-y\right|_2.
  \end{align}
  If $|x-y|_2 \leq 4\sqrt{d}/n$, we get $\left|\lfloor nx \rfloor/n - \lfloor ny \rfloor/n\right|_2 \vee 1/n \geq 1/n = (4\sqrt{d})^{-1} 4\sqrt{d}/n \geq (4\sqrt{d})^{-1}|x-y|_2$. If $|x-y|_2 > 4\sqrt{d}/n$, by the triangle inequality and since $|\lfloor nx \rfloor/n - x|_2 \leq \sqrt{d}/n$, we get $|\lfloor nx \rfloor/n - \lfloor ny \rfloor/n|_2 \geq |x-y|_2 - 2\sqrt{d}/n > |x-y|_2/2$. Hence, \eqref{eq:grid:distance:Euclidean:distance} holds for any $n$ with $c=(4\sqrt{d})^{-1}$ since $(4\sqrt{d})^{-1}\leq 1/2$. 

  Moreover, notice that for the cluster distance $d^{\om}(x,y) \geq |x-y|_1 \geq |x-y|_2$ and hence, inserting \eqref{eq:grid:distance:Euclidean:distance} into Theorem~\ref{thm:green:pointwise:upperbound} while using that, by Assumption~\ref{ass:cluster}-(ii), $nD \subset B^{\om}_D(n)$ yields 
  \begin{align*}
    &g^{\om}_{nD}(\lfloor nx \rfloor,\lfloor ny \rfloor)
    \,\leq\,
    2 C_{\mathrm{HK}} \log\left(\frac{n}{d^{\om}(\lfloor nx \rfloor,\lfloor ny \rfloor)\vee 1}\right)
    \,+\, \const[C]{const:pointwise:schroedinger:bound:2d}N_{\mathrm{HK}}(\om)
    \nonumber\\[.5ex]
    &\,\leq\,
    2 C_{\mathrm{HK}} \log\left(\frac{1}{c|x-y|_2}\right)
    \,+\, \const[C]{const:pointwise:schroedinger:bound:2d}N_{\mathrm{HK}}(\om)
    \,\leq\,
    2 C_{\mathrm{HK}} \log\left(\frac{1}{|x-y|_2}\right)
    \,+\, c N_{\mathrm{HK}}(\om)
  \end{align*}
  for any $n \geq \const[N]{Nconst:ass:green:bound}(\om,C_d r_D)$, $\prob_0$-a.e.\ $\om$ and any $x \not= y$. Noticing that in $d\geq 3$
  \begin{align*}
    n^{d-2} g^{\om}(\lfloor nx \rfloor,\lfloor ny \rfloor)
    \,\leq\,
    \const[C]{const:pointwise:schroedinger:bound:3d}  N_{\mathrm{HK}}(\om)^{d-1} \left(\frac{n}{d^{\om}(\lfloor nx \rfloor,\lfloor ny \rfloor) \vee 1} \right)^{d-2}
  \end{align*}
  and again inserting the bound in \eqref{eq:grid:distance:Euclidean:distance} concludes the proof. 
\end{proof} 

\subsection{Limits for nonlinear functionals of (smeared) Green's functions}\label{subsec:greens:function:limit:tested}

In this subsection, we will restrict ourselves to the $d=2$ case. We introduce the (piecewise constant) function $g_n \colon \bbR^2 \times \bbR^2 \to \bbR $, defined as
\begin{align}\label{eq:def:rescaled:greens:function:d2:abbreviate}
  g_n(x,y) \equiv g^{\om}_{nD} (\lfloor nx \rfloor, \lfloor ny \rfloor)
  \qquad
  \text{for all}
  \; x,y \in \bbR^2.
\end{align}
Further, we also introduce smeared approximations of the killed Green's kernel, which will appear in the next section when expanding the square in \eqref{eq:L2:smeared:approximation:introduction}:
\begin{align}\label{eq:def:smeared:green:kernels:rcm}
  \begin{split}
    g^{0,\ve}_n(x,y)
    &\,\ldef\,
    \int g_n(x,z') \rh^{\ve}_{y}(z') \, \md z',
    % \;
    % \text{and}
    % \;
    \\[.5ex]
    g^{\ve,\ve}_n(x,y)
    &\,\ldef\,
    \int \int  \rh^{\ve}_{x}(z) g_n(z,z') \rh^{\ve}_{y}(z') \, \md z \, \md z'.
  \end{split}
\end{align}
Moreover, for $k \in \bbN$ and $G \in \{G_n,G^{0,\ve}_n,G^{\ve,\ve}_n,G^{\Si}\}$ and functions $H \colon \bbR_+ \to \bbR$, we define the following operators associated to the kernel $g \in \{g_n, g^{0,\ve}_n, g^{\ve,\ve}_n,g^{\Si}\}$ acting on integrable functions $f \colon D \to \bbR$ as 
\begin{align}\label{eq:def:operator:H:g}
  (H(G) f ) (x)
  \,\ldef\,
  \int_{D} H(g(x ,y )) \, f(y) \,\md y,
\end{align}
where we also abbreviated $g^{\Si} \equiv g^{\Si}_D$. To make the notation feasible, we will further abbreviate $1^n_{\cC_{\infty}(\om)}(x,y) \equiv \indicator_{\{ \lfloor nx \rfloor \in \mathcal{C}_{\infty}(\omega) \}}\,\indicator_{\{ \lfloor ny \rfloor \in \mathcal{C}_{\infty}(\omega) \}}$. 
We now state the main result in this subsection, which will be the technical backbone of the proof of the scaling limit. For the class of functions considered here, this is specific to dimension $d=2$. In $d=3$, one could prove this also for $H(a)=a^2$, while in all other dimensions $d \geq 4$, it only holds for linear $H(a)=a$, as shown in \cite[Theorem 4.1 Equation (4.1)]{andres2025scaling}, where bespoke linearity allowed for a potential theoretic treatment via maximal inequalities. For a general nonlinear $H$, however, we heavily rely on the precise pointwise near diagonal estimate in Theorem~\ref{thm:green:pointwise:upperbound}, that yields the UI-estimate in Proposition~\ref{prop:UI:assumption:analytic}. This is precisely where the integrability condition encoded in the parameter $\ga$ is obtained. 
\begin{prop}\label{prop:tested:green:convex:F}
  Let $d = 2$ and $D \subset \mathbb{R}^{d}$ be a bounded, strongly regular domain. Suppose there exist $\theta \in (0, 1)$ and $p, q \in [1, \infty]$ such that Assumptions~\ref{ass:law},~\ref{ass:cluster} and~\ref{ass:pq} hold. Moreover, let $H$ be a non negative, non decreasing, convex, locally Lipschitz function and assume that there exist $\be>0$ and $M<\infty$ such that $H(a) \leq Me^{\be a}$. Further, assume that $H(0)=0$. Then, for any $|\ga| < \sqrt{\frac{\th_0}{\be} \left(\frac{1}{C_{\Si}} \wedge \frac{\th_0}{C_{\mathrm{HK}}}\right)}$ and $\prob_0$-a.e.\ $\om$
  \begin{align}
    \lim_{n \to \infty} \scprreal{f}{H(\ga^{2} G_n) f}{D}\label{eq:Green:tested:convergence:no:smear}
    &\,=\,
    \th_0^2 \; \scprreal{f}{H(\ga^{2}\th_0^{-1} G^{\Si}) f}{D},
    \\[.5ex]
    \lim_{\ve \downarrow 0} \lim_{n \to \infty} \scprreal{f}{H((\ga \th_0^{-1})^{2} G^{\ve,\ve}_n) f}{D}\label{eq:Green:tested:convergence:both:smear}
    &\,=\,
    \scprreal{f}{H(\ga^{2} \th_0^{-1}  G^{\Si}) f}{D},
    \\[.5ex]
    \lim_{\ve \downarrow 0} \lim_{n \to \infty} \scprreal{f}{H(\ga^{2} \th_0^{-1} G^{0,\ve}_n) f}{D}\label{eq:Green:tested:convergence:one:smear}
    &\,=\,
    \th_0 \scprreal{f}{H(\ga^{2} \th_0^{-1}  G^{\Si}) f}{D},
  \end{align}
  for any bounded measurable $f \colon D \to \bbR$, with $\Sigma$ as in Theorem~\ref{thm:QIP:RG}.
\end{prop}
The main ingredient of the proof is the local limit theorem for the killed Green's kernel, which was shown by Andres, Slowik and Sokol in \cite{andres2025scaling}. For any $\omega \in \Omega^{*}$ and $n \in \mathbb{N}$, we define the function $\pi_{n}\colon \mathbb{R}^{d} \to \mathcal{C}_{\infty}(\omega)$ that maps any $x \in \mathbb{R}^{d}$ to a closest point of $nx$ in $\mathcal{C}_{\infty}(\omega)$, breaking ties using a fixed ordering on $\mathbb{Z}^{d}$. Then, this reads:
\begin{theorem}[\cite{andres2025scaling} Quenched LCLT for Green's function]\label{thm:LCLT_green}
  Let $d \geq 2$ and $D \subset \mathbb{R}^{d}$ be a bounded, strongly regular domain. Suppose there exist $\theta \in (0, 1)$ and $p, q \in [1, \infty]$ such that Assumptions~\ref{ass:law},~\ref{ass:cluster} and~\ref{ass:pq} hold. For any $\varepsilon, \delta > 0$, set
  \begin{align} \label{eq:defK_eps}
    K_{\varepsilon, \delta}
    \;\ldef\;
    \big\{
    (x, y) \in D \times D :
    \operatorname{dist}(x, \partial D) \wedge \operatorname{dist}(y, \partial D)
    \geq \delta, |x - y|_{2} \geq \varepsilon
    \big\}.
  \end{align}
  Then, for any $\varepsilon, \delta > 0$ with $0 < \varepsilon < \delta$ such that $K_{\varepsilon, \delta} \neq \emptyset$ and $\prob_{0}$-a.e.\ $\omega$,
  \begin{align*}
    \lim_{n \to \infty} \sup_{(x, y) \in K_{\varepsilon, \delta}}
    \biggl|
      n^{d-2} g^{\omega}_{n D}(\pi_{n}(x), \pi_{n}(y)) -  \frac{g_{D}^{\Sigma}(x, y)}{\mathbb{P}[0 \in \mathcal{C}_{\infty}]}
    \biggr|
    \;=\;
    0,
  \end{align*}
  with $\Sigma$ as in Theorem~\ref{thm:QIP:RG}.
\end{theorem}
To elevate the local limit Theorem to the result in Proposition~\ref{prop:tested:green:convex:F}, the following uniform integrability on the diagonal is crucial. With Corollary~\ref{cor:green:pointwise:upperbound:corollary} at hand, this is now fairly easy to show: 
\begin{prop}\label{prop:UI:assumption:analytic}
  Let $d \geq 2$ and $D \subset \bbR^d$ bounded. Suppose that there exist $\theta \in (0, 1)$ and $p, q \in [1, \infty]$ such that Assumptions~\ref{ass:law}, \ref{ass:cluster} and~\ref{ass:pq} hold. Moreover, let $H$ be a non negative, non decreasing and convex function and assume that there exist $\be>0$ and $M<\infty$ such that $H(a) \leq Me^{\be a}$. Then, for any $|\ga| < \sqrt{\th_0^2/\be C_{\mathrm{HK}}}$, we have 
  \begin{align*}
    \limsup_{\ve \downarrow 0} \limsup_{n \to \infty} \int_{|x-y|_2< \ve} H((\ga \th_0^{-1})^{2} g^{\om}_{nD} (\lfloor nx \rfloor, \lfloor ny \rfloor))\, \md x \, \md y 
    \,=\,
    0
    \qquad
    \text{for $\prob_0$-a.e.\ $\om$.}
  \end{align*}
\end{prop}
\begin{proof}
  Firstly, using the exponential growth assumption and Corollary \ref{cor:green:pointwise:upperbound:corollary} yields for any $n \geq \const[N]{Nconst:ass:green:bound}(\om,C_d r_D)$, $\prob_0$-a.e.\ $\om$ and Lebesgue almost every $x,y \in D$
  \begin{align}\label{eq:green:upper:bound:corollary:applied}
    H((\ga \th_0^{-1})^{2} g^{\om}_{nD} (\lfloor nx \rfloor, \lfloor ny \rfloor))
    \,\leq\,
    M e^{(\ga \th_0^{-1})^{2} \be \const[C]{const:pointwise:schroedinger:bound:corollary:2d} N_{\mathrm{HK}}(\om)} \left|x-y\right|_2^{-(\ga \th_0^{-1})^{2} \be 2 C_{\mathrm{HK}}}
  \end{align}
  which, since $N(\om) < \infty$ for $\prob_0$-a.e.\ $\om$, is locally integrable in $d=2$ if $\ga$ is chosen such that $(\ga \th_0^{-1})^{2} \be 2 C_{\mathrm{HK}}< 2$. Hence, the claim follows from the absolute continuity of the Lebesgue measure.
\end{proof}
Moreover, we need the following off-diagonal estimate, which also follows from Corollary \ref{cor:green:pointwise:upperbound:corollary}. An equivalently useful estimate could be obtained using the maximal inequalities in \cite[Proposition 5.6]{andres2025scaling}.
\begin{lemma}\label{lem:applied:Green's:maxbound}
  Let $d \geq 2$ and suppose there exist $\theta \in (0, 1)$ and $p, q \in [1, \infty]$ such that Assumptions~\ref{ass:law}, \ref{ass:cluster} and~\ref{ass:pq} hold. Then, there exists $c < \infty$ such that for any $n \geq \const[N]{Nconst:ass:green:bound}(\om,C_d r_D)$, $\prob_0$-a.e.\ $\om$, any $\ve >0$ and any $\et > 2 \ve$
  \begin{align*}
    \sup_{|x-y|_2 \geq \et} \tilde{g}^{\ve}_n(x,y)
    \,\leq\,
    c \log \left( \frac{1}{\et - 2 \ve}\right) + c N_{\mathrm{HK}}(\om)
  \end{align*}
  for any $\tilde{g}^{\ve}_n \in \{g_n, g^{0,\ve}_n, g^{\ve,\ve}_n\}$.
\end{lemma}
\begin{proof}
  For $g = g_n \equiv g^{\om}_{nD} (\lfloor n \cdot \rfloor, \lfloor n \cdot \rfloor)$ the bound follows direclty from Corollary~\ref{cor:green:pointwise:upperbound:corollary}, since 
  \begin{align}
    \sup_{|x-y|_2 \geq \et} g^{\om}_{nD} (\lfloor nx \rfloor, \lfloor ny \rfloor)
    \,\leq\,
    2 C_{\mathrm{HK}} \log\left(\frac{1}{\et}\right) + \const[C]{const:pointwise:schroedinger:bound:corollary:2d} N_{\mathrm{HK}}(\om).
  \end{align}
  For $g=g^{\ve,\ve}_n$, the bound follows since if $|x-y|_2 \geq \et$ and $(z,z') \in \supp(\rh^{\ve}_x) \times \supp(\rh^{\ve}_y) \subset B_{x,\ve} \times B_{y,\ve}$, we also have that $|z-z'|_2 \geq \et-2\ve >0$, implying
  \begin{align}\label{eq:step3:both:smear:max:bound}
    \sup_{|x-y|\geq \et} g^{\ve,\ve}_n(x,y) 
    &\,\leq\,
    \sup_{|x-y|\geq \et} \; \sup_{(z,z') \in B_{x,\ve} \times B_{y,\ve}} g_n(z,z') \left(\int \int \rh^{\ve}_{x}(z) \rh^{\ve}_{y}(z') \,\md z \md z'  \right)
    \nonumber\\[.5ex]
    &\,\leq\,
    \sup_{|z-z'|>\et - 2\ve} g_n(z,z')
    \,\leq\,
    2 C_{\mathrm{HK}} \log\left(\frac{1}{\et-2 \ve}\right) + \const[C]{const:pointwise:schroedinger:bound:corollary:2d} N_{\mathrm{HK}}(\om).
  \end{align}
  Analogously, one gets that 
  \begin{align*}
    \sup_{|x-y|\geq \et} g^{0,\ve}_n(x,y) 
    \,\leq\,
    \sup_{|x-z'|>\et - \ve} g_n(x,z')
    \,\leq\,
    2 C_{\mathrm{HK}} \log\left(\frac{1}{\et-\ve}\right) + \const[C]{const:pointwise:schroedinger:bound:corollary:2d} N_{\mathrm{HK}}(\om)
  \end{align*}
  Since $\et \geq \et - \ve \geq \et - 2 \ve$, the claim is concluded.
\end{proof}
We will require the following generalisation of the ergodic Theorem in \cite[Theorem 3]{boivin2003spectral}, which we will prove in the appendix:
\begin{lemma}\label{lem:ergodic:lemma:extend:LocHoelder}
  Let $D \subset \bbR^d$ be bounded and suppose that $\prob$ satisfies Assumption~\ref{ass:law}. Then, for any bounded functions $h \colon D \times D \to \bbR$ and $f \colon D \to \bbR$ and locally Hoelder continuous $\cH \colon \bbR \to \bbR$,  
  \begin{align}\label{eq:ergodic:lemma:extend:LocHoelder}
    \lim_{n \to \infty} &\int_{D} \cH \left( \int_{D} h(x,y)1^n_{\cC_{\infty}(\om)}(x)\,
      \md x\, \right) f(y) 1^n_{\cC_{\infty}(\om)}(y) \md y
    \nonumber\\[.5ex] 
    &\,=\,
    \th_0 \int_{D} \cH\left( \th_0 \int_{D} h(x,y)\,
      \md x\, \right)f(y) \md y,
  \end{align}
  for $\prob_0$-a.e.\ $\om$. Moreover, if $\cH(0)=0$, it also holds that 
  \begin{align}\label{eq:ergodic:lemma:extend:LocHoelder:H:mapsorigin:to:origin}
    \lim_{n \to \infty} &\int_{D} \cH \left( \int_{D} h(x,y)1^n_{\cC_{\infty}(\om)}(x,y)\,
      \md x\, \right) f(y) \md y
    \nonumber\\[.5ex] 
    &\,=\,
    \th_0 \int_{D} \cH\left( \th_0 \int_{D} h(x,y)\,
      \md x\, \right)f(y) \md y,
  \end{align}
  for $\prob_0$-a.e.\ $\om$. 
\end{lemma}
In the following calculations, we will also encounter the kernels $g^{\Si,0,\ve}$, $g^{\Si,\ve,\ve}$ and $g^{\Si,\ve,\de}$, which are defined as in \eqref{eq:def:smeared:green:kernels:rcm} but with $g^{\Si}$ replacing $g_n$. The final technical ingredient we list before starting the actual proof is a purely analytic Lemma regarding those smooth approximations of $g^{\Si}$ which is also shown in the Appendix: 
\begin{lemma}\label{lem:continuous:Green's:smeared:converge}
  Let $d \geq 2$ and $D \subset \bbR^d$ be bounded. For a convex and continuous function $\cH \colon \bbR_+ \to \bbR$ assume that $ \cH \circ g^{\Si}_D \in L^1(D\times D)$. Then, for $g^{\ve,\de} \in \{g^{\Si,\ve,\ve},g^{\Si,0,\ve},g^{\Si,\ve,\de}\}$ we have 
  \begin{align}\label{eq:continuous:Green's:smeared:L1:kernel:converge}
    \lim_{\ve,\de \downarrow 0}\cH \circ g^{\ve,\de}
    \,=\,
    \cH \circ g^{\Si}
    \qquad\;
    \text{in}
    \;
    L^1(D \times D)\,\text{and a.e.}.
  \end{align}
\end{lemma}
We have now collected all major technical inputs to give the 
\begin{proof}[Proof of Proposition~\ref{prop:tested:green:convex:F}]
  Throughout the proof, we will, mainly for convenience, set $\de \equiv \de(\ve) \ldef \sqrt{\ve}$. In order to apply the LCLT in Theorem~\ref{thm:LCLT_green}, we will cover the integration region by $D \times D \subset K_{3\ve,\de} \cup T_{ 3 \ve} \cup S_{\de}$, where we defined the diagonal region
  \begin{align*}
    T_{3 \ve}
    \,=\,
    \left\{(x,y) \in D \times D \,:\, |x -y|_2 < 3 \ve \right\},
  \end{align*}
  and the boundary region
  \begin{align*}
    S_{\de}
    \,=\,
    \left\{(x,y) \in D \times D \,:\, \dist(x,\partial D) \wedge \dist(y,\partial D) < \de \right\}.
  \end{align*}
  The proof will be devided in three steps, in which we will show convergence of \eqref{eq:Green:tested:convergence:no:smear}-\eqref{eq:Green:tested:convergence:one:smear} by treating those three regions seperately. 
  
  Before we start, note that \eqref{eq:continuous:green:pointwise:log} and the exponential growth assumption on $H$ imply that 
  \begin{align}\label{eq:H_ga:locally:integrable}
    H(\ga^2 \th_0^{-1} g^{\Si}_D(x,y))
    \,\leq\,
    M e^{\ga^2 \th_0^{-1} \be c} |x-y|_2^{-\ga^2 \th_0^{-1} \be 2 C_{\Si}},
  \end{align}
  which is locally integrable in $d=2$ for $|\ga| < \sqrt{\th_0/ \be C_{\Si}}$.
  Since $|T_{3 \ve} \cup S_{\de}| \downarrow \emptyset$ as $\ve \downarrow 0$, the continuity of the Lebesgue measure implies that 
  \begin{align*}
    \lim_{\ve \downarrow 0} \int_{T_{3 \ve} \cup S_{\de}} H(\ga^2 \th_0^{-1} g^{\Si}(x,y)) \, \md x \md y 
    \,=\,
    0.
  \end{align*}
  Thus, at least on the regions $T_{3 \ve}$ and $S_{\de}$, we will only be concerned with $g_n, g^{0,\ve}_n$ and $g^{0,\ve,\ve}_n$. 
  
  \smallskip
  \textbf{1.1 LCLT region: No smearing \eqref{eq:Green:tested:convergence:no:smear}.} We abbreviate $\tilde{H} \equiv H(\ga^2 \cdot)$. Adding $0$ and using the triangle inequality, we get
  \begin{align}\label{eq:step1:no:smear:split:before:LCLT}
    &\left| \int_{K_{3\ve, \de}} \left( \tilde{H}(g_n(x,y)) - \th_0^2 \tilde{H}\left(\frac{g^{\Si}(x,y)}{\th_0}\right) \right)f(x) f(y)  \, \md x \, \md y \right|
    \nonumber\\[.5ex]
    &\,\leq\,
    \Norm{f}{\infty}^2 \bigg(\int_{K_{3\ve, \de}}  \left| \tilde{H}\left(g_n(x,y)\right) - \tilde{H}\left(\frac{g^{\Si}(x,y)}{\th_0}1^n_{\cC_{\infty}(\om)}(x,y) \right) \right|\, \md x \, \md y 
    \nonumber\\[.5ex] 
    &\,+\,
    \left|  \int_{K_{3\ve, \de}} \left( \tilde{H}\left(\frac{g^{\Si}(x,y)}{\th_0}1^n_{\cC_{\infty}(\om)}(x,y)\right)   - \th_0^2 \tilde{H}\left(\frac{g^{\Si}(x,y)}{\th_0}\right)\right)
      \, \md x \md y\right| \bigg).
  \end{align}
  We first take a look at the first term in \eqref{eq:step1:no:smear:split:before:LCLT}. Since, by Lemma \ref{lem:applied:Green's:maxbound} (with $\et=3 \ve$) and \eqref{eq:continuous:green:pointwise:log}, $g_n$ and $g^{\Si}$ are uniformly bounded by a constant $C(\ve)$ on $K_{3\ve,\de}$ for any $\ve>0, n \geq \const[N]{Nconst:ass:green:bound}(\om,C_d r_D)$ and $H$ is locally Lipschitz, there exists $L\equiv L_{[0,C(\ve)]}$ such that $|\tilde{H}(a)-\tilde{H}(b)|\leq L |a-b|$ for all $a,b \in [0,C(\ve)]$ and hence, the integral in the first term in \eqref{eq:step1:no:smear:split:before:LCLT} can be bounded by 
  \begin{align*}
    |K_{3 \ve, \de}| L \sup_{(x,y)\in K_{3\ve, \de}} \left|\left(g^{\om}_{nD}(\pi_n(x),\pi_n(y)) - \frac{g^{\Si}(x,y)}{\th_0}\right)1^n_{\cC_{\infty}(\om)}(x,y)\right|
    \,\xrightarrow{n \to \infty}\,
    0
  \end{align*}
  for $\prob_0$-a.e.\ $\om$, where the convergence follows from the LCLT in Theorem~\ref{thm:LCLT_green} since $3 \ve < \de = \sqrt{\ve}$ for $\ve$ small, where we estimated $1^n_{\cC_{\infty}(\om)}(x,y)\leq1$ and also used that $g_{n}(x,y)=g^{\om}_{nD}(\pi_n(x),\pi_n(y))$ for all $x,y$ such that $\lfloor nx \rfloor, \lfloor ny \rfloor \in \mathcal{C}_{\infty}(\omega)$ and $0$ else. 

  Since $\tilde{H}(0)=0$ implies that $\tilde{H}(a \indicator_A)= \tilde{H}(a)\indicator_A$ for any $a \in \bbR_+$ and any measurable set $A$, we can rewrite the second term in \eqref{eq:step1:no:smear:split:before:LCLT} as
  \begin{align*}
    \left|  \int_{D \times D} \left( \tilde{H}\left(\frac{g^{\Si}(x,y)}{\th_0}\indicator_{K_{3\ve, \de}}\right)1^n_{\cC_{\infty}(\om)}(x,y)   - \th_0^2 \tilde{H}\left(\frac{g^{\Si}(x,y)}{\th_0}\indicator_{K_{3\ve, \de}}\right)\right)
      \, \md x \md y\right|,
  \end{align*}
  which vanishes for any $\ve>0$ fixed and $\prob_0$-a.e.\ $\om$ as $n \to \infty$ upon an application of Lemma~\ref{lem:ergodic:lemma:extend:LocHoelder} choosing $\cH(a)=a$ and $h(x,y) = \tilde{H}\left(\indicator_{K_{3\ve, \de}}(x,y) g^{\Si}(x,y) \th_0^{-1}\right)$, which is bounded by \eqref{eq:continuous:green:pointwise:log}, the definition of $K_{3\ve,\de}$ in \eqref{eq:defK_eps} and the continuity of $H$. 

  \smallskip
  \textbf{1.2 LCLT region: Both entries smeared \eqref{eq:Green:tested:convergence:both:smear}.} Abbreviate $\tilde{H} \equiv H((\ga \th_0^{-1})^{2} \cdot )$. Firstly, adding and substracting $\tilde{H}(\th_0 g^{\Si,\ve,\ve}(x,y))$ and then using again the local Lipschitzness for any fixed $\ve \in (0,1)$ for the first term, we get
  \begin{align}\label{eq:step1:two:smear:start}
    &\left| \int_{K_{3\ve, \de}} \left( \tilde{H}(g^{\ve,\ve}_n(x,y)) - \tilde{H}(\th_0g^{\Si}(x,y))\right) \, f(x) f(y) \md x \md y \right|
    \nonumber\\[.5ex]
    &\quad\,\leq\,
    \Norm{f}{\infty}^2  \bigg(L\int_{K_{3\ve, \de}} \left| g^{\ve,\ve}_n(x,y) - \th_0g^{\Si,\ve,\ve}(x,y)\right| \,\md x \md y 
    \nonumber\\[.5ex]
    &\qquad\,+\,
    \int_{D \times D} \left|\tilde{H} (\th_0 g^{\Si,\ve,\ve}(x,y)) - \tilde{H}(\th_0 g^{\Si}(x,y)) \right| \,\md x \md y \bigg).
  \end{align}
  For the first integral in \eqref{eq:step1:two:smear:start} adding and substracting $(g^{\Si}(z,z')/\th_0) 1^n_{\cC_{\infty}(\om)}(z,z')$ and using that $g_n(z,z')= g^{\om}_{nD}(\pi_n(z),\pi_n(z')) 1^n_{\cC_{\infty}(\om)}(z,z')$, yields
  \begin{align}\label{eq:step1:two:smear:add:percolation}
    &\int_{K_{3\ve, \de}} \left| \int \int \rh^{\ve}_x(z) \rh^{\ve}_y(z') \left( g_n(z,z') - \th_0 g^{\Si}(z,z') \right) \,\md z \md z' \right| \, \md x \, \md y
    \nonumber\\[.5ex] 
    &\,\leq\,
    \int_{K_{3\ve, \de}} \left| \int \int \rh^{\ve}_x(z) \rh^{\ve}_y(z') \left( g^{\om}_{nD}(\pi_n(z),\pi_n(z')) - \frac{g^{\Si}(z,z')}{\th_0} \right)1^n_{\cC_{\infty}(\om)}(z,z') \,\md z \md z' \right| 
    \nonumber\\[.5ex] 
    &\qquad\,+\,
    \left| \int \int \rh^{\ve}_x(z) \rh^{\ve}_y(z') \left( \frac{g^{\Si}(z,z')}{\th_0}1^n_{\cC_{\infty}(\om)}(z,z') - \th_0 g^{\Si}(z,z') \right) \,\md z \md z' \right| \, \md x \, \md y.
  \end{align}
  Now, we observe that if $(x,y) \in K_{3\ve,\de}$ and $(z,z') \in \supp (\rh^{\ve}_x) \times \supp (\rh^{\ve}_y) \subset B_{x,\ve} \times B_{y,\ve}$ we also have $|z-z'| \geq \ve$ and $(z,z') \in K_{\ve,\de-\ve} \left(\not= \emptyset\right)$. Hence, using that $\int_{K_{3\ve, \de}} \left| \int \int \rh^{\ve}_x(z) \rh^{\ve}_y(z') \,\md z \md z' \right| \, \md x \, \md y \leq |K_{3\ve, \de}|$, we estimate the first integral in \eqref{eq:step1:two:smear:add:percolation}, using that the mollifiers are only supported in an $\ve$-ball, 
  \begin{align}\label{eq:step1:two:smear:apply:LCLT}
    &|K_{3\ve, \de}| \sup_{(x,y) \in K_{3\ve, \de}} \sup_{(z,z') \in B_{x,\ve} \times B_{y,\ve}} \left|g^{\om}_{nD}(\pi_n(z),\pi_n(z')) - \frac{g^{\Si}(z,z')}{\th_0}\right|
    \nonumber\\[.5ex] 
    &\,\leq\,
    |K_{3\ve, \de}|  \sup_{(z,z') \in K_{\ve,\de-\ve}} \left|g^{\om}_{nD}(\pi_n(z),\pi_n(z')) - \frac{g^{\Si}(z,z')}{\th_0}\right|
    \,\xrightarrow{n \to \infty}\,
    0
    \qquad
    \text{for $\prob_0$ a.e.\ $\om$,}
  \end{align}
  where convergence follows again from the LCLT in Theorem \ref{thm:LCLT_green}, since $\de-\ve = \sqrt{\ve} - \ve > \ve$ for any $\ve \in (0,1/4)$. For the second term in \eqref{eq:step1:two:smear:add:percolation}, we observe that $h(z,z')\ldef \indicator_{(x,y)\in K_{3 \ve, \de}} \rh^{\ve}_x(z) \rh^{\ve}_y(z')g^{\Si}(z,z')/\th_0 \leq \left(\ve^{-d}\Norm{\rh}{\infty}\right)^2 \indicator_{(z,z')\in K_{\ve,\de-\ve}} g^{\Si}(z,z')/\th_0 \equiv C(\ve)$ is bounded for any fixed $\ve >0$. Hence, applying Lemma \ref{lem:ergodic:lemma:extend:LocHoelder} (choosing $\cH(a)=a$ in \eqref{eq:ergodic:lemma:extend:LocHoelder}) and bespoke $h$ in the inner integral for every $x,y,\ve$ fixed and then using bounded convergence, yields the convergence of the second integral in \eqref{eq:step1:two:smear:add:percolation} to $0$ as $n \to \infty$ for any fixed $\ve>0$ and $\prob_0$-a.e.\ $\om$. 

  The second term in \eqref{eq:step1:two:smear:start} vanishes by Lemma~\ref{lem:continuous:Green's:smeared:converge}, since $\tilde{H}(\th_0 \cdot ) \circ g^{\Si} = H(\ga^{2}\th_0^{-1} \cdot )  \circ g^{\Si} \in L^1(D \times D)$, which follows from \eqref{eq:H_ga:locally:integrable}.

  \smallskip
  \textbf{1.3 LCLT region: One entry smeared \eqref{eq:Green:tested:convergence:one:smear}.} We abbreviate $\tilde{H} \equiv H(\ga^2 \th_0^{-1} \cdot)$. Adding $0$ two times and using the triangle inequality, we get 
  \begin{align*}
    &\left| \int_{K_{3\ve, \de}}  \left(\tilde{H}(g^{0,\ve}_n(x,y)) - \th_0 \tilde{H}(g^{\Si}(x,y)) \right) f(x) f(y)\, \md x \, \md y \right|
    \nonumber\\[.5ex]
    &\,\leq\,
    \Norm{f}{\infty}^2 \bigg( \int_{K_{3\ve, \de}}  \left| \tilde{H}(g^{0,\ve}_n(x,y)) - H\left(\int \frac{g^{\Si}(x,z')}{\th_0} \rh^{\ve}_{y} (z')1^n_{\cC_{\infty}}(x,z') \, \md z'\right) \right|\, \md x \, \md y
    \nonumber\\[.5ex]
    &\quad\,+\,
    \bigg| \int_{K_{3\ve, \de}} \tilde{H}\left(\int \frac{g^{\Si}(x,z')}{\th_0} \rh^{\ve}_{y} (z')1^n_{\cC_{\infty}}(x,z') \, \md z'\right) \, \md x \, \md y 
    \nonumber\\[.5ex]
    &\qquad\qquad\qquad\qquad\qquad\,-\,
    \int_{K_{3\ve, \de}} \th_0 \tilde{H} \left(\int g^{\Si}(x,z') \rh^{\ve}_{y} (z')\, \md z'\right)\, \md x \, \md y \bigg|
    \nonumber\\[.5ex]
    &\quad\,+\,
    \int_{D \times D} \th_0 \left| \tilde{H}\left(\int g^{\Si}(x,z') \rh^{\ve}_{y} (z')\, \md z'\right) - \th_0 \tilde{H}(g^{\Si}(x,y))\right|\,  \md x \, \md y \bigg).
  \end{align*}
  The first term above vanishes $\prob_0$ almost everywhere as $n \to \infty$ for any fixed $\ve >0$ by using the local Lipschitz continuity and following again the arguments in \eqref{eq:step1:two:smear:add:percolation} and \eqref{eq:step1:two:smear:apply:LCLT} leading up to the application of the LCLT on $K_{2 \ve, \de-\ve}$ (since there is only one mollifier this time, $2\ve$ instead of $\ve$ appears here). 
  The second term vanishes $\prob_0$ almost everywhere as $n\to \infty$ for any fixed $\ve \in (0,1)$ upon applying \eqref{eq:ergodic:lemma:extend:LocHoelder:H:mapsorigin:to:origin} in Lemma~\ref{lem:ergodic:lemma:extend:LocHoelder}. The third term above vanishes as $\ve \downarrow 0$ by \eqref{eq:continuous:Green's:smeared:L1:kernel:converge}, which can be applied since $\tilde{H} \circ g^{\Si} = H(\ga^{2} \th_0^{-1} \cdot) \circ g^{\Si} \in L^1(D \times D)$, by \eqref{eq:H_ga:locally:integrable}.

  \smallskip
  \textbf{2. Diagonal region.} In this step we show that for any $|\ga| < 1/\sqrt{\th_0^2 \be C_{\mathrm{HK}}}$, any $\tilde{g}^{\ve}_n \in \{g_n, g^{\ve,\ve}_n, g^{0,\ve}_n, \}$ and $\tilde{H} \in \{H(\ga^2 \cdot), H((\ga \th_0^{-1})^{2} \cdot), H(\ga^{2} \th_0^{-1} \cdot)\}$ the corresponding function from first step, we have 
  \begin{align*}
    \lim_{\ve \downarrow 0} \lim_{n \to \infty} \int_{|x-y|_2< 3 \ve} \tilde{H}(\tilde{g}^{\ve}_n(x,y)) \,\md x \md y
    \,=\,
    0.
  \end{align*}
  This will follow from Proposition \eqref{prop:UI:assumption:analytic} and applying the following with $\et = 3 \ve$: Namely, we claim that for any $\ve >0$ and $\et > 2\ve$ 
  \begin{align}\label{eq:step2:claim:on:diagonal:eta:epsilon}
    \int_{|x-y|_2< \et} \tilde{H}(\tilde{g}^{\ve}_n(x,y)) \,\md x \md y
    \,\leq\,
    \int_{|x-y|_2 < \et + 2 \ve} H((\ga \th_0^{-1})^{2}g^{\om}_{nD}(\lfloor n x \rfloor,\lfloor n y \rfloor)) \,\md x \md y.
  \end{align}

  First note that by monotonicity $\tilde{H}\leq H((\ga \th_0^{-1})^{2} \cdot)$ for any $\tilde{H}$. 

  If $\tilde{g}^{\ve}_n = g_n$, the claim follows since $|x-y|_2< \et$ implies $|x-y|_2< \et + 2 \ve$ for any $\ve \geq 0$. 

  If $\tilde{g}^{\ve}_n = g^{\ve,\ve}_n$, we use that if $|x-y|_2<\et$ and $(z,z') \in \supp (\rh^{\ve}_x) \times \supp (\rh^{\ve}_y) \subset B_{x,\ve} \times B_{y,\ve}$ it follows that $|z-z'|_2 < \et + 2\ve$, which, using $H(0)=0$ in the first inequality, implies
  \begin{align}\label{eq:step2:claim:both:smear:diagonal}
    \int_{|x-y|_2<\et} &H((\ga\th_0^{-1})^{2} g^{\ve,\ve}_n(x,y)) \,\md x \md y
    \nonumber\\[.5ex]
    &\,\leq\,
    \int_{D \times D} H\left( (\ga\th_0^{-1})^{2}\int \int \rh^{\ve}_{x}(z) \rh^{\ve}_{y}(z') g_n(z,z') \indicator_{|z-z'|_2<\et + 2\ve} \,\md z \md z' \right) \,\md x \md y
    \nonumber\\[.5ex]
    &\,\leq\,
    \int \int H((\ga\th_0^{-1})^{2} g_n(z,z')) \indicator_{|z-z'|_2< \et + 2\ve} \,\md z \,\md z',
  \end{align}
  where the second inequality follows by applying Jensens inequality and then using that $\int_D \int_D\rh^{\ve}_{x}(z) \rh^{\ve}_{y}(z') \, \md x\md y \leq 1$.

  If $\tilde{g}^{\ve}_n = g^{0,\ve}_n$ the claim follows by repeating the exact same argument as in \eqref{eq:step2:claim:both:smear:diagonal}, but this time with only one mollifier, and then using that $|x-y|_2 < \et + \ve$ implies that $|x-y|_2 < \et + 2\ve$. Hence, \eqref{eq:step2:claim:on:diagonal:eta:epsilon} holds. 

  \smallskip
  \textbf{3. Boundary region.} Let us briefly recall that we defined $\de = \de(\ve)=\sqrt{\ve}$. For any $\et> 2\ve$, Lemma~\ref{lem:applied:Green's:maxbound} and the exponential growth assumption imply 
  \begin{align}\label{eq:step3:nosmear:boundary:converge}
    \int_{S_{\de}} \tilde{H} (\tilde{g}^{\ve}_n(x,y)) \,\md x \md y
    &\,\leq\,
    \int_{S_{\de} \cap \{|x-y| \geq \et\}} \tilde{H}(\tilde{g}^{\ve}_n(x,y)) \,\md x \md y
    +
    \int_{|x-y| < \et} \tilde{H}(\tilde{g}^{\ve}_n(x,y)) \,\md x \md y
    \nonumber\\[.5ex]
    &\,\leq\,
    c(\om) |S_{\de}| (\et - 2 \ve)^{-(\ga\th_0^{-1})^{2}c}
    +
    \int_{|x-y| < \et} \tilde{H}(\tilde{g}^{\ve}_n(x,y)) \,\md x \md y,
  \end{align}
  where we abbreviated $c(\om) \equiv M e^{\ga^2 \th_0^{-2} c N_{\mathrm{HK}}(\om)} < \infty $ for $\prob_0$-a.e.\ $\om$. Now, since $|S_{\de}| \downarrow 0$ as $\ve \downarrow 0$, the first term vanishes for $\prob_0$-a.e.\ $\om$ as $\ve \downarrow 0$ for any $n \geq \const[N]{Nconst:ass:green:bound}(\om,C_d r_D)$ and any $\et >0$, while the second term vanishes for $\prob_0$-a.e.\ $\om$ upon applying \eqref{eq:step2:claim:on:diagonal:eta:epsilon} and then Proposition~\ref{prop:UI:assumption:analytic}.
\end{proof}

\section{Scaling limit}\label{sec:scaling:limit}
\subsection{Smooth approximation}\label{subsec:scaling:limit:smooth:approximation}
Similar to the definition in the introduction of the smeared CGFF $\Psi^{\ve} \ldef \{\Psi^{\ve}(x) \,:\, x \in D\}$, where $\Psi^{\ve}(x)\ldef \scprfield{\Psi}{\rh^{\ve}_x}$, we introduce a smooth approximation of the DGFF. Hence, we define 
\begin{align*}
  \vp^{\ve}_{n}(x)
  \,\ldef\,
  \scprfield{\Phi_{n}}{\rh^{\ve}_x}
  \,=\,
  \int_D \vp^{nD}_{\lfloor nz \rfloor} \indicator_{\lfloor nz \rfloor \in \cC_{\infty}(\om)} \rh^{\ve}_x(z) \, \md z.
\end{align*}
and denote the smeared DGFF by $\Phi^{\ve}_{n} \ldef \{\vp^{\ve}_{n}(x) \,:\, x \in D\}$. We will show convergence of this process in $\left(C(\bar{D}),\Norm{\cdot}{\infty}\right)$, the space of all continuous functions mapping from $\bar{D}$ to $\bbR$ endowed with the topology of uniform convergence. 

\begin{prop}\label{prop:weak:convergence:smooth:approximation}
  Let $d \geq 2$ and $D \subset \mathbb{R}^{d}$ be a bounded, strongly regular domain. Suppose there exist $\theta \in (0, 1)$ and $p, q \in [1, \infty]$ such that Assumptions~\ref{ass:law},~\ref{ass:cluster} and~\ref{ass:pq} hold. Then, for any $\ve \in (0,1)$ and $\prob_0$-a.e.\ $\om$
  \begin{align*}
    \Phi^{\ve}_{n}
    \underset{n \to \infty}{\overset{\text{law}}{\;\longrightarrow\;}}
    \sqrt{\th_0}\Psi^{\ve}
    \qquad
    \text{in}
    \;
    \left(C(\bar{D}),\Norm{\cdot}{\infty}\right)
  \end{align*}
  under $\probPhi^{\om}$ with $\Sigma$ as in Theorem~\ref{thm:QIP:RG}.
\end{prop}
This follows upon showing marginal convergence and tightness in the following two lemmata:
\begin{lemma}\label{lem:fdd:convergence:smooth:approximation}
  In the setting of Proposition~\ref{prop:weak:convergence:smooth:approximation}
  \begin{align*}
    \Phi^{\ve}_{n}
    \underset{n \to \infty}{\overset{\text{f.d.d.}}{\;\longrightarrow\;}}
    \sqrt{\th_0}\Psi^{\ve}.
  \end{align*}
\end{lemma}
\begin{proof}
  Since $\Phi^{\ve}_{n}$ and $\th_0^{1/2}\Psi^{\ve}$ are centered Gaussian processes, it suffices to show convergence of the covariances. Using polarisation for bilinear forms to rewrite
  \begin{align}\label{eq:fdd:convergence:polarization:dgff}
    \meanPhi^{\om} \left[\vp^{\ve}_{n}(x)\vp^{\ve}_{n}(y)\right]
    &\,=\,
    \frac{1}{2} \left(
      \meanPhi^{\om} \left[\vp^{\ve}_{n}(x)^2\right]
      +
      \meanPhi^{\om} \left[\vp^{\ve}_{n}(y)^2\right]
      -
      \meanPhi^{\om} \left[(\vp^{\ve}_{n}(x)-\vp^{\ve}_{n}(y))^2\right]
    \right)
    \nonumber\\[.5ex]
    &\,=\,
    \frac{1}{2} \left(
      \meanPhi^{\om} \left[\scprfield{\Phi_{n}^{D}}{\rh^{\ve}_x}^2\right]
      +
      \meanPhi^{\om} \left[\scprfield{\Phi_{n}^{D}}{\rh^{\ve}_y}^2\right]
      -
      \meanPhi^{\om} \left[\scprfield{\Phi_{n}^{D}}{\rh^{\ve}_x-\rh^{\ve}_y}^2\right]
    \right)
    \nonumber\\[.5ex]
    &\,\xrightarrow{n \to \infty}\,
    \meanPhi^{\Si}[(\sqrt{\th_0}\Psi^{\ve}(x))(\sqrt{\th_0}\Psi^{\ve}(y))],
  \end{align}
  where we used that $\scprfield{\Phi_{n}^{D}}{\rh^{\ve}_x}-\scprfield{\Phi_{n}^{D}}{\rh^{\ve}_y} = \scprfield{\Phi_{n}^{D}}{\rh^{\ve}_x- \rh^{\ve}_y}$ by linearity of the field in the second inequality. Now, convergence of each of the three terms follows from \cite[Theorem 4.1]{andres2025scaling} or alternatively from \eqref{eq:Green:tested:convergence:no:smear} in Proposition~\ref{prop:tested:green:convex:F}, since the functions $\rh^{\ve}_x, \rh^{\ve}_y$ and $(\rh^{\ve}_x-\rh^{\ve}_y)$ are bounded for any $\ve \in (0,1)$ fixed.  
\end{proof}
Next, we use Kolmogorov Centsov to show continuity and hence tightness of the sequence $\{\Phi^{\ve}_{n}\}_n$:
\begin{lemma}\label{lem:smooth:approx:continuity:kol:centsov}
  In the setting of Proposition~\ref{prop:weak:convergence:smooth:approximation}, for any $\ve \in (0,1)$ and any $k > 1$ there exists a constant $C_H(\ve,k)<\infty$ such that for any $n \geq \const[N]{Nconst:ass:green:bound}(\om,C_d r_D)$ and any $x,y \in D$, 
  \begin{align}\label{eq:Kolmogorov:centsov:estimate}
    \meanPhi^{\om} \left[
      (\vp^{\ve}_n(x) - \vp^{\ve}_n(y))^{2k}
    \right]
    \,\leq\,
    C_H(\ve,k) |x-y|^{2k}
  \end{align}
  for $\prob_0$-a.e.\ $\om$. Moreover, the process $\Phi^{\ve}_n$ is $\ga$-Hoelder continuous for any $\ga <1$ and $\prob_0$-a.e.\ $\om$.
\end{lemma}
\begin{proof}
  Note that for $Z \sim \cN(0,\si^2)$, $\bbE[Z^{2k}] = \tilde{c}(k) \si^{2k}$, where $\tilde{c}(k) = (2k-1)!!$. Since $\vp^{\ve}_n(x) - \vp^{\ve}_n(y) = \scprfield{\Phi_n}{\rh^{\ve}_x-\rh^{\ve}_y} \sim \cN(0,\scprreal{\rh^{\ve}_x-\rh^{\ve}_y}{G_n (\rh^{\ve}_x-\rh^{\ve}_y)}{D})$, this implies
  \begin{align}\label{eq:kolmogorov:centsov:first:step}
    \meanPhi^{\om} \left[
      (\vp^{\ve}_n(x) - \vp^{\ve}_n(y))^{2k}
    \right]
    &\,=\,
    \tilde{c}(k) \left(\scprreal{\rh^{\ve}_x-\rh^{\ve}_y}{G_n (\rh^{\ve}_x-\rh^{\ve}_y)}{D}\right)^{k}
    \nonumber\\[.5ex]
    &\,\leq\,
    \tilde{c}(k) \Norm{\rh^{\ve}_x-\rh^{\ve}_y}{\infty}^{2k} \left(\scprreal{1}{G_n 1}{D}\right)^{k}.
  \end{align}
  Recall that $nD \cap \cC_{\infty} \subseteq B^{\om}(0,C_d r_D n) \rdef B^{\om}_D(n)$ and rewrite $\scprreal{1}{G_n 1}{D}= n^{-2d} \sum_{x,y} n^{d-2} g^{\om}_{nD}(x,y) \,\leq\, n^{-d} \sum_{x} u(x)$, where $u(x) \ldef \sum_y n^{-2} g^{\om}_{B^{\om}_D(n)}(x,y)$ solves $\cL^{\om} u = - n^{-2}$ on $B^{\om}_D(n)$ and $u=0$ else. Hence, \cite[Proposition 2.2]{andres2025scaling} implies that there exist a $\ka \equiv \ka(\th,d,p,q)$ and a constant $c$ such that for any $n \geq \const[N]{Nconst:ass:cluster}(\omega) \vee (C_d r_D)^{\ka}$, we have $\Norm{u}{\infty,B^{\om}_D(n)} \leq c \Norm{\nu^{\om}}{q, B^{\om}_D(n)}$. Moreover, by Corollary~\ref{cor:KrengelPyke}, $\Norm{\nu^{\om}}{q, B^{\om}_D(n)} \leq \bar{\nu}$ for any $n \geq \const[N]{Nconst:ergodic:constants}(\om,1) \geq \const[N]{Nconst:ergodic:constants}(\om,C_d r_D)$, and hence, we have 
  \begin{align}\label{eq:rcm:green:tested:1:upperbound:via:bp}
    \scprreal{1}{G_n 1}{D}
    \,\leq\,
    c \Norm{\nu^{\om}}{q, B^{\om}_D(n)}
    \,\leq\,
    c \bar{\nu},
  \end{align}
  for any $n \geq  \const[N]{Nconst:ass:cluster}(\omega) \vee (C_d r_D)^{\ka} \vee \const[N]{Nconst:ergodic:constants}(\om,1) \left( = \const[N]{Nconst:ass:green:bound}(\om,C_d r_D)\right)$. Since $\rh$ is smooth with compact support, $M \ldef \sup_w |\nabla \rh(w)| < \infty$ and hence, recalling the definition $\rh^{\ve}_x(z) \ldef \ve^{-d} \rh\left((x-z)/\ve\right)$, we have 
  \begin{align}\label{eq:mollifier:derivative:estimate}
    \Norm{\rh^{\ve}_x-\rh^{\ve}_y}{\infty}
    \,=\,
    \sup_{z \in \bbR^d}
    \left|\rh^{\ve}_x(z) - \rh^{\ve}_y(z)\right| 
    \,\leq\,
    \frac{M}{\ve^{d+1}} |x-y|.
  \end{align}
  Plugging \eqref{eq:rcm:green:tested:1:upperbound:via:bp} and \eqref{eq:mollifier:derivative:estimate} into \eqref{eq:kolmogorov:centsov:first:step} and setting $C_H(\ve,k) \ldef \tilde{c}(k) \left(c \bar{\nu} (M \ve^{-(d+1)})^2\right)^k$ finally yields \eqref{eq:Kolmogorov:centsov:estimate}. 

  For any $\ga \in (0,1)$ choose $k \ldef \lceil\frac{d}{2(1-\ga)}\rceil+1$. Then, applying Kolmogorov Centsov as stated in \cite[Problem 2.9]{karatzas2014brownian} with $\al = 2k$ and $\be = 2k - d$, \eqref{eq:Kolmogorov:centsov:estimate} implies Hoelder continuity for any $\tilde{\ga} < \frac{2k-d}{2k} = 1 - \frac{d}{2k}$, but $1 - \frac{d}{2k} > \ga$, hence the claim follows.
\end{proof}
\begin{proof}[Proof of Proposition~\ref{prop:weak:convergence:smooth:approximation}]
  Follows from the Lemma~\ref{lem:fdd:convergence:smooth:approximation} and Lemma~\ref{lem:smooth:approx:continuity:kol:centsov}.
\end{proof}
We wish to extend the preceding Proposition to hold for analytic functions of the field. The following Proposition is formulated with Theorem~\ref{thm:marginal:convergence:analytic:field} and ultimately the application of the Cramer-Wold device in the proof of Theorem~\ref{thm:mainresult:2} in mind. 
\begin{prop}\label{prop:convergence:analyticfunction:smooth:approximation}
  Let $d \geq 2$ and $D \subset \mathbb{R}^{d}$ be a bounded, strongly regular domain. Suppose there exist $\theta \in (0, 1)$ and $p, q \in [1, \infty]$ such that Assumptions~\ref{ass:law},~\ref{ass:cluster} and~\ref{ass:pq} hold. Further, assume that $F_k$ are $\be_k$-Fock-entire functions. Then, for any $\ga_k \in \bbR$, any $\ve \in (0,1)$ and $\prob_0$-a.e.\ $\om$ 
  \begin{align*}
    \sum_{k =1 }^K \al_k \scprfield{\wick{F_k\left(\frac{\ga_k \Phi^{\ve}_n}{\th_0}\right)}}{f_k},
    &\underset{n \to \infty}{\overset{\text{law}}{\;\longrightarrow\;}}
    \sum_{k =1 }^K \al_k \scprfield{\wick{F_k\left(\frac{\ga_k \Psi^{\ve}}{\sqrt{\th_0}}\right)}}{f_k},
  \end{align*}
  under $\probPhi^{\om}$ for any $f_1,\dots,f_K\colon D \to \bbR^d$  bounded measurable, and any $\al_1,\dots,\al_K \in \bbR$, with $\Sigma$ as in Theorem~\ref{thm:QIP:RG}.
\end{prop}
Let us denote the (properly scaled) variances, which are nothing else then the diagonals of the smoothend Green's kernel, as 
\begin{align*}
  \si^{\ve}_n(x)^2 
  \,\ldef\,
  \meanPhi^{\om} \left[\vp^{\ve}_n(x)^2\right]
  \qquad
  \,\text{and}\,
  \qquad
  \si^{\ve}(x)^2
  \,\ldef\,
  \meanPhi^{\Si} \left[\left(\th_0^{\frac{1}{2}}\Psi^{\ve}(x)\right)^2\right]
\end{align*}
Moreover, for any $\chi \in C(D)$, $F(x) = \sum_{k\geq 1} a_k x^k$, and any bounded measurable functions $f$, we define an operator acting on $C(D)$ as 
\begin{align}\label{def:Teps:F:continuous:operator}
  T^{F,\ga,f}_{\ve}(\chi)
  &\,\ldef\,
  \sum_{k \geq 0} a_k \left(\frac{\ga}{\th_0}\right)^k \int_D \mathbf{H}_k(\chi(x),\si^{\ve}(x)^2) f(x) \, \md x.
\end{align}
Hence, in the setting of Proposition~\ref{prop:convergence:analyticfunction:smooth:approximation}, we introduce 
\begin{align}\label{def:Teps:K:sum:continuous:operator}
  T^{K}_{\ve}(\chi)
  \,\ldef\,
  \sum_{k=1}^K T^{F_k,\ga_k,f_k}_{\ve}(\chi)
\end{align}
and note that by \eqref{eq:def:wick:by:hermite} and the definition of the Wick product in \eqref{eq:def:wick:by:hermite},
\begin{align*}
  \sum_{k =1 }^K \al_k \scprfield{\wick{\th_0 F_k(\ga_k \th_0^{-\frac{1}{2}}\Psi^{\ve})}}{f_k}
  \,=\,
  T^{K}_{\ve}(\th_0^{1/2}\Psi^{\ve}).
\end{align*}
Hence, upon verifing continuity of $T^K_{\ve}$, continuous mapping yields:
\begin{lemma}\label{lem:continuous:mapping:Teps}
  In the setting of Proposition~\ref{prop:convergence:analyticfunction:smooth:approximation}
  \begin{align*}
    T^{K}_{\ve}(\Phi^{\ve}_n)
    &\underset{n \to \infty}{\overset{\text{law}}{\;\longrightarrow\;}}
    T^{K}_{\ve}(\sqrt{\th_0}\Psi^{\ve}).
  \end{align*}
\end{lemma}
\begin{proof}
  Note that it is enough to show continuity of $T^{F,\ga,f}_{\ve}$ for any $\ga \in \bbR$, any $\be$-Fock-entire $F$ and any bounded $f$. 
  
  Hence, let $\ga \in \bbR$, $f$ bounded measurable and $F$ satisfying $F_2(x) \leq M e^{\be x}$. First, we observe that $\partial/\partial x \mathbf{H}_k(x,\si^2) = k \mathbf{H}_{k-1}(x,\si^2)$ for $k\geq1$, $\partial/\partial x \mathbf{H}_0(x,\si^2) =0$, and that, using the explicit representation of the Hermite polynomials \eqref{eq:def:wick:by:hermite} and the triangle inequality in the first step,
  \begin{align}\label{eq:hermite:pointwise:bound:binomial}
    |\mathbf{H}_k(x,\si^2)|
    &\,\leq\,
    \sum_{m=0}^{\lfloor \frac{k}{2}\rfloor} \frac{k!}{m!(k-2m)!} |x|^{k-2m}\si^{2m}
    \,\leq\,
    \sum_{m=0}^{\lfloor \frac{k}{2}\rfloor} \binom{k}{2m}|x|^{k-2m}(\sqrt{k}\si)^{2m} 
    \nonumber\\[.5ex]
    &\,\leq\,
    \sum_{m=0}^{k} \binom{k}{m}|x|^{k-m}(\sqrt{k}\si)^{m} 
    \,\leq\,
    (|x|+(\sqrt{k}\si))^k \leq 2^{k-1}(|x|^k + (\sqrt{k}\si)^k),
  \end{align}
  where the second inequality followed from estimating the fraction in the first sum by $\binom{k}{2m} k^m$. Hence, the mean value theorem implies that for any $x,y \in \bbR$ and $\si >0$ 
  \begin{align*}
    \left|\mathbf{H}_k(x,\si^2) - \mathbf{H}_k(y,\si^2)\right|
    \,\leq\,
    (k-1) 2^{k-2}\left((|x| \vee |y|)^{k-1} + (\sqrt{k-1} \si)^{k-1}\right) \left|x - y\right|.
  \end{align*}
  Moreover, recalling \eqref{eq:rcm:green:tested:1:upperbound:via:bp} and that $g^{\Si}_D \in L^1(D \times D)$, for any fixed $\ve >0$, any $x \in D$ and any $n \geq \const[N]{Nconst:ass:green:bound}(\om,C_d r_D)$, we have 
  \begin{align}\label{eq:bound:sigma_bar_epsilon}
    \left(\si^{\ve}_n(x) \vee \si^{\ve}(x) \right)^2
    &\,\leq\,
    \Norm{\rh^{\ve}}{\infty}^2 \left( \scprreal{1}{G^{\Si}_D 1}{D} \vee \scprreal{1}{G_n 1}{D} \right)
    \nonumber\\[.5ex]
    &\,\leq\,
    \Norm{\rh^{\ve}}{\infty}^2 \left(\scprreal{1}{G^{\Si}_D 1}{D} \vee c \bar{\nu}\right)
    \,\rdef\,
    (\bar{\si}^{\ve})^2
    \,<\,
    \infty.
  \end{align}
  Now, let $\{\chi_n\}_{n \in \bbN} \subset C(D)$ be a sequence such that $\Norm{\chi_n - \chi}{\infty} \to 0$ as $n \to \infty$. Hence, there exists an $N \in \bbN$ such that for any $\Norm{\chi_n \vee \chi}{\infty} \leq \Norm{\chi}{\infty} + 1$ for any $n \geq N$ which, combined with the observations from above implies that 
  \begin{align*}
    &\Norm{T^{F,\ga,f}_{\ve}\left(\chi_n\right) - T^{F,\ga,f}_{\ve}\left(\chi\right)}{\infty}
    \nonumber\\[.5ex]
    &\,\leq\,
    \Norm{f}{\infty} \sum_{k \geq 0} |a_k| \left(\frac{\ga}{\th_0}\right)^k |D| \Norm{\mathbf{H}_k(\chi_n,\si^{\ve}(\cdot)^2) - \mathbf{H}_k(\chi,\si^{\ve}(\cdot)^2)}{\infty}
    \nonumber\\[.5ex]
    &\,\leq\,
    \Norm{\chi_n - \chi}{\infty} \Norm{f}{\infty} \sum_{k \geq 0} |a_k| \left(\frac{\ga}{\th_0}\right)^k |D| k 2^{k-2} \left((\Norm{\chi}{\infty} + 1)^{k-1} + (\sqrt{k}\bar{\si}^{\ve})^{k-1}\right)
    \nonumber\\[.5ex]
    &\,\rdef\,
    \Norm{\chi_n - \chi}{\infty} \cJ(\ve,\Norm{\chi}{\infty}).
  \end{align*}
  Hence, continuity of $T^{F,\ga,f}_{\ve}$ follows once we verified that $\cJ(\ve,\Norm{\chi}{\infty})$ is finite. Ignoring the constant $\Norm{f}{\infty}|D|$, this boils down to showing that
  \begin{align}\label{eq:hermite:continuity:summability:condition}
    \sum_{k \geq 0} |a_k| k \left(\frac{2\ga (\Norm{\chi}{\infty} + 1)}{\th_0}\right)^k 
    \,+\,
    \sum_{k \geq 0} |a_k| k^{1+ \frac{k}{2}} \left(\frac{2\ga \bar{\si}^{\ve}}{\th_0}\right)^k 
    \nonumber\\[.5ex]
    \,\leq\,
    \sum_{k \geq 0} |a_k| k^{\frac{k}{2}+2} \left(\frac{2 \ga\left((\Norm{\chi}{\infty} + 1)\vee \bar{\si}^{\ve}\right)}{\th_0}\right)^k 
  \end{align}
  is finite. However, by positivity of the coefficients in $F_2$, we have $a_k^2 k! \leq \frac{F_2(x)}{x^k} \leq M \frac{e^{\be x}}{x^k}$ for any $x > 0$, whence, choosing $x=k/\be$ yields $|a_k| \leq \sqrt{M} (\sqrt{\be e})^k/ \sqrt{k! k^k}$, which implies \eqref{eq:hermite:continuity:summability:condition}, since 
  \begin{align}\label{eq:exponential:growth:implies:hermite:condition}
    \sum_{k \geq 0} |a_k| k^{\frac{k}{2}+2} |x|^k 
    \,\leq\,
    \sqrt{M} \sum_{k \geq 0} \left(\sqrt{\be e}\right)^k \frac{k^2}{\sqrt{k!}} |x|^k 
    \,<\,
    \infty 
    \qquad
    \text{for any}
    \;
    x \in \bbR.
  \end{align}
  Since we have verified the continuity of $T^{F,\ga,f}_{\ve}$ and hence also of $T^{K}_{\ve}$, the claim follows from Proposition~\ref{prop:weak:convergence:smooth:approximation} and continuous mapping. 
\end{proof}
It remains to show that 
\begin{lemma}\label{lem:operator:approximation:Teps:L1}
  In the setting of Proposition~\ref{prop:convergence:analyticfunction:smooth:approximation} 
  \begin{align*}
    \meanPhi^{\om}\left[
      \left|
        \sum_{k =1 }^K \al_k \scprfield{\wick{F_k(\ga_k \th_0^{-1} \Phi^{\ve}_n)}}{f_k}
        \,-\, 
        T^{K}_{\ve}(\Phi^{\ve}_n)
      \right|
    \right]
    &\,\xrightarrow{n \to \infty}\,
    0.
  \end{align*}
\end{lemma}
\begin{proof}
  By the triangle inequality it is again enough to show the claim only for $T^{F,\ga,f}_{\ve}(\Phi^{\ve}_n)$ for arbitrary $\ga \in \bbR$, $f$ bounded measurable and $F$ satisfying the growth condition. By \eqref{eq:def:wick:by:hermite}, we have
  \begin{align*}
    \scprfield{\wick{F(\ga \th_0^{-1} \Phi^{\ve}_n)}}{f}
    &\,=\,
    \sum_{k \geq 0} a_k \left(\frac{\ga}{\th_0}\right)^k \int_D \wick{\left(\vp^{\ve}_n(x)\right)^k} f(x) \, \md x 
    \nonumber\\[.5ex]
    &\,=\,
    \sum_{k \geq 0} a_k \left(\frac{\ga}{\th_0}\right)^k \int_D \mathbf{H}_k\left(\vp^{\ve}_n(x),\si^{\ve}_n(x)^2\right) f(x) \, \md x.
  \end{align*}
  Hence, by the triangle inequality and the definition of $T^{F,\ga,f}_{\ve}$ in \eqref{def:Teps:F:continuous:operator},
  \begin{align}\label{eq:operator:approximation:L1:step1}
    \meanPhi^{\om}&\left[
      \left|\scprfield{\wick{F(\ga \th_0^{-1} \Phi^{\ve}_n)}}{f} \,-\, T^{F,\ga,f}_{\ve}(\Phi^{\ve}_n)\right|
    \right]
    \,\leq\,
    \nonumber\\[.5ex]
    &\Norm{f}{\infty} \sum_{k \geq 0} |a_k| \left(\frac{\ga}{\th_0}\right)^k \int_D \meanPhi^{\om}\left[\left| \mathbf{H}_k\left(\vp^{\ve}_n(x),\si^{\ve}_n(x)^2\right)  - \mathbf{H}_k\left(\vp^{\ve}_n(x),\th_0 \si^{\ve}(x)^2\right) \right| \right] \md x.
  \end{align}
  Observe that $(\partial/\partial \si^2) \mathbf{H}_k(x,\si^2) = - (k(k-1)/2) \mathbf{H}_{k-2}(x, \si^2)$ for $k \geq 2$ and $(\partial/\partial \si^2) \mathbf{H}_k(x,\si^2) =0$ for $k\in\{0,1\}$. Recalling the bound in \eqref{eq:hermite:pointwise:bound:binomial}, the mean value theorem implies 
  \begin{align*}
    &\left| \mathbf{H}_k\left(\vp^{\ve}_n(x),\si^{\ve}_n(x)^2\right)  \,-\, \mathbf{H}_k\left(\vp^{\ve}_n(x),\si^{\ve}(x)^2\right) \right|
    \nonumber\\[.5ex]
    &\,\leq\,
    \frac{k(k-1)}{2} 2^{k-3} \left(|\vp^{\ve}_n(x)|^{k-2} + \left(\sqrt{k-2}(\si^{\ve}_n(x) \vee \si^{\ve}(x))\right)^{k-2}\right) \left|\si^{\ve}_n(x)^2 - \si^{\ve}(x)^2 
    \right|. 
  \end{align*}
  Finally, noting that $Z \sim \cN(0,\si^2)$, $\bbE[|Z|^k] = \si^{k} 2^{k/2} \Ga((k+1)/2)/\sqrt{\pi}$, that $\vp^{\ve}_n(x) \sim \cN(0, \si^{\ve}_n(x)^2)$, and recalling $\bar{\si}^{\ve}$ from \eqref{eq:bound:sigma_bar_epsilon} yields that 
  \begin{align*}
    \meanPhi^{\om}\left[
      |\vp^{\ve}_n(x)|^{k-2}
    \right]
    \,=\,
    \si^{\ve}_n(x)^{k-2} 2^{\frac{k-2}{2}} \frac{\Ga\left(\frac{k-1}{2}\right)}{\sqrt{\pi}}
    \,\leq\,
    (\bar{\si}^{\ve})^{k-2} 2^{\frac{k-2}{2}} \frac{\Ga\left(\frac{k-1}{2}\right)}{\sqrt{\pi}}
  \end{align*}
  for any $n \geq \const[N]{Nconst:ass:green:bound}(\om,C_d r_D)$. Using the above considerations and linearity of the expectation in the first step gives
  \begin{align*}
    &\meanPhi^{\om}\left[
      \left| \mathbf{H}_k\left(\vp^{\ve}_n(x),\si^{\ve}_n(x)^2\right)  \,-\, \mathbf{H}_k\left(\vp^{\ve}_n(x),\si^{\ve}(x)^2\right) \right|
    \right]
    \nonumber\\[.5ex]
    &\quad\,\leq\,
    \left|\si^{\ve}_n(x)^2 - \si^{\ve}(x)^2 
    \right| \frac{k(k-1)}{2} 2^{k-3}  \left(
      (\bar{\si}^{\ve})^{k-2} 2^{\frac{k-2}{2}} \frac{\Ga\left(\frac{k-1}{2}\right)}{\sqrt{\pi}}
      \,+\,
      \left(\sqrt{k} \bar{\si}^{\ve}\right)^k
    \right)
    \nonumber\\[.5ex]
    &\quad\,\leq\,
    \left|\si^{\ve}_n(x)^2 - \si^{\ve}(x)^2 
    \right| c k^{\frac{k}{2}+2} (4 \bar{\si}^{\ve})^k
  \end{align*}
  for some $c < \infty$ and any $n \geq \const[N]{Nconst:ass:green:bound}(\om,C_d r_D)$, where we estimated $\Ga((k-1)/2) \leq k^{k/2}$ in the last step. Inserting this into \eqref{eq:operator:approximation:L1:step1}, we arrive at 
  \begin{align*}
    \meanPhi^{\om}&\left[
      \left|\scprfield{\wick{F(\ga \th_0^{-1} \Phi^{\ve}_n)}}{f} \,-\, T^{F,\ga,f}_{\ve}(\Phi^{\ve}_n)\right|
    \right]
    \nonumber\\[.5ex]
    &\,\leq\,
    \left(\int_D \left|\si^{\ve}_n(x)^2 - \si^{\ve}(x)^2 
      \right| \, \md x \right) c \Norm{f}{\infty} \sum_{k \geq 0} |a_k|k^{\frac{k}{2}+2} \left(\frac{4 \ga \bar{\si}^{\ve}}{\th_0}\right)^k 
    \nonumber\\[.5ex]
    &\,\rdef\,
    \left(\int_D \left|\si^{\ve}_n(x)^2 - \si^{\ve}(x)^2 
      \right| \, \md x \right) \cJ(\ve),
  \end{align*}
  for any $n \geq \const[N]{Nconst:ass:green:bound}(\om,C_d r_D)$. Lemma~\ref{lem:fdd:convergence:smooth:approximation} implies that $\left|\si^{\ve}_n(x)^2 - \si^{\ve}(x)^2 \right| \to 0$ as $n \to \infty$ for any $\ve>0$ and $x \in D$. Since $|D|<\infty$, for any $n \geq \const[N]{Nconst:ass:green:bound}(\om,C_d r_D)$, an integrable dominant is given by $2 (\bar{\si}^{\ve})^2$ and hence, dominated convergence implies that 
  \begin{align*}
    \int_D \left|\si^{\ve}_n(x)^2 - \si^{\ve}(x)^2 
    \right| \, \md x
    \,\xrightarrow{n \to \infty}\,
    0.
  \end{align*}
  Since \eqref{eq:exponential:growth:implies:hermite:condition} again implies that $\cJ(\ve)< \infty$, the proof is concluded.
\end{proof}
We are now able to give the short
\begin{proof}[Proof of Proposition~\ref{prop:convergence:analyticfunction:smooth:approximation}]
  We abbreviate 
  \begin{align*}
    X^{\ve}_n 
    &\,\ldef\, 
    \sum_{k =1 }^K \al_k \scprfield{\wick{F_k(\ga_k \th_0^{-1} \Phi^{\ve}_n)}}{f_k},
    \quad 
    \tilde{X}^{\ve}_n 
    \,\ldef\, 
    T^{K}_{\ve}(\Phi^{\ve}_n),
    \nonumber\\[.5ex]
    % \text{and}\quad 
    X^{\ve} 
    &\,\ldef\,
    \sum_{k =1 }^K \al_k \scprfield{\wick{F_k(\ga_k \th_0^{-\frac{1}{2}}\Psi^{\ve})}}{f_k} \; \left(\,=\, T^{K}_{\ve}(\th_0^{1/2}\Psi^{\ve})\right).
  \end{align*}
  Moreover, denote by $\cF_Z(t) \ldef E [e^{i t Z}]$ the characteristic function of a random variable $Z$. Then, for any $t \in \bbR$
  \begin{align*}
    \left|\cF_{X^{\ve}_n}(t) - \cF_{X^{\ve}}(t)\right|
    \,\leq\,
    \left|\cF_{X^{\ve}_n}(t) - \cF_{\tilde{X}^{\ve}_n}(t)\right|
    \,+\,
    \left|\cF_{\tilde{X}^{\ve}_n}(t) - \cF_{X^{\ve}}(t)\right|,
  \end{align*}
  where we observe right away that the second term converges by Lemma \ref{lem:continuous:mapping:Teps}. Using that $|e^{ix}-e^{iy}| \leq |x-y|$ for any $x,y \in \bbR$, the convergence of the first term follows from Lemma \ref{lem:operator:approximation:Teps:L1}.
\end{proof}

\subsection{$L^2$-convergence of the smeared fields}\label{subsec:scaling:limit:L2:convergence}
We quickly recall the definition of $g^{\Si,\ve,\de}$ just before Lemma~\ref{lem:continuous:Green's:smeared:converge}, where, for every $\ve,\de >0$ we introduced
\begin{align}
  g^{\Si,\ve,\de}(x,y)
  \,\ldef\,
  \int \int g^{\Si}_D (z,z') \rh^{\ve}_x(z) \rh^{\de}_y(z') \,\md z \md z'
\end{align}
and the associated operator $G^{\Si,\de,\ve}$ as usual. We first show a formula for the second moments of the relevant objects:
\begin{lemma}\label{lem:covariance:formulas:tested:fields}
  For any measurable $f$ and any $F(x) = \sum_k a_k x^k$ such that the right hand side below is finite, 
  \begin{align}
    \meanPhi^{\om}[\scprfield{\wick{F(\ga \Phi_n)}}{f}^2] 
    &\,=\,\label{eq:covariance:formula:dgff:no:smear}
    \scprreal{f}{F_2(\ga^2 G_n) f}{D}
    \\[.5ex]
    \meanPhi^{\om}[\scprfield{\wick{F(\ga  \th_0^{-1} \Phi^{\ve}_n)}}{f}^2] 
    &\,=\,\label{eq:covariance:formula:dgff:both:smear}
    \scprreal{f}{F_2((\ga \th_0^{-1})^2 G^{\ve,\ve}_n) f}{D}
    \\[.5ex]
    \meanPhi^{\om}[\scprfield{\wick{F(\ga \Phi_n)}}{f}\scprfield{\wick{F(\ga \th_0^{-1} \Phi^{\ve}_n)}}{f}]
    &\,=\,\label{eq:covariance:formula:dgff:one:smear}
    \scprreal{f}{F_2(\ga^2 \th_0^{-1} G^{0,\ve}_n) f}{D} 
    \\[.5ex]
    \meanPhi^{\Si}[\scprfield{\wick{F(\ga \th_0^{-\frac{1}{2}}\Psi_{\ve})}}{f}\scprfield{\wick{F(\ga \th_0^{-\frac{1}{2}} \Psi_{\de})}}{f}] 
    &\,=\,\label{eq:covariance:formula:cgff:smeared}
    \scprreal{f}{F_2(\ga^2 \th_0^{-1} G^{\Si,\ve,\de}) f}{D}.
  \end{align}
\end{lemma}
This is more or less a consequence of Fubini and \eqref{eq:wick:covariance:gaussian:xy}, which we adjust to our setting in the following Lemma:
\begin{lemma}\label{lem:wick:formulas}
  Let $\ve >0$ and $n \in \bbN$. For $\chi_1, \chi_2 \in \{\vp^{\om}_{\lfloor n \cdot \rfloor}, \vp^{\ve}_n(\cdot)\}$, we have for any $x,y \in D$ 
  \begin{align}\label{eq:wick:formula:dgff}
    \meanPhi^{\om} \left[\wick{\chi_1(x)^k}\wick{\chi_2(y)^l} \right]
    \,=\,
    \de_{kl} k! \meanPhi^{\om} \left[\chi_1(x)\chi_2(y) \right]^k.
  \end{align}
  Moreover, for any $\ve,\de>0$, we have for any $x,y \in D$ 
  \begin{align}\label{eq:wick:formula:cgff:epsilon:delta}
    \meanPhi^{\Si}\left[\wick{\Psi^{\ve}(x)^k}\wick{\Psi^{\de}(y)^l}\right]
    \,=\,
    \de_{kl} k! \meanPhi^{\Si}\left[\Psi^{\ve}(x)\Psi^{\de}(y)\right]^k
  \end{align}
\end{lemma}
\begin{proof}
  Since $\chi_1, \chi_2$ are jointly Gaussian w.r.t. $\probPhi^{\om}$ while $\Psi^{\ve}$ is Gaussian w.r.t. $\probPhi^{\Si}$, this follows, for instance, from \cite[Lemma 4.2.9.]{berglund2022introduction} or \cite[Theorem 3.9]{janson1997gaussian}.
\end{proof}
\begin{proof}[Proof of Lemma~\ref{lem:covariance:formulas:tested:fields}]
  We will only show \eqref{eq:covariance:formula:dgff:one:smear}, which is perhaps the most non-trivial. The proof of items \eqref{eq:covariance:formula:dgff:no:smear}, \eqref{eq:covariance:formula:dgff:both:smear} and \eqref{eq:covariance:formula:cgff:smeared} follows analogously from Lemma~\ref{lem:wick:formulas}. Firstly, applying \eqref{eq:wick:formula:dgff} with $\chi_1 = \vp^{\om}_{\lfloor n \cdot \rfloor}$ and $\chi_2 = \vp^{\ve}_n(\cdot)$ and then using that $\meanPhi^{\om}\left[\vp^{\om}_{\lfloor nx \rfloor}\vp^{\ve}_n(y) \right] = g^{0,\ve}_n(x,y)$ (recall the Definition in \eqref{eq:def:smeared:green:kernels:rcm}) yields 
  \begin{align*}
    &\meanPhi^{\om}\left[
      \scprfield{\wick{(\Phi_n)^k}}{f}\scprfield{\wick{(\Phi^{\ve}_n)^l}}{f}
    \right]
    \nonumber\\[.5ex]
    &\qquad\,=\,
    \int_{D \times D} \meanPhi^{\om}\left[
      \wick{(\vp^{\om}_{\lfloor nx \rfloor})^k} \wick{(\vp^{\ve}_n(y))^l}
    \right]
    \indicator_{\lfloor nx \rfloor \in \cC_{\infty(\om)}} f(x) f(y)\,\md x \md y
    \nonumber\\[.5ex]
    &\qquad\,=\,
    \de_{kl} k! \int_{D \times D} g^{0,\ve}_n(x,y)^k f(x) f(y) \,\md x \md y
    \,=\,
    \de_{kl} k! \scprreal{f}{(G^{0,\ve}_n)^k f}{D}.
  \end{align*}
  Hence, using Fubini, we get 
  \begin{align*}
    \meanPhi^{\om}&[\scprfield{\wick{F(\ga \Phi_n)}}{f}\scprfield{\wick{F(\ga \th_0^{-1} \Phi^{\ve}_n)}}{f}]
    \nonumber\\[.5ex]
    &\,=\,
    \meanPhi^{\om}\left[
      \sum_{k \geq 0} a_k \ga^k \scprfield{\wick{(\Phi_n)^k}}{f} \sum_{l \geq 0} a_l (\ga \th_0^{-1})^l \scprfield{\wick{(\Phi^{\ve}_n)^l}}{f}
    \right]
    \nonumber\\[.5ex]
    &\,=\,
    \sum_{k,l \geq 0} a_k a_l \ga^{k+l} (\th_0^{-1})^l \meanPhi^{\om}\left[
      \scprfield{\wick{(\Phi_n)^k}}{f}\scprfield{\wick{(\Phi^{\ve}_n)^l}}{f}
    \right]
    \nonumber\\[.5ex]
    &\,=\,
    \sum_{k \geq 0} a_k^2 \ga^{2k} (\th_0^{-1})^k k! \scprreal{f}{(G^{0,\ve}_n)^k f}{D}
    \,=\,
    \scprreal{f}{F_2(\ga^2 \th_0^{-1} G^{0,\ve}_n)f}{D}.
  \end{align*}
\end{proof}
Next, we show that one can remove the mollifier for the inhomogeneous field. This is the key application of Proposition~\ref{prop:tested:green:convex:F}.
\begin{prop}\label{prop:L2:convergence:dgff:smeared}
  Let $d = 2$ and $D \subset \mathbb{R}^{d}$ be a bounded, strongly regular domain. Suppose there exist $\theta \in (0, 1)$ and $p, q \in [1, \infty]$ such that Assumptions~\ref{ass:law},~\ref{ass:cluster} and~\ref{ass:pq} hold. Further, assume that $F$ is a $\be$-Fock-entire function. Then, for any $|\ga| < \sqrt{\frac{\th_0}{\be} \left(\frac{1}{C_{\Si}} \wedge \frac{\th_0}{C_{\mathrm{HK}}}\right)}$ and $\prob_0$-a.e.\ $\om$ 
  \begin{align*}
    \lim_{\ve \downarrow 0} \lim_{n \to \infty} \Norm{\scprfield{\wick{F(\ga \Phi_n)}}{f} - \scprfield{\wick{ \th_0 F\left(\ga \Phi^{\ve}_n/\th_0\right)}}{f}}{L^2(\probPhi^{\om})}
    \,=\,
    0
  \end{align*}
  for any $f \colon D \to \bbR$ bounded measurable. 
\end{prop}
\begin{proof}
  Define $\tilde{F}(x) \ldef F(x) - F(0) = F(x) - a_0 = \sum_{k \geq 1} a_k x^k$, hence $\tilde{F}(0) = 0$. Expanding the square and then applying \eqref{eq:covariance:formula:dgff:no:smear}-\eqref{eq:covariance:formula:dgff:one:smear} from Lemma~\ref{lem:covariance:formulas:tested:fields} yields
  \begin{align}\label{eq:L2:proof:tilde:F:converges}
    &\Norm{\scprfield{\wick{\tilde{F}(\ga \Phi_n)}}{f} - \scprfield{\wick{ \th_0 \tilde{F}(\ga \th_0^{-1}\Phi^{\ve}_n)}}{f}}{L^2(\probPhi^{\om})}^2
    \nonumber\\[.5ex]
    &\quad\,=\,
    \meanPhi^{\om} \left[\scprfield{\wick{\tilde{F}(\ga \Phi_n)}}{f}^2\right]
    \,+\,
    \meanPhi^{\om} \left[\scprfield{\wick{ \th_0 \tilde{F}(\ga \th_0^{-1}\Phi^{\ve}_n)}}{f}^2\right]
    \nonumber\\[.5ex]
    &\quad\qquad\,-\,
    2 \meanPhi^{\om}\left[\scprfield{\wick{\tilde{F}(\ga \Phi_n)}}{f} \scprfield{\wick{ \th_0 \tilde{F}(\ga \th_0^{-1}\Phi^{\ve}_n)}}{f}\right]
    \nonumber\\[.5ex]
    &\quad\,=\,
    \scprreal{f}{\tilde{F}_2(\ga^2 G_n) f}{D}
    \,+\,
    \th_0^2 \scprreal{f}{\tilde{F}_2((\ga \th_0^{-1})^2 G^{\ve,\ve}_n) f}{D}
    \nonumber\\[.5ex]
    &\quad\qquad\,-\,
    2 \th_0  \scprreal{f}{\tilde{F}_2(\ga^2 \th_0^{-1} G^{0,\ve}_n) f}{D}
    \nonumber\\[.5ex]
    &\quad\,\xrightarrow{\ve \downarrow 0, n \to \infty}\,
    \th_0^2 \scprreal{f}{\tilde{F}_2(\ga^2 \th_0^{-1} G^{\Si}) f}{D} 
    \,+\,
    \th_0^2 \scprreal{f}{\tilde{F}_2((\ga \th_0^{-1})^2 \th_0 G^{\Si}) f}{D}
    \nonumber\\[.5ex]
    &\quad\qquad\,-\,
    2 \th_0 \th_0 \scprreal{f}{\tilde{F}_2(\ga^2 \th_0^{-1} G^{\Si}) f}{D}
    \,=\,
    0,
  \end{align}
  where convergence follows from Proposition~\ref{prop:tested:green:convex:F}. Further, we also have that 
  \begin{align}\label{eq:L2:proof:a0:converges}
    &\Norm{\scprfield{a_0 1^n_{\cC_{\infty}(\om)}}{f} - \scprfield{\th_0 a_0 }{f}}{L^2(\probPhi^{\om})}^2
    \nonumber\\[.5ex]
    &\quad\,=\,
    a_0^2 \left(\scprfield{1^n_{\cC_{\infty}(\om)}}{f}^2 + \th_0^2 \scprfield{1}{f}^2 - 2 \th_0 \scprfield{ 1^n_{\cC_{\infty}(\om)}}{f}\scprfield{1}{f}\right)
    \nonumber\\[.5ex]
    &\quad\,\xrightarrow{n \to \infty}\,
    a_0^2 \left(\th_0^2 \scprfield{1}{f}^2 + \th_0^2 \scprfield{1}{f}^2 - 2 \th_0 \th_0 \scprfield{1}{f}\scprfield{1}{f}\right)
    \,=\,0
  \end{align}
  since $\scprfield{1^n_{\cC_{\infty}(\om)}}{f} = \int 1^n_{\cC_{\infty}(\om)}(x) f(x) \md x \to \th_0 \int f(x) \md x$ as $n \to \infty$ by Theorem~\ref{thm:boivin:ergodic:theorem}. 

  Since $F(x) = \tilde{F}(x) + F(0)= \tilde{F}(x) + a_0$, the linearity of the Wick product ensures that 
  \begin{align*}
    \scprfield{\wick{F(\ga \Phi_n)}}{f}
    &\,=\,
    \scprfield{\wick{\tilde{F}(\ga \Phi_n)}}{f} 
    \,+\, 
    \scprfield{a_0 1^n_{\cC_{\infty}(\om)}}{f},
    \qquad
    \text{and}
    \nonumber\\[.5ex]
    \scprfield{\wick{ \th_0 F(\ga \th_0^{-1}\Phi^{\ve}_n)}}{f}
    &\,=\,
    \scprfield{\wick{ \th_0 \tilde{F}(\ga \th_0^{-1}\Phi^{\ve}_n)}}{f}
    \,+\,
    \scprfield{\th_0 a_0}{f},
  \end{align*}
  hence, the claim follows upon applying the triangle inequality and then \eqref{eq:L2:proof:tilde:F:converges} and \eqref{eq:L2:proof:a0:converges}.
\end{proof}
We now construct the limiting object:
\begin{prop}\label{prop:L2:convergence:cgff:smeared}
  Let $d=2$ and $D \subset \mathbb{R}^{d}$ be a bounded, strongly regular domain. Further, assume that $F$ is a $\be$-Fock-entire function. Then, for any $s>0$, any $|\ga| < \sqrt{\th_0/\be C_{\Si}}$ and any $f \in L^2(D)$, there exists a unique (in an $L^2(\probPhi^{\Si})$-sense) random variable $\scprfield{\wick{F(\ga \th_0^{-1/2}\Psi)}}{f} \in L^2(\probPhi^{\Si})$ such that 
  \begin{align}\label{eq:prop:def:continuous:limit:L2}
    \scprfield{\wick{F\left(\frac{\ga \Psi}{\sqrt{\th_0}}\right)}}{f}
    \,\ldef\,
    \lim_{\ve \downarrow 0} \scprreal{\wick{F\left(\frac{\ga \Psi^{\ve}}{\sqrt{\th_0}}\right)}}{f}{D}
    \quad
    \text{in}
    \; L^2(\probPhi^{\Si}).
  \end{align}
  Moreover, $\wick{F(\ga \th_0^{-1/2}\Psi)}$ is $\probPhi^{\Si}$-a.s.\ linear, i.e.\ for all $f,h \in L^2(D)$ and $a,b \in \bbR$,
  \begin{align}\label{eq:prop:linearity:continuous:field}
    \scprfield{\wick{F(\ga \th_0^{-\frac{1}{2}}\Psi)}}{f + h}
    \,=\,
    \scprfield{\wick{F(\ga \th_0^{-\frac{1}{2}}\Psi)}}{f} 
    \,+\,
    \scprfield{\wick{F(\ga \th_0^{-\frac{1}{2}}\Psi)}}{h}
    \qquad
    \text{$\probPhi^{\Si}$- a.s.},
  \end{align}
  has mean $a_0 \scprreal{1}{f}{D} $ and second moment given by
  \begin{align}\label{eq:prop:variance:continuous:field}
    \meanPhi^{\Si} \left[\scprfield{\wick{F(\ga \th_0^{-\frac{1}{2}}\Psi)}}{f}^2\right] 
    \,=\,
    \scprreal{f}{F_2\left(\ga^{2} \th_0^{-1} G^{\Si}_D\right) f}{D}.
  \end{align}
\end{prop}
For the proof, we will need the following Lemma, which is a consequence of Lemma~\ref{lem:continuous:Green's:smeared:converge} and the Hardy-Littlewood-Sobolev inequality (cf.~\cite[Chapter 4.3, page 106]{lieb2001analysis}). A proof can be found in the appendix.
\begin{lemma}\label{lem:continuous:Green's:smeared:converge:tested:sobolev}
  Let $d =2$ and $D \subset \bbR^d$ be bounded. For a convex, continuous, increasing function $H \colon \bbR_+ \to \bbR$ assume that $ H(a) \leq M e^{\be a}$. Then, for any $f \in L^2(D)$ and any $|\ga| < \sqrt{\th_0/ \be C_{\Si}}$ it holds that
  \begin{align}\label{eq:continuous:Green's:smeared:converge:tested:sobolev}
    \lim_{\ve,\de \downarrow 0} \scprreal{f}{H(\ga^{2}\th_0^{-1} G^{\Si,\ve,\de}) f}{D}
    &\,=\,
    \scprreal{f}{H(\ga^{2}\th_0^{-1} G^{\Si}) f}{D}.
  \end{align}
\end{lemma}
\begin{proof}[Proof of Proposition~\ref{prop:L2:convergence:cgff:smeared}]
  We will first show that 
  \begin{align}
    \lim_{\ve,\de \downarrow 0} \meanPhi^{\Si} \left[\left(\scprfield{\wick{F(\ga \th_0^{-1/2}\Psi^{\ve})}}{f} - \scprfield{\wick{F(\ga \th_0^{-1/2}\Psi^{\de})}}{f}\right)^2\right]
    \,=\,
    0,
  \end{align} 
  which implies \eqref{eq:prop:def:continuous:limit:L2} by the completeness of $L^2(\probPhi^{\Si})$. 

  Expanding the square and using \eqref{eq:covariance:formula:cgff:smeared} in Lemma~\ref{lem:covariance:formulas:tested:fields} we get 
  \begin{align*}
    &\meanPhi^{\Si} \left[\left(\scprfield{\wick{F(\ga \th_0^{-\frac{1}{2}}\Psi^{\ve})}}{f} - \scprfield{\wick{F(\ga \th_0^{-\frac{1}{2}}\Psi^{\de})}}{f}\right)^2\right]
    \nonumber\\[.5ex]
    &\quad\,=\,
    \meanPhi^{\Si} \left[\scprfield{\wick{F(\ga \th_0^{-\frac{1}{2}}\Psi^{\ve})}}{f}^2\right]
    \,+\,
    \meanPhi^{\Si} \left[\scprfield{\wick{F(\ga \th_0^{-\frac{1}{2}}\Psi^{\de})}}{f}^2\right]
    \nonumber\\[.5ex]
    &\quad\qquad\,-\,
    2 \meanPhi^{\Si}\left[\scprfield{\wick{F(\ga \th_0^{-\frac{1}{2}}\Psi^{\ve})}}{f}\scprfield{\wick{F(\ga \th_0^{-\frac{1}{2}}\Psi^{\de})}}{f}\right]
    \nonumber\\[.5ex]
    &\quad\,=\,
    \scprreal{f}{F_2(\ga^2 \th_0^{-1} G^{\Si,\ve,\ve}) f}{D}
    \,+\,
    \scprreal{f}{F_2(\ga^2 \th_0^{-1} G^{\Si,\de,\de}) f}{D}
    \nonumber\\[.5ex]
    &\quad\qquad\,-\,
    2 \scprreal{f}{F_2(\ga^2 \th_0^{-1} G^{\Si,\ve,\de}) f}{D}
    \,\xrightarrow{\ve,\de \downarrow 0}\,
    0,
  \end{align*}
  where convergence follows from Lemma~\ref{lem:continuous:Green's:smeared:converge:tested:sobolev}. Moreover, \eqref{eq:prop:variance:continuous:field} follows, since 
  \begin{align*}
    \meanPhi^{\Si} \left[\scprfield{\wick{F(\ga \th_0^{-\frac{1}{2}}\Psi)}}{f}^2\right] 
    &\,=\,
    \lim_{\ve \downarrow 0} \meanPhi^{\Si} \left[\scprfield{\wick{F(\ga \th_0^{-\frac{1}{2}}\Psi^{\ve})}}{f}^2 \right] 
    \nonumber\\[.5ex]
    &\,=\,
    \scprreal{f}{F_2\left(\ga^{2} \th_0^{-1} G^{\Si}_D\right) f}{D},
  \end{align*}
  where the first equality follows from the $L^2(\probPhi^{\Si})$ convergence and the second from \eqref{eq:covariance:formula:cgff:smeared} in Lemma~\ref{lem:covariance:formulas:tested:fields} and Lemma~\ref{lem:continuous:Green's:smeared:converge:tested:sobolev}. 
\end{proof}
\begin{lemma}\label{lem:bounded:approximation:HLS}
  In the Setting of Proposition~\ref{prop:L2:convergence:dgff:smeared}, for any $n \geq \const[N]{Nconst:ass:green:bound}(\om,C_d r_D)$, any $f,h \in L^2(D)$, some constant $c < \infty$, and $\prob_0$-a.e.\ $\om$
  \begin{equation*}
    \left.
      \begin{aligned}
        &\Norm{\scprfield{\wick{F(\ga \Phi_n)}}{f} - \scprfield{\wick{F(\ga \Phi_n)}}{h}}{L^2(\probPhi^{\om})} \\[.75ex]
        &\Norm{\scprfield{\wick{F(\ga \th_0^{-\frac{1}{2}}\Psi)}}{f} - \scprfield{\wick{F(\ga \th_0^{-\frac{1}{2}}\Psi)}}{h}}{L^2(\probPhi^{\Si})}
      \end{aligned}
    \right\}
    \,\leq\,
    c \Norm{f- h}{L^2(D)}.
  \end{equation*}
\end{lemma}
\begin{proof}
  Using linearity and Lemma~\ref{lem:covariance:formulas:tested:fields} yields for any $n \geq \const[N]{Nconst:ass:green:bound}(\om,C_d r_D)$ and $\prob_0$-a.e.\ $\om$ 
  \begin{align*}
    \meanPhi^{\om} &\left[\left(\scprfield{\wick{F(\Phi_n)}}{f} - \scprfield{\wick{F(\ga \Phi_n)}}{h} \right)^2\right]
    \,=\,
    \scprreal{(f- h)}{F_2(\ga^{2} G_n)(f- h)}{D} 
    \nonumber\\[.5ex]
    &\quad\,\leq\, 
    \int_{D \times D} c \frac{(f- h)(x) (f- h)(y)}{|x-y|^{ \ga^{2} \be C_{\mathrm{HK}}}} \, \md x \md y 
    \,\leq\,
    c \Norm{f- h}{L^2(D)}^2,
  \end{align*}
  where the first inequality followed from applying Corollary~\ref{cor:green:pointwise:upperbound:corollary} as in \eqref{eq:green:upper:bound:corollary:applied} and the second inequality follows from the Hardy-Littlewood-Sobolev inequality (cf.~\cite[Chapter 4.3 page 106]{lieb2001analysis}), which we can apply since $\ga^{2} \be C_{\mathrm{HK}}< 2 (=d)$. The same arguments based on linearity in \eqref{eq:prop:linearity:continuous:field}, the variance formula in \eqref{eq:prop:variance:continuous:field} and using again that $\ga^{2}/\th_0 \be C_{\Si} < 2 (=d)$ allows us to apply the Hardy-Littlewood-Sobolev inequality, yield
  \begin{align*}
    \meanPhi^{\Si} \left[\left(\scprfield{\wick{F(\ga \th_0^{-\frac{1}{2}}\Psi)}}{f} - \scprfield{\wick{F(\ga \th_0^{-\frac{1}{2}}\Psi)}}{h}\right)^2\right] 
    \,\leq\,
    c \Norm{f- h}{L^2(D)}^2.
  \end{align*}
\end{proof}

\subsection{Marginal convergence}\label{subsec:scaling:limit:marginal:convergence}
We now show marginal convergence for all square integrable test functions. The Theorem is formulated with the application of the Cramer-Wold device in the proof of Theorem~\ref{thm:mainresult:2} in mind. 
\begin{theorem}\label{thm:marginal:convergence:analytic:field}
  Let $d = 2$, $K \in \bbN$ and $D \subset \mathbb{R}^{d}$ be a bounded, strongly regular domain. Suppose there exist $\theta \in (0, 1)$ and $p, q \in [1, \infty]$ such that Assumptions~\ref{ass:law},~\ref{ass:cluster} and~\ref{ass:pq} hold. Further, assume that $F_k$ are $\be_k$-Fock-entire functions. Then, for any $\al_k \in \bbR$, any $|\ga_k| < \sqrt{\frac{\th_0}{\be_k} \left(\frac{1}{C_{\Si}} \wedge \frac{\th_0}{C_{\mathrm{HK}}}\right)}$ and $\prob_0$-a.e.\ $\om$ 
  \begin{align*}
    \sum_{k =1 }^K \al_k \scprfield{\wick{F_k(\ga_k \Phi_n)}}{f_k}
    \,\underset{n \to \infty}{\overset{\text{law}}{\;\longrightarrow\;}}\,
    \sum_{k =1 }^K \al_k \scprfield{\wick{\th_0 F_k(\ga_k \th_0^{-\frac{1}{2}}\Psi)}}{f_k}
  \end{align*}
  under $\probPhi^{\om}$ for any $f_k \in L^2(D)$, with $\Sigma$ as in Theorem~\ref{thm:QIP:RG}.
\end{theorem}
\begin{proof}
  For any $f_k \in L^2(D)$,there exists a sequence $\{f^{\de}_k\}_{\de >0}$ of bounded functions such that $f^{\de}_k \to f_k$ as $\de \downarrow 0$ in $L^2(D)$. Henceforth, we abbreviate 
  \begin{alignat*}{2}
    X_n &\equiv \sum_{k =1 }^K \al_k \scprfield{\wick{F_k(\ga_k \Phi_n)}}{f_k}, 
    &\quad X^{\de,\ve}_n &\equiv \sum_{k =1 }^K \al_k \scprfield{\wick{\th_0 F_k(\ga_k \th_0^{-1} \Phi^{\ve}_n)}}{f^{\de}_k},
    \\[1ex]
    X^{\de,\ve} &\equiv \sum_{k =1 }^K \al_k \scprfield{\wick{\th_0 F_k(\ga_k \th_0^{-\frac{1}{2}}\Psi^{\ve})}}{f^{\de}_k},
    &\quad X &\equiv \sum_{k =1 }^K \al_k \scprfield{\wick{\th_0 F_k(\ga_k \th_0^{-\frac{1}{2}}\Psi)}}{f_k}. 
  \end{alignat*}
  Further, denoting again by $\cF_Z(t) \ldef E [e^{i t Z}]$ the characteristic funciton of a random variable $Z$, the triangle inequality implies
  \begin{align}\label{eq:characteristic:functions:triangle}
    |\cF_{X_n}(t) - \cF_{X}(t)|
    \leq
    |\cF_{X_n}(t) -\cF_{X^{\de,\ve}_n}(t)|
    + 
    |\cF_{X^{\de,\ve}_n}(t) - \cF_{X^{\de,\ve}}(t)|
    +
    |\cF_{X^{\de,\ve}}(t) - \cF_{X}(t)|
  \end{align}
  for which we will show convergence as first $n \to \infty$, then $\ve \downarrow 0$ and finally $\de \downarrow 0$.
  
  Note right away that, since the $f^{\de}_k$ are bounded, the third term in \eqref{eq:characteristic:functions:triangle} vanishes by Proposition~\ref{prop:convergence:analyticfunction:smooth:approximation} as $n \to \infty$ for any $\ve >0$ and $\de>0$. 

  Using $|e^{ix}-e^{iy}|\leq |x-y|$ combined with Cauchy-Schwarz and then the triangle inequality yields
  \begin{align}\label{eq:L2:convergence:implies:characteristic}
    |\cF_{X_n}(t) \,-\, \cF_{X^{\de,\ve}_n}(t)|
    &\,\leq\,
    |t| \Norm{X_n - X^{\de,\ve}_n}{L^2(\probPhi^{\om})}
    \nonumber\\[.5ex]
    &\,\leq\,
    |t| \sum_{k=1}^K |\al_k| \Norm{\scprfield{\wick{F_k(\ga_k \Phi_n)}}{f_k} - \scprfield{\wick{\th_0 F_k(\ga_k \th_0^{-1} \Phi^{\ve}_n)}}{f^{\de}_k}}{L^2(\probPhi^{\om})}
  \end{align}
  which, adding and substracting zero and using again the triangle inequality, vanishes as first $n \to \infty$, then $\ve \downarrow 0$ and finally $\de \downarrow 0$, since each summand vanishes by Lemma~\ref{lem:bounded:approximation:HLS} and Proposition~\ref{prop:L2:convergence:dgff:smeared}. Analogously, Lemma~\ref{lem:bounded:approximation:HLS} and Proposition~\ref{prop:L2:convergence:cgff:smeared} imply that 
  \begin{align*}
    &|\cF_{X^{\de,\ve}}(t) \,-\, \cF_{X}(t)|
    \nonumber\\[.5ex]
    &\mspace{36mu}\leq\,
    |t| \sum_{k=1}^K |\al_k| \Norm{\scprfield{\wick{\th_0 F_k(\ga_k \th_0^{-\frac{1}{2}}\Psi^{\ve})}}{f^{\de}_k} - \scprfield{\wick{\th_0 F_k(\ga_k \th_0^{-\frac{1}{2}}\Psi)}}{f_k}}{L^2(\probPhi^{\Si})},
  \end{align*}
  vanishes as first $\ve \downarrow 0$ and then $\de \downarrow 0$, hence, convergence in \eqref{eq:characteristic:functions:triangle} follows, which concludes the claim.
\end{proof}

\subsection{Tightness}\label{subsec:scaling:limit:tightness}
We recall \eqref{eq:def:dgff:nonlinear:tested} and that we may interpret $\wick{F\left(\ga \Phi_{n}\right)}$ as a linear functional on $H^{s}_0(D)$ mapping
\begin{align*}
  H^{s}_0(D) \ni f \mapsto \wick{F\left(\ga \Phi_{n}\right)}(f) \equiv \scprfield{\wick{F\left(\ga \Phi_{n}\right)}}{f}.
\end{align*}
\begin{prop}[Tightness]
  \label{prop:DGFF:tightness}
  Let $d =2$ and $D \subset \mathbb{R}^{d}$ be a bounded, strongly regular domain. Further, assume that $F$ is a $\be$-Fock-entire function. Then, for any $|\ga| < 1/\sqrt{\be C_{\mathrm{HK}}}$, the sequence $\left\{\wick{F\left(\ga \Phi_{n}\right)}\right\}_{n \in \bbN}$ is tight in $H^{-s}(D)$ for any $s > \ga^2 \be C_{\mathrm{HK}}$ and $\prob_0$-a.e.\ $\om$. 
\end{prop}
\begin{proof} For the convenience of the reader, we will keep the structure of the proof given in \cite[Proposition 4.2]{andres2025scaling}.

  \smallskip
  \textsc{Step 1.} In a first step, we show that for $\prob_{0}$-a.e.\ $\omega$
  \begin{align}
    \label{eq:uniform:boundedness}
    \sup_{n \geq \const[N]{Nconst:ass:green:bound}(\om,C_d r_D)}
    \meanPhi^{\omega}\Bigl[ \| \wick{F\left(\ga \Phi_{n}\right)} \|_{H^{-s}(D)}^{2}\Bigr]
    \;<\;
    \infty,
  \end{align}
  for any $s > \ga^2 \be C_{\mathrm{HK}}$.
  By \cite[Theorem~6.3.1]{davies1995spectral}, for any bounded domain $D \subset \mathbb{R}^{d}$, the negative Dirichlet Laplacian, $-\Delta$, on $D$ has an empty essential spectrum and a compact resolvent. Furthermore, by \cite[Corollary~4.2.3]{davies1995spectral} there exists a complete orthonormal set of eigenfunctions $(e_{k}^{D})_{k \in \mathbb{N}}$ of $-\Delta$ in $L^{2}(D)$ with corresponding eigenvalues $\lambda_{k}^{D}$ of $-\Delta$ satisfying $0 < \lambda_{1}^{D}$ and $\lambda_{k}^{D} \leq \lambda_{k+1}^{D}$ for any $k \geq 1$. Hence, in view of \eqref{eq:def:sobolev:scalarproduct:Hs}, $\bigl((1-\Delta)^{-s/2}e_{k}^{D}\bigr)_{k \in \mathbb{N}}$ is an orthonormal basis of $H^{s}_0(D)$ and
  \begin{align*}
    \|\wick{F\left(\ga \Phi_{n}\right)}\|_{H^{-s}(D)}^{2}
    \;\ldef\;
    \sum_{k=1}^{\infty}
    % \frac{\Phi_{n}^{D}(e_{n}^{D})^{2}}{(1+\lambda_{n}^{D})^{s}}
    \frac{1}{(1+\lambda_{k}^{D})^{s}}\, \wick{F\left(\ga \Phi_{n}\right)}(e^{D}_k)^{2},
  \end{align*}
  cf.~\cite[Proof of Lemma~5.5]{giacomin2001equilibrium}, \cite[Lemma~4.1]{cipriani2020discrete} and \cite[p.~1872f.]{chiarini2024stochastic} (note that the exponent in the definition of the Sobolev space and its dual given in \cite{chiarini2024stochastic} differs from ours by a factor of $2$). Further, we define for $s >0$ 
  \begin{align}\label{eq:definition:G_s:fractional:kernel}
    G_s(x,y)
    \,\ldef\,
    \sum_{k \geq 0} \frac{1}{(\la^{D}_k)^s} e^{D}_k(x) e^{D}_k(y),
  \end{align}
  and observe that by Lemma~\ref{lem:G_s:fractional:kernel:bounds}, there exists a $c \equiv c(s) < \infty$ such that 
  \begin{align}\label{eq:G_s:fractional:kernel:bounds}
    G_s(x,y)
    \,\leq\,
    c
    \begin{cases}
      |x-y|_2^{2s - 2},
      &\quad 0\,<\,s\,<\,d/2,
      \\[.5ex] 
      \log \left(\frac{1}{|x-y|_2}\right)\,+\, 1,
      &\quad s\,=\,d/2,
      \\[.5ex] 
      1,
      &\quad s\,>\,d/2.
    \end{cases}
  \end{align}
  From \eqref{eq:covariance:formula:dgff:no:smear} in Lemma~\ref{lem:covariance:formulas:tested:fields}, it follows that
  \begin{align*}
    \meanPhi^{\omega}\bigl[\left(\wick{F\left(\ga \Phi_{n}\right)}(e_{k}^{D})\right)^2\bigr]
    &\,=\,
    \scprreal{e_{k}^{D}}{F_2(\ga^2 G_n) e_{k}^{D}}{D}
    \nonumber\\[.5ex]
    &\,=\,
    \int_{D \times D} e^{D}_k (x) e^{D}_k (y) F_2(\ga^2 g^{\om}_{nD}(\lfloor n x \rfloor,\lfloor n y \rfloor)) \, \md x \md y,
  \end{align*}
  and hence, using Fubini, estimating $1+ \la^{D}_k > \la^{D}_k$ and recalling \eqref{eq:definition:G_s:fractional:kernel} in the first step and then applying Corollary~\ref{cor:green:pointwise:upperbound:corollary} combined with the estimate in \eqref{eq:green:upper:bound:corollary:applied} yields 
  \begin{align*}
    \meanPhi^{\omega}\Bigl[\|\wick{F\left(\ga \Phi_{n}\right)}\|_{H^{-s}(D)}^{2}\Bigr]
    &\,<\,
    \int_{D\times D} G_s(x,y) F_2(\ga^2 g^{\om}_{nD}(\lfloor n x \rfloor,\lfloor n y \rfloor))  \, \md x \md y
    \nonumber\\[.5ex]
    &\,\leq\,
    c(\om) \int_{D\times D} G_s(x,y) |x-y|_2^{- \ga^{2} \be 2 C_{\mathrm{HK}}}  \, \md x \md y
  \end{align*}
  for any $n \geq \const[N]{Nconst:ass:green:bound}(\om,C_d r_D)$, where we abbreviate $c(\om) \equiv M e^{\ga^2 \be \const[C]{const:pointwise:schroedinger:bound:corollary:2d} N_{\mathrm{HK}}(\om)} < \infty$, which is finite for $\prob_{0}$-a.e.\ $\om$. 

  Now, first assume that $\ga^2 \be C_{\mathrm{HK}} <s < 1 (=d/2)$. Then, inserting the first bound in \eqref{eq:G_s:fractional:kernel:bounds} yields
  \begin{align*}
    \meanPhi^{\omega}\Bigl[\|\wick{F\left(\ga \Phi_{n}\right)}\|_{H^{-s}(D)}^{2}\Bigr]
    \,\leq\, 
    c(\om) c \int_{D \times D} |x-y|_2^{2s - 2 - \ga^{2} \be 2 C_{\mathrm{HK}}} \, \md x \md y, 
  \end{align*}
  which is finite in $d=2$ if $s-\ga^{2} \be C_{\mathrm{HK}} > 0$, hence if $s > \ga^{2} \be C_{\mathrm{HK}}$. 

  If $s= 1$, inserting the second bound in \eqref{eq:G_s:fractional:kernel:bounds} yields 
  \begin{align*}
    \meanPhi^{\omega}\Bigl[\|\wick{F\left(\ga \Phi_{n}\right)}\|_{H^{-s}(D)}^{2}\Bigr]
    \leq
    c(\om) c \int_{D \times D} \left(\log\left(\frac{1}{|x-y|_2}\right) + 1 \right)|x-y|_2^{- \ga^{2} \be 2 C_{\mathrm{HK}}}  \, \md x \md y,
  \end{align*}
  which is finite for any $\ga^{2} \be C_{\mathrm{HK}}<1$. Similarly, for any $s > 1$, \eqref{eq:G_s:fractional:kernel:bounds} ensures that $G_s(x,y) \leq c$ and hence, the corresponding integral is again finite for any $\ga^{2} \be C_{\mathrm{HK}}<1$ and any $s > 1$. Hence, we have shown \eqref{eq:uniform:boundedness} for any $s > \ga^2 \be C_{\mathrm{HK}}$. 
  \smallskip

  \textsc{Step 2.} 
  In order to prove tightness, fix $s > \ga^2 \be C_{\mathrm{HK}}$. Then, there exists $t > \ga^2 \be C_{\mathrm{HK}}$ such that $s > t > \ga^2 \be C_{\mathrm{HK}}$. By \emph{Step 1}, for any $\varepsilon > 0$ there exists $R_{\varepsilon} \in (0, \infty)$ such that, for $\prob_{0}$-a.e.\ $\omega$,
  \begin{align*}
    \sup_{n \in \mathbb{N}}
    \probPhi^{\omega}\bigl[ \|\wick{F\left(\ga \Phi_{n}\right)}\|_{H^{-t}(D)}^{2} > R_{\varepsilon}\bigr]
    \;\leq\;
    \frac{1}{R_{\varepsilon}}
    \sup_{n \in \mathbb{N}}
    \meanPhi^{\omega}\Bigl[ \|\wick{F\left(\ga \Phi_{n}\right)}\|_{H^{-t}(D)}^{2}\Bigr]
    \;\leq\;
    \varepsilon.
  \end{align*}
  On the other hand, by Rellich's theorem, see \cite[Proposition~4.3.4]{taylor2023partial}, the inclusion operator $H_{0}^{s}(D) \hookrightarrow H_{0}^{t}(D)$ is compact. This implies that, for any $R > 0$, the closed ball $\bigl\{G \in H^{-t}(D) : \| G \|_{H^{-t}(D)} \leq R\bigr\}$ is compact in $H^{-s}(D)$, which yields tightness of $\left\{\wick{F\left(\ga \Phi_{n}\right)}\right\}_{n \in \bbN}$ in $H^{-s}(D)$. 
\end{proof}
\begin{proof}[Proof of Theorem~\ref{thm:mainresult:1}]
  Follows immediately from Proposition~\ref{prop:DGFF:tightness} and applying Theorem~\ref{thm:marginal:convergence:analytic:field} with $K=1$.
\end{proof}
\begin{proof}[Proof of Theorem~\ref{thm:mainresult:2}]
  Note that Theorem~\ref{thm:marginal:convergence:analytic:field} and the Cramer Wold device imply that the vector
  \begin{align*}
    &\left(\scprfield{\wick{F_1(\ga_1 \Phi_n)}}{f_1},\dots, \scprfield{\wick{F_K(\ga_K \Phi_n)}}{f_K} \right)
    \\[.5ex]
    &\qquad 
    \,\underset{n \to \infty}{\overset{\text{law}}{\;\longrightarrow\;}}\,
    \left(\scprfield{\wick{\th_0 F_1(\ga_1 \th_0^{-\frac{1}{2}}\Psi)}}{f_1},\dots,  \scprfield{\wick{\th_0 F_K(\ga_K \th_0^{-\frac{1}{2}}\Psi)}}{f_K}\right)
  \end{align*}
  converges for any $(f_k)_{k=1}^K \in H^{-s_1}(D) \times \dots \times H^{-s_K}(D)$. Moreover, Proposition~\ref{prop:DGFF:tightness} implies the tightness of each component, hence, using a union bound, tightness of the whole vector on the tuple space follows. 
\end{proof}

\appendix
\section{Extension of the Ergodic Theorem}\label{sec:appendix:ergodic:theorem}
Recall the definitions of $\th_0$ and $1^n_{\cC_{\infty}(\om)}(x,y)$. The following is a slight reformulation of \cite[Theorem 5.3]{flegel2019homogenization}, which in turn is already a generalisation of the earlier result in \cite[Theorem 3]{boivin2003spectral}:
\begin{theorem}[Ergodic Theorem]\label{thm:boivin:ergodic:theorem}
  Let $D \subset \bbR^d$ be bounded and suppose that $\prob$ satisfies Assumption~\ref{ass:law}. Then, for any bounded function $h \colon D \to \bbR$,
  \begin{align*}
    \lim_{n \to \infty} \int_{D} h(x) 1^n_{\cC_{\infty}(\om)}(x) \,\md x 
    \,=\,
    \th_0 \int_{D} h(x) \,\md x, 
  \end{align*}
  for $\prob_0$-a.e.\ $\om$. Moreover, the null set does not depend on $h$.
\end{theorem}
\begin{proof}
  Note that we can rewrite the above integral defining the cube average $h_{n}(z) \ldef n^{d} \int_{z + [0,1/n)^d} h(x)\, \md x$ as 
  \begin{align*}
    \int_{D} h(x) \indicator_{\lfloor nx \rfloor \in \cC_{\infty}(\om)} \, \md x
    \;=\;
    \frac{1}{n^d} \sum_{z \in nD \cap \bbZ^d} h_n\left(\frac{z}{n}\right) \indicator_{0 \in \cC_{\infty}(\tau_{z} \om)},
  \end{align*}
  hence, the convergence follows upon an application of \cite[Theorem 5.3]{flegel2019homogenization}, since we assumed that $h$ is bounded and the Lebesgue differentiation Theorem guarantees that $h_n \to h$ almost everywhere. 
\end{proof}
\begin{proof}[Proof of Lemma~\ref{lem:ergodic:lemma:extend:LocHoelder}]
  We first extend $h$ by $0$ on $\partial D$ to define the smooth approximation $h^{\ve}(x,y) \ldef \int \int \rh^{\ve}_x(z) \rh^{\ve}_y(z') h(z,z') \, \md z \md z'$ for every $(x,y) \in \bar{D}\times \bar{D}$ and $\ve>0$ and observe that  
  \begin{align}\label{eq:lem:ergodic:mollifier:delta:epsilon}
    \lim_{\ve \downarrow 0} \Norm{h^{\ve} - h}{L^1(D \times D)}
    \,=\,
    0,
    \;
    \text{and}
    \;
    h^{\ve} \in C(\bar{D} \times \bar{D}).
  \end{align}
  Further, we define the following subalgebra of $C(\bar{D} \times \bar{D})$ as 
  \begin{align*}
    \cA
    \,\ldef\,
    \left\{ h_M(x,y) = \sum_{i =1}^M f_i(x) g_i(y) \,:\, M \in \bbN, \,\text{and}\, f_i,g_i  \in C(\bar{D})\right\},
  \end{align*}
  and one readily verifies that $\cA$ inherits the point seperating property from $ C(\bar{D})$ and also contains a non-zero constant function. By the Stone-Weierstrass Theorem, $\cA$ is dense in $C(\bar{D}\times \bar{D})$ and hence, for any $\ve >0$, there exists a sequence $\{h^{\ve}_M \}_M \subset \cA$ such that $h^{\ve}_M$ converges uniformly on $\bar{D}\times \bar{D}$ to $h^{\ve}$ as $M \to \infty$ for any $\ve >0$. Moreover, $h^{\ve}_M \leq C \Norm{h}{\infty}$ and, by the triangle inequality and \eqref{eq:lem:ergodic:mollifier:delta:epsilon},
  \begin{align}\label{eq:ergodic:lemma:extension:L1:approx}
    \lim_{\ve \downarrow 0} \lim_{M \to \infty} \Norm{h^{\ve}_M - h}{L^1(D \times D)}
    \,=\,
    0.
  \end{align}
  Further, since $h$ and $h^{\ve}_M$ are bounded, there exists a constant $c>0$ depending on $|D|$ such that both $\int h^{\ve}_M (x,y) \md x$ and $\int h(x,y) \md y$ are in the intervall $I_h = [-c\Norm{h}{\infty},c\Norm{h}{\infty}]$ for any $x,y \in D$. By the local Hoelderness of $\cH$, there exist a $\ga \in (0,1]$ and a constant $L \equiv L(I_h)$ such that $|\cH(a)-\cH(b)|\leq L |a-b|^{\ga}$ for any $a,b \in I_h$ and thus, we can approximate 
  \begin{align}\label{eq:ergodic:lemma:extension:approx:L1:h:eps:M}
    \lim_{\ve \downarrow 0} &\lim_{M \to \infty} \int_{D} \left|\cH \left( \int_{D} h^{\ve}_M(x,y)\,
        \md x\, \right)
      \,-\,
      \int_{D} \cH \left( \int_{D} h(x,y)\,
        \md x\, \right)\right| f(y)\md y
    \nonumber\\[.5ex]
    &\,\leq\,
    \lim_{\ve \downarrow 0} \lim_{M \to \infty} 
    \Norm{f}{\infty} L \int_{D} \int_{D} \left|
      h^{\ve}_M(x,y) - h(x,y)
    \right|^{\ga}
    \,\md x \md y
    \nonumber\\[.5ex]
    &\,\leq\,
    \lim_{\ve \downarrow 0} \lim_{M \to \infty} 
    \Norm{f}{\infty} L
    |D \times D|^{1-\ga} \left(\int_{D} \int_{D} \left|
        h^{\ve}_M(x,y) - h(x,y)
      \right|
      \,\md x \md y\right)^{\ga}
    \,=\,
    0,
  \end{align}
  where the last step involved Jensens inequality and convergence follows from \eqref{eq:ergodic:lemma:extension:L1:approx}. Using that $1^n_{\cC_{\infty}}\leq 1$, one repeats the same argument to approximate $h(x,y)1^n_{\cC_{\infty}}(x)$ by $h^{\ve}_M(x,y)1^n_{\cC_{\infty}}(x)$ uniformly in $n$. Hence, it is only left to verify that 
  \begin{align*}
    \lim_{n \to \infty} \int_{D} \cH &\left( \int_{D} h^{\ve}_M(x,y)1^n_{\cC_{\infty}(\om)}(x)\,
      \md x\, \right) f(y) 1^n_{\cC_{\infty}(\om)}(y) \md y
    \nonumber\\[.5ex]
    &\,=\,
    \th_0 \int_{D} \cH \left( \th_0 \int_{D} h^{\ve}_M(x,y)\,
      \md x\, \right) f(y)  \md y,
  \end{align*}
  for any $\ve, M$ fixed, since the result will then follow upon first letting $n\to \infty$ and then proceeding as in \eqref{eq:ergodic:lemma:extension:approx:L1:h:eps:M}. Since $h^{\ve}_M \in \cA$, there exist (bounded) $f_i,g_i \in C(\bar{D})$ such that $h^{\ve}_M(x,y) = \sum_{i=1}^{M} f_i(x) g_i(y)$ and hence, the left hand side above equals
  \begin{align*}
    &\int_D \cH \left(\sum_{i=1}^M  g_i(y) \int_{D} f_i(x)1^n_{\cC_{\infty}(\om)}(x)\,
      \md x\, \right)f(y) 1^n_{\cC_{\infty}(\om)}(y)  \md y
    \nonumber\\[.5ex]
    &\,=\,
    \int_D \cH \left(\sum_{i=1}^M  g_i(y) A^i\,\right) f(y)1^n_{\cC_{\infty}(\om)}(y) \md y
    \nonumber\\[.5ex]
    &\qquad\,+\,
    \int_D \cH \left( \left(\sum_{i=1}^M  g_i(y) A^i_n(\om)\,\right) - \cH \left(\sum_{i=1}^M  g_i(y) A^i\,\right) \right)f(y) 1^n_{\cC_{\infty}(\om)}(y) \md y
  \end{align*}
  where we defined $A^i_n(\om) \ldef \int_{D} f_i(x)1^n_{\cC_{\infty}(\om)}(x)\,\md x$ and $A^i \ldef \th_0 \int_{D} f_i(x)\,\md x$. Now, by Theorem~\ref{thm:boivin:ergodic:theorem}, we know that $\lim_{n \to \infty} A^i_n(\om) = A^i$ for any $i \leq M$ and hence, the second integral above vanishes by an application of the domiated convergence Theorem as $n \to \infty$ for any $\ve,M$ fixed and $\prob_0$ -a.e.\ $\om$. Again by Theorem~\ref{thm:boivin:ergodic:theorem}, the first integral converges to 
  \begin{align*}
    \lim_{n \to \infty} \int_D \cH \left(\sum_{i=1}^M  g_i(y) A^i\,\right) f(y) 1^n_{\cC_{\infty}(\om)}(y) \md y
    \,=\,
    \th_0 \int_D \cH \left(\sum_{i=1}^M  g_i(y) A^i\,\right) f(y)\md y,
  \end{align*}
  which concludes the proof since, recalling the definitions above, $\sum_{i=1}^M  g_i(y) A^i = \sum_{i=1}^M g_i(y) \th_0 \int_{D} f_i(x)\,\md x = \th_0 \int_D h^{\ve}_M(x,y) \,\md x$.

  Note that \eqref{eq:ergodic:lemma:extend:LocHoelder:H:mapsorigin:to:origin} follows from \eqref{eq:ergodic:lemma:extend:LocHoelder} since $\cH(0)=0$ readily implies that $\cH(a 1^n_{\cC_{\infty}(\om)}(x,y)) = \cH(a 1^n_{\cC_{\infty}(\om)}(x)) 1^n_{\cC_{\infty}(\om)}(y)$ for any $a \in \bbR$. 
\end{proof}

\section{Analytical Lemmas}\label{sec:appendix:analytical:results}
\begin{proof}[Proof of Lemma~\ref{lem:continuous:Green's:smeared:converge}]
  We start with $g^{\Si,\ve,\ve}_D$: Note that by the definition of the mollifier $\rh^{\ve}_x$, we have $\rh^{\ve}_x (z) \rh^{\ve}_y (z') \leq \ve^{-2d} \Norm{\rh}{\infty}^2 \indicator_{B_{x,\ve}\times B_{y,\ve}}(z,z')$ and $\int \int \rh^{\ve}_x (z) \rh^{\ve}_y (z') \, \md z \md z'=1$. Hence, since $\ve^{-2d} \leq c (|B_{x,\ve}\times B_{y,\ve}|)^{-1}$,
  \begin{align}\label{eq:proof:lem:green:smeared:eps:eps:LDT:application}
    \left|g^{\Si,\ve,\ve}_D(x,y) - g^{\Si}_D(x,y)\right|
    &\,=\,
    \left|\int \int \rh^{\ve}_x (z) \rh^{\ve}_y (z') \left(g^{\Si}_D(z,z') -  g^{\Si}_D(x,y)\right)\, \md z \md z'\right|
    \nonumber\\[.5ex]
    &\,\leq\,
    \frac{c}{|B_{x,\ve}\times B_{y,\ve}|} \int_{B_{x,\ve}\times B_{y,\ve}} \left|g^{\Si}_D(z,z') -  g^{\Si}_D(x,y)\right| \, \md z \md z'. 
  \end{align}
  Since $|B_{x,\ve}\times B_{y,\ve}| \downarrow \{(x,y)\}$ as $\ve \downarrow 0$ and $g^{\Si}_D \in L^1(D \times D)$, the Lebesque differentiation theorem implies that the right hand side above tends to $0$ as $\ve \downarrow 0$ for Lebesgue almost every $(x,y) \in \bbR^{d} \times \bbR^{d}$. Since $\cH$ is continuous, this also implies that $\cH(g^{\Si,\ve,\ve}_D(x,y))$ converges to $\cH(g^{\Si}_D(x,y))$ for for Lebesgue almost every $(x,y) \in \bbR^{d} \times \bbR^{d}$. It only remains to argue why it is valid to swap limits. For this end, notice that the convexity of $\cH$ combined with an application of Jensen's inequality yields the dominant sequence
  \begin{align}\label{eq:proof:lem:green:smeared:eps:eps:apply:Jensen:application}
    \left|\cH(g^{\Si,\ve,\ve}_D(x,y)) \right|
    \,\leq\,
    \int \int \rh^{\ve}_x(z) \rh^{\ve}_y(z') \left|\cH(g^{\Si}_D(z,z'))\right|\,\md z \md z'
    \,\rdef\,
    \cH^{\ve,\ve}(g^{\Si}_D(x,y))
  \end{align}
  for any $\ve >0$. Using that mollifiers integrate to $1$ and that $
  \cH \circ g^{\Si}_D \in L^1(D \times D)$, one easily verifies that $\cH^{\ve,\ve}(g^{\Si}_D(x,y))$ is integrable for any $\ve>0$. Moreover, again since $\cH \circ g^{\Si}_D \in L^1(D \times D)$, one gets convergence of $\cH^{\ve,\ve}(g^{\Si}_D(x,y))$ to $\left|\cH(g^{\Si}_D (x,y))\right|$ in $L^1(D \times D)$ as $\ve \downarrow 0$. Hence, by the generalised DCT, it follows that $\cH(g^{\Si,\ve,\ve}(x,y))$ converges to $ \cH (g^{\Si}(x,y))$ in $L^1(D\times D)$.
  
  We next show the claim for $g^{\Si,0,\ve}_D$. Fix $x$ and observe that the function $g_x(y) = g^{\Si}_D(x,y)$ is in $L^1(D)$ for Lebesgue a.e.\ $x$, since $\int \int g_x(y) \,\md y \md x < \infty$ and $g_x \geq 0$ implies that $x \mapsto \int g_x(y) \,\md y$ is finite for a.e.\ $x$. Hence, an analogous application of the Lebesgue differentiation theorem as in \eqref{eq:proof:lem:green:smeared:eps:eps:LDT:application} yields that for almost every $x,y \in \bbR^d$
  \begin{align*}
    \left|g^{\Si,0,\de}_D(x,y) - g^{\Si}_D(x,y)\right|
    \,\leq\,
    \frac{c}{|B_{y,\de}|} \int_{B_{y,\de}} \left|g^{\Si,0,\de}_D(x,z') -  g^{\Si}_D(x,y)\right| \, \md z'
    \,\xrightarrow{\de \downarrow 0}\,
    0,
  \end{align*}
  which implies that $\cH(g^{\Si,0,\de}(x,y)) \to \cH(g^{\Si}(x,y))$ for a.e.\ $x,y \in \bbR^d$. Choosing the dominant sequence $\cH^{0,\de}(g^{\Si}_D(x,y))$ which converges to $\cH(g^{\Si}_D(x,y))$ in $L^1(D \times D)$, we get that $\cH(g^{\Si,0,\de}(x,y)) \to \cH(g^{\Si}(x,y))$ in $L^1(D \times D)$ by the generalised DCT.
  
  To conclude, we treat $g^{\Si,\ve,\de}$: We first fix $\de >0$. Since the function $g^{\de}_y(x) \ldef g^{\Si,0,\de}_D(x,y) $ is in $L^1(D)$ for any fixed $y$ and $\de>0$, an application of the Lebesgue differentiation theorem yields that $g^{\Si,\ve,\de}_D(x,y)\to g^{\Si,0,\de}_D(x,y)$ for Lebesgue a.e.\ $x \in \bbR^d$ and any fixed $y \in \bbR^d$, $\de >0$. Hence, again also $\cH(g^{\Si,\ve,\de}_D(x,y)) \to \cH(g^{\Si,0,\de}_D(x,y))$ for Lebesgue a.e.\ $x$. As a dominant sequence, we again pick $\cH^{\ve,\de}(g^{\Si}_D(x,y))$, which converges in $L^1(D \times D)$ to $\cH^{0,\de}(g^{\Si}_D(x,y))$ for any fixed $\de$. Hence, by the generalised DCT, $\cH(g^{\Si,\ve,\de}(x,y))$ converges to $\cH(g^{\Si,0,\de}_D (x,y))$ in $L^1(D \times D)$ as $\ve \downarrow 0$ for any fixed $\de >0$. Now, the previous step implies that $\cH(g^{\Si,0,\de}_D (x,y))$ converges to $\cH(g^{\Si}_D (x,y))$ in $L^1(D \times D)$ as $\de \downarrow 0$, which concludes the proof.
\end{proof}
\begin{proof}[Proof of Lemma \ref{lem:continuous:Green's:smeared:converge:tested:sobolev}]
  Since $f \in L^2(D)$ is finite almost everywhere, Lemma~\ref{lem:continuous:Green's:smeared:converge} and the local integrability of $H(\ga^2 \th_0^{-1} \cdot) \circ g^{\Si}_D$ for any $\ga < \sqrt{\th_0/ \be C_{\Si}}$ (cf.\ \eqref{eq:H_ga:locally:integrable}) imply that 
  \begin{align*}
    \lim_{\ve,\de \downarrow 0} f(x) f(y) H(\ga^{2}\th_0^{-1} g^{\Si,\ve,\de}_D(x,y))
    \,=\,
    f(x) f(y) H(\ga^{2}\th_0^{-1} g^{\Si}_D(x,y)),
  \end{align*}
  consequently, it only remains to find a suitable dominant. Distinguishing cases where $|x-y|_2 \geq 3 \ve$ or else yields for any $\ve,\de>0$ and some $c>0$
  \begin{align*}
    \left.
      \mspace{-6mu}
      \begin{array}{lr}
        &\int \int \rh^{\ve}_x(z) \rh^{\de}_y(z') |z-z'|_2^{-\tilde{\ga} \be C_{\Si}} \,\md z \md z'
        \\[.5ex] 
        &\int \rh^{\ve}_y(z') |x-z'|_2^{-\tilde{\ga} \be C_{\Si}} \,\md z'
      \end{array}
    \right\}
    \,\leq\,
    c  |x-y|_2^{-\tilde{\ga} \be C_{\Si}}
    \qquad
    \text{for a.e. $x,y$}.
  \end{align*}
  Hence, using the convexity of $H$ as in \eqref{eq:proof:lem:green:smeared:eps:eps:apply:Jensen:application} and then inserting the bounds from above and from \eqref{eq:H_ga:locally:integrable} yields the dominant
  \begin{align*}
    \left|f(x)f(y)H(\ga^{2}\th_0^{-1} g^{\Si,\ve,\de}_D(x,y)) \right|
    \,\leq\,
    c \frac{f(x) f(y)}{|x-y|_2^{\ga^{2}\th_0^{-1}  \be 2 C_{\Si}}}.
  \end{align*}
  Now, applying the Hardy-Littlewood-Sobolev inequality as stated in \cite[Chapter 4.3, page 106]{lieb2001analysis} with $\la = \ga^{2}\th_0^{-1}  \be 2 C_{\Si}<2 (=d)$ and $p=r=2d/(2d-\la) \leq 4/(4-2)=2$ yields
  \begin{align*}
    \int_{D \times D}\frac{f(x) f(y)}{|x-y|_2^{\tilde{\ga} \be C_{\Si}}} \, \md x \md y
    \,\leq\,
    c \Norm{f}{p} \Norm{f}{r}
    \,\leq\,
    c \Norm{f}{2}^2
    \,<\,
    \infty,
  \end{align*}
  where we also used that that $L^{p}(D)\subseteq L^{2}(D)$ for $p \leq 2$ since $D$ is bounded. Hence, the claim follows from the dominated convergence theorem. 
\end{proof}
\begin{lemma}\label{lem:G_s:fractional:kernel:bounds}
  For $d\geq 2$, $G_s$ as defined in \eqref{eq:definition:G_s:fractional:kernel}, then there exists $c \equiv c(s) <\infty$  such that 
  \begin{align}
    G_s(x,y)
    \,\leq\,
    c
    \begin{cases}
      |x-y|_2^{2s - 2},
      &\quad 0\,<\,s\,<\,d/2,
      \\[.5ex] 
      \log \left(\frac{1}{|x-y|_2}\right)\,+\, 1,
      &\quad s\,=\,d/2,
      \\[.5ex] 
      1,
      &\quad s\,>\,d/2.
    \end{cases}
  \end{align}
  for any $x \not= y$. 
\end{lemma}
\begin{proof}
  Recall the definition of the Gamma function $\Ga(s) \ldef \int_0^{\infty} u^{s-1} e^{-u} \, \md u$. Since $1/(\la^{D}_k)^s = 1/\Ga(s) \int_0^{\infty} t^{s-1}e^{-\la^{D}_k t} \, \md t$, we get 
  \begin{align*}
    G_s(x,y) 
    =
    \frac{1}{\Ga(s)} \int_{0}^{\infty} t^{s-1} \left(\sum_{k \geq 0} e^{-\la^{D}_k t} e^{D}_k(x) e^{D}_k(y)\right) \, \md t 
    =
    \frac{1}{\Ga(s)} \int_{0}^{\infty} t^{s-1} k^{D}_{t}(x,y) \, \md t,
  \end{align*} 
  where in the last step we used the spectral representation of the killed heat kernel of a standard Brownian motion denoted by $k^{D}_{t}(x,y)$. 

  Now, assume that $s <d/2$. Using that $k^{D}_{t}(x,y) \leq (4 \pi t)^{-d/2} t^{-d/2} e^{-\frac{|x-y|_2^2}{4t}}$ and substituting $u= |x-y|_2^2/t$ yields 
  \begin{align*}
    G_s(x,y)
    &\,\leq\,
    c \int_{0}^{\infty} t^{s-1-d/2} e^{-\frac{|x-y|_2^2}{4t}}
    \,=\,
    c \int_{0}^{\infty} \left(\frac{|x-y|_2^2}{u}\right)^{s-1-d/2} e^{-\frac{u}{4}} |x-y|_2^2 u^{-2} \,\md u
    \nonumber\\[.5ex]
    &\,=\,
    c |x-y|_2^{2s - d} \int_{0}^{\infty} u^{d/2 -s -1} e^{-u/4} \, \md u
    \,\leq\,
    c |x-y|_2^{2s-d},
  \end{align*}
  since for $s < d/2$ the integral 
  \begin{align*}
    \int_{0}^{\infty} u^{d/2 -s -1} e^{-u/4} \, \md u
    \,\leq\,
    \int_{0}^{1} u^{d/2 -s -1}  \, \md u 
    \,+\,
    \int_{1}^{\infty}  c e^{-\tilde{c} u} \, \md u
    \,<\,
    \infty
  \end{align*} 
  is finite. 

  If $s= d/2$, we first assume that $|x-y|\leq 1$ and split the integral using that $k^{D}_{t}(x,y) \leq c t^{d/2} e^{-|x-y|_2^2/4t}$ 
  \begin{align*}
    G_s(x,y)
    \,\leq\,
    \int_{0}^{|x-y|_2^2} c t^{-1} e^{-|x-y|_2^2/4t} \, \md t
    \,+\,
    \int_{|x-y|_2^2}^{1} c t^{-1} \, \md t
    \,+\,
    \int_{1}^{\infty} c k^{D}_{t}(x,y) \, \md t.
  \end{align*}
  Substituting again $u = |x-y|_2^2/t$ and proceeding as before one readily finds that the first integral is bounded by $\int_{1}^{\infty} \frac{e^{-u/4}}{u} \, \md u \,<\,\infty$. The second integral is bounded by $2c \log(1/|x-y|_2^2)$. For the last integral, we use that the spectral representation implies that $k^{D}_{t}(x,y)\leq c e^{-\la^{D}_1 t}$ and hence, since $\la^{D}_1 >0$, the last integral is bounded by a constant as well. If $|x-y|_2^2>1$, one splits the integral into $t \leq 1$ and $t \geq 1$, and observes that both integrals contribute a finite constant. 
  
  Finally, if $s > d/2$, using again that $k^{D}_{t}(x,y) \leq c t^{d/2}$ for $ t \leq 1$ and that $k^{D}_{t}(x,y)\leq c e^{-\la^{D}_1 t}$ for any $t >1$ yields 
  \begin{align*}
    G_s(x,y) 
    \,\leq\,
    c \int_{0}^{1} t^{s-d/2 -1} \, \md t 
    \,+\,
    c \int_{1}^{\infty} e^{-\la^{D}_1 t} \, \md t
    \,\leq\,
    \tilde{c},
  \end{align*}
  where the first integral is finite since $s > d/2$ and the second is finite since $\la^{D}_1 >0$. 
\end{proof}

\subsection*{Acknowledgement.} C.F.P. is grateful to Perla Sousi for helpful discussions throughout the project, especially on a first draft of the bounds on the Green's kernel.

\bibliographystyle{alpha}
\bibliography{literature}

\end{document}